\documentclass[a4paper,11pt]{amsart}
\usepackage[cm]{fullpage}
\usepackage{hyperref}

\usepackage{mymacros}
\usepackage{mathtools}

\mathtoolsset{showonlyrefs=true}

\title{Period integrals of hypersurfaces via tropical geometry}
\author{Yuto Yamamoto}
\address{
Center for Geometry and Physics, Institute for Basic Science (IBS), Pohang 37673, Korea.}
\email{yuto.yamamoto@riken.jp}
\date{}
\begin{document}

\begin{abstract}
Let $\lc Z_t \rc_t$ be a one-parameter family of complex hypersurfaces of dimension $d \geq 1$ in a toric variety.
We compute asymptotics of period integrals for $\lc Z_t \rc_t$ by applying the method of Abouzaid--Ganatra--Iritani--Sheridan, which uses tropical geometry.
As integrands, we consider Poincar\'{e} residues of meromorphic $(d+1)$-forms on the ambient toric variety, which have poles along the hypersurface $Z_t$.
The cycles over which we integrate them are spheres and tori which correspond to tropical $(0, d)$-cycles and $(d, 0)$-cycles on the tropicalization of $\lc Z_t \rc_t$ respectively.
In the case of $d=1$, we explicitly write down the polarized logarithmic Hodge structure of Kato--Usui at the limit as a corollary.
Throughout this article, we impose the assumption that the tropicalization is dual to a unimodular triangulation of the Newton polytope.
\end{abstract}

\maketitle

\section{Introduction}\label{sc:intro}

Abouzaid--Ganatra--Iritani--Sheridan \cite{MR4194298} computed asymptotics of integrations of holomorphic volume forms on toric Calabi--Yau hypersurfaces over Lagrangian sections of SYZ fibrations by using tropical geometry.
They gave a new proof of the gamma conjecture (an asymptotic version of Hosono's conjecture \cite[Conjecture 2.2]{MR2282969}) for ambient line bundles on Batyrev \cite{MR1269718} pairs of mirror Calabi--Yau hypersurfaces.
Their work also gave us a new perspective from Strominger--Yau--Zaslow conjecture \cite{MR1429831} and tropical geometry to the gamma conjecture.
In particular, they observed that the Riemann zeta values (the effect of gamma classes) appearing in the subleading terms of periods are contributions from the discriminants of SYZ fibrations.
The goal of this article is to generalize their work to integrations of Poincar\'{e} residues of meromorphic forms on the ambient toric varieties, which have poles along the hypersurfaces, for hypersurfaces which are not necessarily Calabi--Yau hypersurfaces.
The periods of the residues of forms with a higher-order pole are also necessary to describe the Hodge structure of the hypersurfaces.

\subsection{Main result}\label{sc:main}

Let $K$ be the convergent Puiseux series field over $\bC$, i.e., the field of formal series $\sum_{j \in \bZ}c_j x^{j/n}$ $\lb c_j \in \bC, n \in \bZ_{\geq 1} \rb$ that have only finitely many coefficients with negative index and whose positive part is convergent in a neighborhood of $x=0$. 
It has the standard non-archimedean valuation 
\begin{align}\label{eq:val}
	\mathrm{val} \colon K \longrightarrow \bQ \cup \{ \infty \},\quad k=\sum_{j \in \bZ}c_j x^{j/n} \mapsto \min \lc j/n \in \bQ \relmid c_j \ne 0 \rc.
\end{align}
Let $d$ be a positive integer.
Consider a free $\bZ$-module $N$ of rank $d+1$, and its dual $M:=\Hom (N, \bZ)$.
We write $N_Q:=N \otimes_\bZ Q$, $M_Q:=M \otimes_\bZ Q$ for a $\bZ$-module $Q$.
Let $\Delta \subset M_\bR$ be a lattice polytope of dimension $d+1$ such that $W:=\Int (\Delta) \cap M \neq \emptyset$, where $\Int (\Delta)$ denotes the interior of $\Delta$.
We consider a Laurent polynomial $f=\sum_{m \in A} k_m z^m \in K \ld M \rd$ $(k_m \neq 0, \forall m \in A:= \Delta \cap M)$ over $K$ such that the function 
\begin{align}\label{eq:lambda}
\lambda \colon A \to \bQ, \quad m \mapsto \mathrm{val}(k_m)
\end{align}
extends to a strictly-convex piecewise affine function on a unimodular triangulation $\scrT$ of $\Delta$, i.e., a triangulation consisting only of $(d+1)$-dimensional simplices of the minimal volume $1/\lb d+1\rb!$ and their faces. 
(In this article, we say that a function $h \colon \Delta \to \bR$ is \emph{convex} if it satisfies $h \lb t m_1+(1-t)m_2 \rb \leq t h(m_1)+(1-t)h(m_2)$ for any $t \in \ld 0, 1\rd$ and $m_1, m_2 \in \Delta$.)

Let $t \in \bR_{>0}$ be an element in a small neighborhood of $0$, where all coefficients $k_m$ of the Laurent polynomial $f$ converge.
We consider the Laurent polynomial $f_t \in \bC \ld M \rd$ over $\bC$ obtained by substituting $t$ for the indeterminate $x$ in the coefficients of $f$.
We set
\begin{align}
\mathring{Z}_t:=\lc z \in N_{\bC^\ast} \relmid f_t(z)=0 \rc
\end{align}
and let $Z_t$ be the closure of $\mathring{Z}_t$ in the toric variety $Y_\Sigma \supset N_{\bC^\ast}$ associated with a rational simplicial fan $\Sigma$ in $N_\bR$ that is a refinement of the normal fan of $\Delta$.

Let further $l \geq 1$ be an integer.
We consider the polytope $l \cdot \Delta:=\lc l \cdot m \in M_\bR \relmid m \in \Delta \rc$, and its triangulation $l \cdot \scrT:=\lc l \cdot \tau \relmid \tau \in \scrT \rc$.
We set $V_l:=M \cap \Int \lb l \cdot \Delta \rb$, and take an element $v \in V_l$.
Let $\tau_v \in \scrT$ be the minimal cell such that $v \in l \cdot \tau_v$.
By the unimodularity assumption on $\scrT$, one can uniquely write $v= \sum_{m \in A \cap \tau_v} p_m \cdot m$ with $p_m \in \bZ \cap \lb 0, l \rd$ such that $\sum_{m \in A \cap \tau_v} p_m =l$.
We consider
\begin{align}\label{eq:omega}
\omega_t^{l, v}:=
\lb \bigwedge_{i=0}^{d} \frac{dz_i}{z_i} \rb
\frac{1}{\lb f_t \rb^l}
\prod_{m \in A \cap \tau_v} \lb k_{m, t} \cdot z^m \rb^{p_m}, 
\end{align}
where $\lb z_0, \cdots, z_{d} \rb$ are $\bC^\ast$-coordinates on $N_{\bC^\ast} \cong \lb \bC^\ast \rb^{d+1}$, and $k_{m, t} \in \bC^\ast$ is the number obtained by substituting $t$ to the indeterminate $x$ in $k_m$.
This extends to a meromorphic $(d+1)$-form on $Y_\Sigma$ that has a pole along $Z_t$, and such forms generate $H^{0} \lb Y_\Sigma, \Omega^{d+1} \lb l \cdot Z_t \rb \rb$ (cf.~\pref{sc:forms}).
We write the image of $\omega_t^{l, v}$ by the Poincar\'{e} residue map 
\begin{align}\label{eq:Res}
\Res \colon H^{0} \lb Y_\Sigma, \Omega^{d+1} \lb l \cdot Z_t \rb \rb \to H^d \lb Z_t, \bC \rb
\end{align}
(cf.~e.g.~\cite[Section 8]{MR0260733}) as $\Omega_t^{l, v} \in H^d \lb Z_t, \bC \rb$.
This is the form which we consider as an integrand in this article.
Note that the cohomology group of a hypersurface decomposes into the residual part (the image of the Poincar\'{e} residue map \eqref{eq:Res}) and the ambient part which consists of cohomology classes coming from the ambient toric variety, and therefore the former part is essential when it comes to Hodge structure of hypersurfaces (cf.~\cite[Section 5]{MR2019444}).

The cycles for which we compute the periods in this article are the ones constructed by Mikhalkin \cite{MR2079993}.
We briefly recall the result.
The tropicalization $\trop \lb f \rb$ of $f$ is the piecewise affine function $\trop \lb f \rb \colon N_\bR \to \bR$ defined by
\begin{align}
\trop(f)(n):=\min_{m \in A} \lc \mathrm{val} (k_m) + \la m, n \ra \rc,
\end{align}
and the tropical hypersurface $X\lb \trop \lb f \rb \rb \subset N_\bR$ defined by $\trop \lb f \rb$ is the corner locus of $\trop \lb f \rb$.
The tropical hypersurface $X\lb \trop \lb f \rb \rb$ is a polyhedral complex of dimension $d$, and gives a polyhedral decomposition $\scrP$ of $N_\bR$, which is dual to $\scrT$ (cf.~e.g.~\cite[Proposition 2.1]{MR2079993}).
It is also the limit of the image of $\mathring{Z}_t \subset N_{\bC^\ast}$ by the map $\Log_t \colon N_{\bC^\ast} \to N_\bR$ defined by
\begin{align}\label{eq:Log}
\Log_t \lb z_0, \cdots, z_d \rb:=\lb \log_t |z_0|, \cdots, \log_t |z_d| \rb
\end{align}
as $t \to +0$ \cite{ISSN:1401-5617, MR2079993}.
The relevant part of the claim of \cite[Theorem 3]{MR2079993} for our hypersurface $\mathring{Z}_t \subset N_{\bC^\ast}$ is as follows:

\begin{theorem}{\rm(\cite[Theorem 3]{MR2079993})}\label{th:mik}
There exists a stratified torus fibration $p \colon \mathring{Z}_t \to V \lb \trop \lb f \rb \rb$ satisfying the following:
\begin{itemize}
\item The induced map $p^\ast \colon H^d \lb V \lb \trop \lb f \rb \rb, \bZ \rb \to H^d \lb \mathring{Z}_t, \bZ \rb$ is injective.
\item For each cell $\sigma \in \scrP$ of dimension $d$ contained in $X\lb \trop \lb f \rb \rb$, there exists a point $n \in \Int (\sigma)$ such that the fiber $p^{-1} \lb n \rb$ is a Lagrangian $d$-torus.
\item There exist Lagrangian embeddings $\phi_i \colon S^d \to  \mathring{Z}_t$ 
$\lb 1 \leq i \leq \rank H^d \lb V \lb \trop \lb f \rb \rb, \bZ \rb \rb$ 
such that the cycles $p \lb \phi_i \lb S^d\rb \rb$ form a basis of $H_d \lb V \lb \trop \lb f \rb \rb, \bZ \rb$.
\end{itemize}
\end{theorem}

We refer the reader to \cite[Definition 7]{MR2079993} for the definition of a stratified torus fibration.
The sphere cycles stated in the last item of \pref{th:mik} are constructed in our setup as follows:
For $w \in W$, we rewrite the equation $f=0$ as $-\lb f- k_w z^w \rb / k_w z^w=1$, and set
\begin{align}
f_t^w:=\left. -\lb f- k_w z^w \rb / k_w z^w \right|_{x=t} \in \bC \ld M \rd.
\end{align}
Then $\mathring{Z}_t=\lc z \in N_{\bC^\ast} \relmid f_t^w(z)=1 \rc$.
The hypersurface $Z_t$ is a member of the family of complex hypersurfaces defined by the second projection
\begin{align}\label{eq:ch-family}
\lc \lb z, \lc a_m \rc_{m \in A \setminus \lc w \rc} \rb \in Y_\Sigma \times \lb \bC^\ast \rb^{A \setminus \lc w \rc} \relmid \sum_{m \in A \setminus \lc w \rc} a_m z^{m-w}=1 \rc 
\to \lb \bC^\ast \rb^{A \setminus \lc w \rc}.
\end{align}
For $m \in A$, we set $\lambda_m:=\mathrm{val}(k_m) \in \mathbb{Q}$.
Let $c_m \in \bC^\ast$ denote the coefficient of $x^{\lambda_m}$ in $k_m \in K$, and we write its absolute value $|c_m|$ as $r_m \in \bR_{>0}$.
We define 
\begin{align}
\tilde{f}_t^w&:=\sum_{m \in A \setminus \lc w \rc} \frac{r_m}{r_w} t^{\lambda_m-\lambda_w} z^{m-w}\\
\tilde{Z}_t^w&:=\lc z \in Y_\Sigma \relmid \tilde{f}_t^w(z)=1 \rc.
\end{align}
The complex hypersurface $\tilde{Z}_t^w$ is the fiber of \eqref{eq:ch-family} over the point $\lc a_m=\frac{r_m}{r_w} t^{\lambda_m-\lambda_w} \rc_{m \in A \setminus \lc w  \rc} \in \lb \bC^\ast \rb^{A \setminus \lc w \rc}$, and is also a member of the family \eqref{eq:ch-family}.
The positive real locus $\tilde{Z}_t^w \cap N_{\bR_{>0}}$ is homeomorphic to a $d$-sphere (cf.~\cite[Section 6.7]{MR2079993}).
We choose a branch of the argument $\arg \lb - c_m/c_w \rb$ for every $m \in A \setminus \lc w \rc$, and transport the positive real locus $\tilde{Z}_t^w \cap N_{\bR_{>0}}$ in \eqref{eq:ch-family} by varying the complex coefficients of the polynomial defining the hypersurface from $\tilde{f}_t^w$ to $f_t^w$ so that the argument $\arg \lb r_m/r_w \rb=0$ of every coefficient in $\tilde{f}_t^w$ changes to the argument $\arg \lb -k_{m, t}/k_{w, t} \rb \sim \arg \lb - c_m/c_w \rb$ (as $t \to +0$) of the coefficient in $f_t^w$ continuously.
Let $C_t^w \subset Z_t$ denote the cycle that we obtain by this procedure.
(The cycle $C_t^w \subset Z_t$ depends on the choices of branches of $\arg \lb - c_m/c_w \rb$.)
This is one of the cycles stated in the last item of \pref{th:mik}.
We will construct the cycle $C_t^w$ concretely in \pref{sc:sphere} by the method of \cite{MR4194298}.

Let $\sigma \in \scrP$ be a cell of dimension $d$ contained in $X\lb \trop \lb f \rb \rb$ whose dual edge in $\scrT$ contains an element in $W$ as its vertex.
We take a point $n_0 \in \Int \lb \sigma \rb$.
There is a submanifold that is diffeomorphic to a $d$-torus in the inverse image of a small neighborhood of $n_0$ in $N_\bR$ by the restriction of the map $\Log_t$ to $\mathring{Z}_t \subset N_{\bC^\ast}$.
(See \pref{sc:integral2} for the detail.)
This is one of the torus cycles stated in the second item of \pref{th:mik}.
We write it as $T_t^\sigma \subset Z_t$.
We will compute the period integrals over these cycles $C_t^w, T_t^\sigma$.

For $w \in W$, we set 
\begin{align}\label{eq:Aw}
A_{w}:=\lc m \in A \setminus \lc w \rc \relmid \conv \lb \lc m, w \rc \rb \in \scrT \rc,
\end{align}
where $\conv \lb \bullet \rb$ denotes the convex hull.
Let further $Y_{w}$ be the toric variety associated with the fan $\Sigma_{w}:=\lc \bR_{\geq 0} \cdot \lb \tau-w \rb \relmid \tau \in \scrT, \tau \ni w \rc$, and $D_m^w$ $\lb m \in A_{w} \rb$ be the toric divisor on $Y_{w}$ associated with the $1$-dimensional cone $\bR_{\geq 0} \cdot \lb m-w \rb \in \Sigma_{w}$.
We also set
\begin{align}\label{eq:os}
\omega_{\lambda}^{w}:=\sum_{m \in A_{w}} \lb \lambda_{m}-\lambda_{w} \rb D_{m}^w, 
\quad \sigma^w:=\sum_{m \in A_{w}} D_{m}^w
\end{align}
and
\begin{align}
E_{v, w}:=
\lb \prod_{m \in A_w \cap \tau_v} \prod_{i=0}^{p_m-1} \lb D_m^w +i \rb \rb
\prod_{i=0}^{p_w-1} \lb \sigma^w -i \rb \in H^{\ast} \lb Y_w, \bR \rb,
\end{align}
where $p_m$ are the integers that we used to write $v= \sum_{m \in A \cap \tau_v} p_m \cdot m$, and if $w \nin \tau_v$, we set $p_w:=0$ and the last product in $E_{v, w}$ means the empty product.
We also define
\begin{align}
\widehat{\Gamma}_{w}&:=\frac{\prod_{m \in A_{w}} \Gamma \lb 1+D_m^w \rb}{\Gamma (1+\sigma^w)} \in H^{\ast} \lb Y_w, \bR \rb
\end{align}
using the power series expansion of the gamma function $\Gamma \lb 1+ x\rb$:
\begin{align}\label{eq:gamma-exp}
\Gamma(1+x)=\exp \lb -\gamma x+ \sum_{k=2}^\infty \frac{\lb -1 \rb^k}{k} \zeta(k) x^k \rb,
\end{align}
where $\gamma$ is the Euler constant, and $\zeta(k)$ is the Riemann zeta value.
Note that the restriction of the class $\widehat{\Gamma}_{w} \in H^{\ast} \lb Y_w, \bR \rb$ to an anticanonical hypersurface of the toric variety $Y_{w}$ is the gamma class of the hypersurface.
See \cite[Section 4.1]{MR4194298}.
One can also find a more explicit expression of $\widehat{\Gamma}_{w}$ at (15) in loc.cit.
The following are the main results of this article.

\begin{theorem}\label{th:main1}
We have
\begin{align}
\int_{C_t^{w}}  \Omega_t^{l, v} 
=
\left\{ \begin{array}{ll}
\frac{(-1)^{d+p_w}}{(l-1)!}
\int_{Y_{w}} 
t^{-\omega_{\lambda}^{w}}
\cdot 
\prod_{m \in A_{w}} 
\lb
-\frac{c_m}{c_{w}}
\rb^{-D_m^w}
\cdot 
E_{v, w}
\cdot 
\widehat{\Gamma}_w
+O \lb t^\epsilon \rb & \conv \lb \lc w \rc \cup \tau_v \rb \in \scrT \\
O \lb t^\epsilon \rb & \mathrm{otherwise}
\end{array} 
\right.
\end{align}
as $t \to +0$, for some constant $\epsilon >0$, where
\begin{align}
\lb
-\frac{c_m}{c_{w}}
\rb^{-D_m^w}
:=\exp \lb -D_m^w \log \lb -\frac{c_m}{c_{w}} \rb \rb
=\exp \lb -D_m^w 
\lb \log \left| -\frac{c_m}{c_{w}} \right|+\sqrt{-1} \arg \lb -\frac{c_m}{c_{w}} \rb \rb \rb
\end{align}
and the branch of $\arg \lb -c_m/c_{w} \rb$ is the one we chose when we constructed the cycle $C_t^w$.
\end{theorem}

\begin{theorem}\label{th:main2}
We have
\begin{align}
\int_{T_t^\sigma} \Omega_t^{l, v} 
=
\left\{ \begin{array}{ll}
-\lb 2 \pi \sqrt{-1} \rb^d + O \lb t^\epsilon \rb 
& \tau_v \mathrm{\ is\ an\ element\ in\ } W\ \mathrm{and}\ \trop(f)=\mathrm{val}(k_{\tau_v})+\la \tau_v, \bullet \ra \mathrm{on}\ \sigma \\
O \lb t^\epsilon \rb 
& \mathrm{otherwise}
\end{array} 
\right.
\end{align}
as $t \to +0$, for some constant $\epsilon >0$.
\end{theorem}

Needless to say, the signs of the results of integrations depend on the choices of orientations of the cycles $C_t^w, T_t^\sigma$.
Concerning these choices of orientations, see \pref{rm:orientation}.

One can see from \pref{th:main1} that Riemann zeta values (the effect of the gamma class) appear in the principal part of $\int_{C_t^{w}}  \Omega_t^{l, v}$ in general as in the case of Calabi--Yau manifolds.
We can also see from \pref{th:main1} that the affine volumes of bounded cells in the tropical hypersurface $X\lb \trop \lb f \rb \rb$ appear in the leading terms of the periods $\int_{C_t^{w}}  \Omega_t^{l, v}$.
This is discussed in \pref{sc:lead}.

The main result of \cite{MR4194298} is \pref{th:main1} with assumptions that 
\begin{itemize}
\item the polytope $\Delta$ is reflexive,
\item $l=1, v=w(=0), k_w=-1$, and
\item the coefficients $k_m$ $(m \in A \setminus \lc w \rc)$ can be written as $t^{\lambda_m}$ or $t^{\lambda_m} e^{\sqrt{-1}\theta}$ with $\theta \in \bR$.
\end{itemize}
(Notice that the orientations of cycles used in loc.cit.~are different from that used in this article.)
Our purpose of working over the convergent Puiseux series field over $\bC$ rather than simply setting $k_m=t^{\lambda_m}$ or $k_m=t^{\lambda_m} e^{\sqrt{-1}\theta}$ as in \cite{MR4194298} is to see which information of limits of Hodge structure are encoded by tropical hypersurfaces or not.
Indeed we can see from \pref{th:main1} that the asymptotics of the period depends also on the complex coefficients $c_m$ (not only their arguments but also their absolute values) which are thrown away by tropicalization.
The dependence of the asymptotics of periods on complex coefficients can be seen also from the work of \cite{MR4179831} for the case of toric degenerations.

When the polytope $\Delta$ is reflexive, one can also obtain \pref{th:main1} and \pref{th:main2} (as well as the main results of \cite{MR4194298}) by taking (the derivatives of) the asymptotic expansions of the result of \cite[Theorem 1.1]{MR3112512}.
For the correspondence between the cycles $C_t^w$ and the ambient line bundles on the mirror Calabi--Yau hypersurfaces, we refer the reader to \cite[Remark 1.3]{MR4194298}.

\subsection{Logarithmic Hodge theory}\label{sc:log}

Let $K'$ be the convergent Laurent series field, i.e., the field of Laurent series $\sum_{j \in \bZ}c_j x^j$ that have only finitely many coefficients with negative index and whose positive part is convergent in a neighborhood of $x=0$.
This is a subfield of $K$ of \pref{sc:main}.
We consider a polynomial $f=\sum_{m \in A} k_m z^m$ over $K'$.
If $\rho>0$ is a real number that is smaller than the radius of convergence of every coefficient $k_m$ of $f$, and $D_\rho^\ast:=\lc z \in \bC^\ast \relmid \left| z \right| < \rho \rc$, then one can substitute elements $q \in D_\rho^\ast$ to the indeterminate $x$ in the coefficients of $f$ to obtain a family $\lc Z_q \rc_q$ of complex hypersurfaces in the toric variety $Y_\Sigma$ over the punctured disc $D_\rho^\ast$.
It defines a variation of polarized Hodge structure over $D_\rho^\ast$, which extends to a \emph{logarithmic variation of polarized Hodge structure} (LVPH) of Kato--Usui \cite{MR2465224} over the whole disc $D_\rho:=\lc z \in \bC \relmid \left| z \right| < \rho \rc$.
When $d=1$, we can explicitly write down its restriction to the limit $0 \in D_\rho$ by \pref{th:main1} and \pref{th:main2}.
In the following, we assume that $d=1$ and the polynomial $f=\sum_{m \in A} k_m z^m$ satisfies the same assumptions as in \pref{sc:main}.

Let $t \in \bR_{>0}$ a number such that $t < \rho$.
Since $d=1$, the hypersurface $Z_t$ is a Riemann surface, and the tropical hypersurface $X\lb \trop \lb f \rb \rb$ is a tropical curve.
Let $\beta_w$ $(w \in W)$ be the cycle class in $H_1 \lb Z_t, \bZ \rb$ represented by $C_t^w$.
One can take the cycles $C_t^w$ so that we have $\la \beta_{w_0}, \beta_{w_1} \ra=0$ for any $w_0, w_1 \in W$, where $\la \bullet, \bullet \ra$ denotes the intersection pairing (\pref{lm:intersect}).
We can also take a basis $\lc \alpha_w \rc_{w \in W}$ of the subspace of $H_1 \lb Z_t, \bZ \rb$ generated by $\lc T_t^\sigma \rc_\sigma$ so that 
\begin{align}\label{eq:symp}
\la \alpha_{w_0}, \alpha_{w_1}\ra=0, \quad \la \alpha_{w_0}, \beta_{w_1}\ra=\delta_{w_0, w_1},
\end{align}
where $\delta_{w_0, w_1}$ is the Kronecker delta (i.e, $\lc \alpha_w, \beta_w \relmid w \in W \rc$ is a symplectic basis of $H_1 \lb Z_t, \bZ \rb$).
Let $\alpha_w^\ast, \beta_w^\ast$ $(w \in W)$ denote the dual basis of $H^1 \lb Z_t, \bZ \rb$, and define the nilpotent endomorphism $N \colon H^1 \lb Z_t, \bZ \rb \to H^1 \lb Z_t, \bZ \rb$ by
\begin{align}\label{eq:nilp}
N \lb \beta_w^\ast \rb=0, \quad
N \lb \alpha_w^\ast \rb=
\sum_{w' \in \lc w \rc \sqcup \lb A_w \cap W \rb}
(-1)^{1+\delta_{w, w'}}
\cdot
l(w, w')
\beta_{w'}^\ast,
\end{align}
where $l(w, w') \in \bR_{> 0}$ denotes the affine length of the $1$-cell in $\scrP$ that is dual to $\conv \lb \lc w, w' \rc \rb \in \scrT$ when $w \neq w'$, and the affine length of the boundary of the $2$-cell in $\scrP$ that is dual to $w \in \scrT$ when $w=w'$.
The data $\lc l \lb w, w' \rb \rc_{w, w' \in W}$ is exactly the \emph{tropical periods} of the tropical curve $X\lb \trop \lb f \rb \rb$ introduced in \cite{MR2457739}.
We also set
\begin{align}\label{eq:Pvw}
P \lb v, w \rb
:=
\left\{ \begin{array}{ll}
-\sum_{m \in A_{w}} 
\lb
\int_{Y_{w}} 
D_m^{w}
\cdot \sigma^{w}
\rb
\log 
\lb
-\frac{c_m}{c_{w}}
\rb
& v=w \\
\sum_{m \in A_{w}} 
\lb
\int_{Y_{w}} 
D_m^{w} \cdot D_{v}^{w}
\rb
\log
\lb
-\frac{c_m}{c_{w}}
\rb
& v \in A_{w} \\
0 & \mathrm{otherwise}
\end{array} 
\right.
\end{align}
for $v, w \in W$.

The one parameter family $\lc Z_q \rc_{q \in D_\rho^\ast}$ of complex hypersurfaces (curves) defines the variation of polarized Hodge structure over the punctured disk $D_\rho^\ast$, which extends to the LVPH on the whole disk $D_\rho$.

\begin{corollary}\label{cr:log}
The inverse image of the above LVPH by the inclusion $\lc 0 \rc \hookrightarrow D_\rho$ is isomorphic to the polarized logarithmic Hodge structure $\lb H_\bZ, Q, \scrF \rb$ at the standard log point $\lc 0 \rc$ given as follows:
\begin{itemize}
\item $H_\bZ$ is the locally constant sheaf on $\lc 0 \rc^\mathrm{log} \cong S^1$ whose stalk is isomorphic to
\begin{align}\label{eq:stalk}
H^1 \lb Z_t, \bZ \rb \cong \bigoplus_{w \in W} \bZ \alpha_w^\ast \oplus \bZ \beta_w^\ast
\end{align}
and the monodromy is given by $\exp \lb N \rb=\id+N$,
\item $Q \colon H_\bZ \times H_\bZ \to \bZ$ is the pairing given by the cup product $\la \bullet, \bullet \ra$ of $H^1 \lb Z_t, \bZ \rb$, and
\item $\scrF=\lc \scrF^p \rc_{p=0}^2$ is the decreasing filtration of $\scO_{\lc 0 \rc}^{\mathrm{log}} \otimes_{\bZ} H_\bZ \cong \scO_{\lc 0 \rc}^{\mathrm{log}} \otimes_{\bZ} H^1 \lb Z_t, \bZ \rb$ defined by
\begin{align}\label{eq:filtration0}
\scrF^p:=\scO_{\lc 0 \rc}^{\mathrm{log}} \otimes_{\bZ} F^p
\end{align}
with 
\begin{align}\label{eq:filtration}
F^p:=
\left\{ \begin{array}{ll}
H^1 \lb Z_t, \bZ \rb & p=0\\
\bigoplus_{v \in W} \bC \cdot \lb -2 \pi \sqrt{-1} \alpha_v^\ast + \sum_{w \in W} P(v, w) \beta_{w}^\ast \rb & p=1\\
\lc 0 \rc & p=2.
\end{array} 
\right.
\end{align}
\end{itemize}
\end{corollary}

Here $\lb \lc 0 \rc^\mathrm{log}, \scO_{\lc 0 \rc}^{\mathrm{log}} \rb$ is the Kato--Nakayama space associated with the standard log point $\lc 0 \rc$ \cite{MR1700591}.
We refer the reader to \cite[Section 2]{MR2465224} or \cite[Section 5.1]{MR4484542} for the definition of polarized logarithmic Hodge structure.

\subsection{Related work}\label{sc:relate}

The Hodge structure of hypersurfaces in toric varieties have been studied in many papers (e.g.~\cite{MR0260733, MR1290195, MR1733735, MR2019444}), and classical mirror symmetry for toric complete intersections was originally studied in \cite{MR1408320, MR1653024, MR1621573}.
In the case where $\Delta$ is reflexive (the case of Calabi--Yau hypersurfaces), the image of the Poincar\'{e} residue map \eqref{eq:Res} has the \emph{residual B-model Hodge structure} introduced by Iritani \cite{MR3112512}.
It is known to be isomorphic to the \emph{ambient A-model Hodge structure} of the mirror Calabi--Yau hypersurface via the mirror map (cf.~\cite[Theorem 6.9]{MR3112512}).
The asymptotics of the periods in this case can be written down by using Givental's $I$-function, and one can see that the polarized logarithmic Hodge structure at the limit can be described in terms of the tropical Calabi--Yau hypersurface obtained by tropicalization \cite{MR4484542}.

There is a homology theory for tropical varieties, which was introduced in \cite{MR3330789, MR3961331}.
A cycle representing a class of a tropical homology group is called a \emph{tropical cycle}, and it has a bidegree $(p, q)$ $(p, q \in \bZ_{\geq 0})$.
It is expected that classical cycles and tropical $(p, q)$-cycles correspond in such a way that a classical cycle admits a torus fibration structure over a tropical cycle.
The integers $p, q$ are the dimension of its fiber and the dimension of the tropical cycle respectively.
The cycles $C_t^w$ and $T_t^\sigma$ which we consider in this article correspond to a tropical $(0, d)$-cycle and $(d, 0)$-cycle on the tropical hypersurface $X \lb \trop \lb f \rb \rb$ respectively.

Concerning period integrals over cycles corresponding to tropical cycles of other degree, there is work by Ruddat and Siebert \cite{MR4179831, MR4347312}.
They computed integrals of holomorphic volume forms over cycles corresponding to tropical cycles of dimension $1$ for toric degenerations constructed from wall structures.
Their technique was also used by Wang \cite{Wan20} to compute central charges of the Hori--Vafa mirror of the canonical bundle of a smooth projective toric Fano variety.
(Notice that tropical cycles in these work are on integral affine manifolds with singularities rather than on tropical varieties.
Relations between tropical cycles of these two sorts of tropical spaces are studied in \cite{Yam21}.)

It is also known that tropical (co)homology groups correspond to the grade pieces of the limiting mixed Hodge structure of the corresponding degenerating family, and the residues of the logarithmic extensions of the Gauss--Manin connections correspond to the cup products of the invariants of tropical spaces called radiance obstructions or eigenwaves \cite{MR2669728, MR2681794, MR3330789, MR3961331}.
Radiance obstructions and eigenwaves are closely related to the affine volumes of (bounded cells in) tropical spaces.
The assumption that the triangulation $\scrT$ is unimodular corresponds to the smoothness assumption in tropical geometry.
Notice that in general, there are discrepancies between tropical cohomology groups and ordinary cohomology groups if we do not impose the smoothness assumption (cf.~\cite{MR2681794}).

\subsection{Organization of this article}

In \pref{sc:forms}, we review a description of meromorphic forms on toric varieties, which have poles along hypersurfaces.
In \pref{sc:sphere}, we construct the sphere cycle $C_t^w$.
We will prove \pref{th:main1} and \pref{th:main2} in \pref{sc:integral} and \pref{sc:integral2} respectively.
In \pref{sc:lead}, we will see that the affine volumes of bounded cells in tropical hypersurfaces appear in the leading terms of period integrals.
In \pref{sc:ex}, we give a concrete example illustrating \pref{th:main1}.
Lastly, in \pref{sc:log2}, we discuss polarized logarithmic Hodge structure at the limit in the case of $d=1$.
\pref{cr:log} is proved in this section.

\section{Forms having poles along hypersurfaces}\label{sc:forms}

We recall a description of meromorphic forms on toric varieties, which have poles along hypersurfaces.
In particular, we will see that forms of \eqref{eq:omega} generate $H^{0} \lb Y_\Sigma, \Omega^{d+1} \lb l \cdot Z_t \rb \rb$.
We basically follow \cite{MR1290195}.

Let $\Delta \subset M_\bR$ be a lattice polytope of dimension $d+1$, and $\Sigma$ be a rational simplicial fan in $N_\bR$, which is a refinement of the normal fan of $\Delta$.
We consider the toric variety $Y_\Sigma$ over $\bC$ associated with the fan $\Sigma$.
Recall that the class group $\operatorname{Cl} \lb Y_\Sigma \rb$ of the toric variety $Y_\Sigma$ is isomorphic to the cokernel of the map
\begin{align}
M \to \bigoplus_{\rho \in \Sigma(1)} \bZ D_\rho, \quad m \mapsto \sum_{\rho \in \Sigma(1)} \la m, n_\rho \ra D_\rho,
\end{align}
where $\Sigma(1)$ is the set of $1$-dimensional cones in the fan $\Sigma$, $D_\rho$ is the toric divisor on $Y_\Sigma$ corresponding to $\rho \in \Sigma(1)$, and $n_\rho \in N$ is the primitive generator of $\rho \in \Sigma(1)$.
The polynomial ring $S:=\bC \ld y_\rho \colon \rho \in \Sigma(1) \rd$ together with the natural grading by $\operatorname{Cl}(Y_\Sigma)$, which is defined by
\begin{align}
\deg \lb \prod_{\rho \in \Sigma(1)} y_\rho^{a_\rho} \rb:= \ld \sum_{\rho \in \Sigma(1)} a_\rho D_\rho \rd \in \operatorname{Cl}(Y_\Sigma)
\end{align}
is called the \emph{homogeneous coordinate ring} of the toric variety $Y_\Sigma$ \cite[Section 1]{MR1299003}.
For a class $\mu \in \operatorname{Cl}(Y_\Sigma)$, let $S_\mu \subset S$ denote the corresponding graded piece of $S$.
A homogeneous polynomial $F \in S_\mu$ defines a hypersurface $Z_F$ in the toric variety $Y_\Sigma$ (cf.~\cite[Section 3]{MR1290195}).
We write $1$-dimensional cones in $\Sigma$ as $\Sigma(1)=\lc \rho_0, \cdots, \rho_r \rc$.
For a subset $I=\lc i_0, \cdots, i_{d} \rc \subset \lc 0, \cdots, r \rc$ consisting of $d+1$ elements, we set
\begin{align}
\det \lb e_I \rb:=\det \lb \la e_j, n_{i_k} \ra_{0 \leq j, k \leq d} \rb,
\end{align}
where $e_0, \cdots, e_{d}$ are a basis of the lattice $M$.
We define the $(d+1)$-form $\Omega_0$ by
\begin{align}
\Omega_0:=\sum_{| I | =d+1} \det \lb e_I \rb \widehat{y_I} dy_I,
\end{align}
where 
$\widehat{y_I}:=\prod_{i \nin I} y_{\rho_i}$ and 
$dy_I:=\bigwedge_{i \in I} dy_{\rho_i}$ (\cite[Definition 9.3]{MR1290195})).

\begin{theorem}{\rm(\cite[Theorem 9.7]{{MR1290195}})}\label{th:BC}
One has 
\begin{align}\label{eq:h0}
H^0 \lb Y_\Sigma, \Omega^{d+1} (Z_F) \rb=\lc \frac{A\Omega_0}{F} \relmid A \in S_{\mu-\mu_0} \rc,
\end{align}
where $\mu_0:=\sum_{\rho \in \Sigma(1)} \deg \lb y_\rho \rb \in \operatorname{Cl}(Y_\Sigma)$.
\end{theorem}

Since the fan $\Sigma$ is a refinement of the normal fan of the lattice polytope $\Delta$, one can write
\begin{align}
\Delta=\lc m \in M_\bR \relmid \la m, n_\rho \ra+ k_\rho \geq 0, \forall \rho \in \Sigma(1) \rc,
\end{align}
where $k_\rho$ is the integer defined by
\begin{align}\label{eq:supp}
k_\rho:=-\inf_{m \in \Delta} \la m, n_\rho \ra.
\end{align}
Hence, one also has
\begin{align}\label{eq:lD}
l \cdot \Delta=\lc m \in M_\bR \relmid \la m, n_\rho \ra+ l k_\rho \geq 0, \forall \rho \in \Sigma(1) \rc.
\end{align}

Let $f \in \bC \ld M \rd$ be a Laurent polynomial over $\bC$ whose Newton polytope is $\Delta$.
One can write the polynomial $f^l$ $(l \in \bZ_{>0})$ in the homogeneous coordinates by replacing every monomial $z^{m}$ $\lb m \in M \cap l \cdot \Delta \rb$ in $f^l$ with $\prod_{\rho \in \Sigma(1)} y_\rho^{\la m, n_\rho \ra}$ and multiplying $\prod_{\rho \in \Sigma(1)} y_\rho^{l k_\rho}$ so that we get an element in $S$.
We will write it as $F_l \in S$ in the following.
Similarly, the form $\bigwedge_{i=0}^d dz_i/z_i=\bigwedge_{i=0}^d d \log z_i$ with $z_i:=x^{e_i}$ can be written in the homogeneous coordinates as
\begin{align}\label{eq:form}
\bigwedge_{i=0}^{d} \lb \sum_{\rho \in \Sigma(1)} \la e_i, n_\rho \ra \cdot d \log y_\rho \rb=\sum_{| I | =d+1} \det \lb e_I \rb \frac{dy_I}{\prod_{i \in I} y_{\rho_i}}
=\frac{\Omega_0}{\prod_{\rho \in \Sigma(1)} y_\rho}.
\end{align}
We apply \pref{th:BC} to the homogeneous polynomial $F_l$.
Since $\deg \lb F_l \rb=\deg \lb \prod_{\rho \in \Sigma(1)} y_\rho^{l k_\rho} \rb$, the cohomology group \eqref{eq:h0} for $F_l$ is generated by elements
\begin{align}\label{eq:generator}
\lb \prod_{\rho \in \Sigma(1)} y_\rho^{l k_\rho-1} \rb \lb \prod_{\rho \in \Sigma(1)} y_\rho^{\la m, n_\rho \ra} \rb \frac{\Omega_0}{F_l}
\end{align}
with $m \in M$ such that $\lb \prod_{\rho \in \Sigma(1)} y_\rho^{l k_\rho-1} \rb \lb \prod_{\rho \in \Sigma(1)} y_\rho^{\la m, n_\rho \ra} \rb \in S$.
This condition for $m \in M$ holds if and only if $\la m, n_\rho \ra+ l k_\rho-1 \geq 0$ for all $\rho \in \Sigma(1)$.
One can see from \eqref{eq:lD} that this is equivalent to $m \in M \cap \Int \lb l \cdot \Delta \rb$.
Furthermore, we can see by \eqref{eq:form} that \eqref{eq:generator} is written in the affine coordinates as 
\begin{align}\label{eq:generator2}
\frac{z^m}{f^l} \bigwedge_{i=0}^d \frac{dz_i}{z_i}.
\end{align}
We can conclude that \eqref{eq:generator2} with $m \in M \cap \Int \lb l \cdot \Delta \rb$ define elements in $H^0 \lb Y_\Sigma, \Omega^{d+1} \lb Z_{F_l} \rb \rb$ and generate $H^0 \lb Y_\Sigma, \Omega^{d+1} \lb Z_{F_l} \rb \rb$.
Also in the setup of \pref{sc:intro}, we can see that \eqref{eq:omega} defines an element in $H^{0} \lb Y_\Sigma, \Omega^{d+1} \lb l \cdot Z_t \rb \rb$ and such forms generate $H^{0} \lb Y_\Sigma, \Omega^{d+1} \lb l \cdot Z_t \rb \rb$ since we have 
\begin{align}
\prod_{m \in A \cap \tau_v} \lb k_{m, t} \cdot z^m \rb^{p_m}=z^v \prod_{m \in A \cap \tau_v} \lb k_{m, t} \rb^{p_m}
\end{align}
and $v \in V_l:=M \cap \Int \lb l \cdot \Delta \rb$.

\begin{remark}
Forms of the form \eqref{eq:generator2} were originally considered in \cite{MR1203231} in the case of hypersurfaces in algebraic tori $N_{\bC^\ast}$.
It is known that such forms with $m \in M \cap \Int \lb l \cdot \Delta \rb$ generate the lowest weight component of the middle cohomology group of the complement of the hypersurface (\cite[Theorem 8.2]{MR1203231}).
\end{remark}

\section{Construction of sphere cycles}\label{sc:sphere}

In this section, we construct the sphere cycle $C_t^w$ $\lb w \in W \rb$ in \pref{sc:main} by using the technique in \cite[Section 5.2]{MR4194298}.
There are also similar or related constructions in \cite{MR2240909, MR2529936, MR2871160, MR3228454, MR3948684}.
Following \cite{MR4194298}, we construct the cycle $C_t^w$ by sliding the cycle that arises as the positive real locus of $\tilde{Z}_t^w$ to the purely imaginary direction and modifying it to an actual cycle in $Z_t$ by a small perturbation $\delta_t$.
The slide is given by \eqref{eq:Phi}, and the perturbation $\delta_t$ is made in \pref{pr:delta}.
The period of the Poincar\'{e} residue $\Omega_t^{l, v}$ of $\omega_t^{l, v}$ is given by
\begin{align}
\int_{C} \Omega_t^{l, v} 
=\frac{1}{2 \pi \sqrt{-1}} \int_{T(C)} \omega_t^{l, v}
\end{align}
for any $d$-cycle $C \subset Z_t$, where $T(C)$ denotes a tube over the cycle $C$ (the boundary of a small tubular neighborhood of $C$, which sits in $Y_\Sigma \setminus Z_t$) (cf.~e.g.~\cite[Section 8]{MR0260733}).
In order to compute period integrals over $C_t^w$, we will also construct a tube $T_t^w$ over $C_t^w$ in addition to $C_t^w$ itself in the same way as we construct $C_t^w$.
We work under the same assumptions and use the same notation as in \pref{sc:main}.

For $m \in A$, we set
\begin{align}
\mu_m \colon N_\bC \to \bC, \quad n \mapsto \lambda_m + \la m, n \ra,
\end{align}
where $\lambda_m:=\mathrm{val} (k_m)$.
Recall that the polyhedral decomposition $\scrP$ of $N_\bR$ induced by the tropical hypersurface $X\lb \trop \lb f \rb \rb \subset N_\bR$ is dual to the triangulation $\scrT$ of $\Delta$ induced by \eqref{eq:lambda} (cf.~e.g.~\cite[Proposition 2.1]{MR2079993}).
The correspondence is given by
\begin{align}\label{eq:dual}
\scrP \to \scrT, \quad \sigma \mapsto \conv \lb \lc m \in A \relmid \trop (f)(n)=\mu_m(n), \forall n \in \sigma \rc \rb.
\end{align}
We fix $w \in W$ for which we construct the cycle $C_t^w$, and define
\begin{align}\label{eq:nabla}
\nabla^w:=\lc n \in N_\bR \relmid \mu_m(n) \geq \mu_w(n), \forall m \in A \rc.
\end{align}
This is the element in $\scrP$ dual to $\lc w \rc \in \scrT$.
The normal fan of $\nabla^w$ is $\Sigma_w$ defined just after \eqref{eq:Aw}.
Due to the assumption that $\scrT$ is unimodular, the normal fan $\Sigma_w$ is also unimodular.
For a real constant $\kappa >0$, we set
\begin{align}\label{eq:nkw}
N_\kappa^w:=\lc n \in N_\bR \relmid \mu_w (n)-\kappa \leq \min_{m \in A \setminus \lc w \rc} \mu_m(n) \leq \mu_w(n)+\kappa \rc.
\end{align}
This is a neighborhood of $\partial \nabla^w$.
Notice that this set $N_\kappa^w$ is different from 
\begin{align}\label{eq:nk0}
N_\kappa \lb \Delta_\lambda \rb:=\lc n \in N_\bR \relmid \mu_w (n)-\kappa \leq \min_{m \in A \setminus \lc w \rc} \mu_m(n) \rc
\end{align}
which was considered in \cite[Section 5.2]{MR4194298} (for the case where $w=0$ and $\lambda_w=0$).
The reason why we consider \eqref{eq:nkw} rather than \eqref{eq:nk0} will  be explained later in \pref{rm:kappa}.

For an element $n \in N_\bR$, we also set
\begin{align}\label{eq:K}
K_\kappa^n&:= \lc k \in A \setminus \lc w \rc \relmid \mu_k(n) \leq \mu_w(n)+\kappa \rc \\ \label{eq:L}
L_\kappa^{n}&:= \lc k \in A \setminus \lc w \rc \relmid \mu_k(n) \leq \min_{m \in A \setminus \lc w \rc} \mu_m (n) +\kappa \rc.
\end{align}

\begin{example}\label{eg:curve}
Let $d=1$, and fix a basis $\lc e_1, e_2 \rc$ of the lattice $M \cong \bZ^2$.
Consider the polynomial
\begin{align}\label{eq:ex-poly}
f:=-x^{\lambda_0}z^0+\sum_{m \in \lb \Delta \cap M \rb \setminus \lc 0 \rc} x^{\lambda_m} z^m
\end{align}
with
\begin{align}
\Delta:=\conv \lb \lc 2e_1, -e_1, e_2, -e_2 \rc \rb \subset M_\bR, \quad
\lambda_m:=
\left\{ \begin{array}{ll}
3 & m=2e_1 \\
0 & m=0 \\
1 & m \in \lb \Delta \cap M \rb \setminus \lc 0, 2e_1 \rc.
\end{array} 
\right. 
\end{align}
The triangulation $\scrT$ of the Newton polytope $\Delta$ induced by \eqref{eq:lambda} and the tropicalization $V(\trop (f)) \subset N_\bR$ in this case are shown in \pref{fg:t-curve}.

\begin{figure}[htbp]
\begin{center}
\includegraphics[scale=0.4]{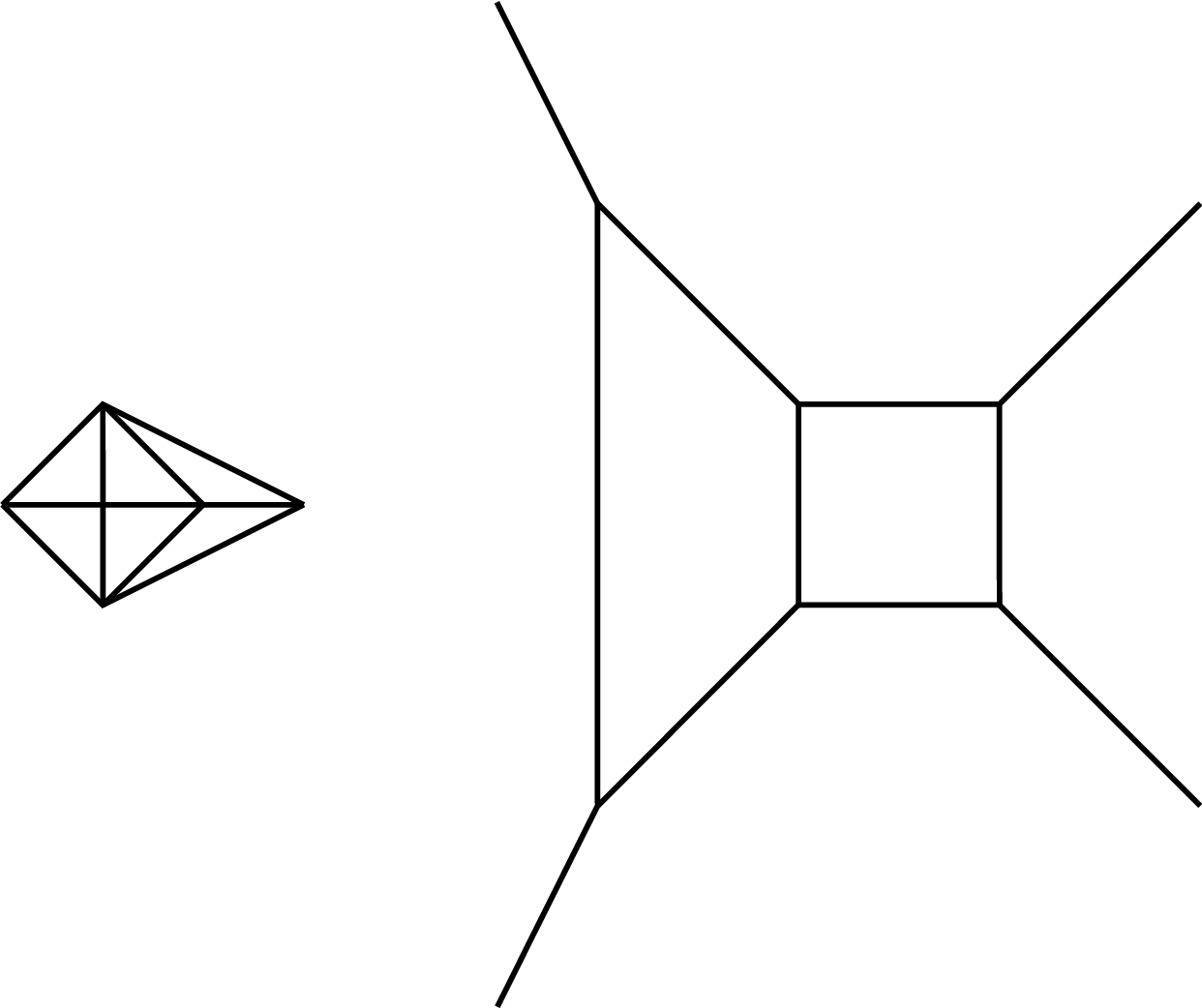}
\end{center}
\caption{The triangulation $\scrT$ of $\Delta$ and the tropicalization $V(\trop (f))$}
\label{fg:t-curve}
\end{figure}

The gray region on the left of \pref{fg:nbd} shows the subset $N_\kappa^w \subset N_\bR$ with $w=0$.
The green lines are parts of the tropicalization $V(\trop (f))$.
The figure in the middle (resp. on the right) in \pref{fg:nbd} is the partition of $N_\kappa^w$ such that the index set $K_\kappa^n$ (resp. $L_\kappa^n$) is constant on each piece and differs on different pieces.
For instance, $K_\kappa^n=\lc e_1, e_2 \rc$ (resp. $K_\kappa^n=\lc e_1 \rc$) on the red (resp. blue) region of the middle figure, and $L_\kappa^n=\lc e_1, e_2 \rc$ (resp. $L_\kappa^n=\lc e_1 \rc$) on the red (resp. blue) region of the right figure.

\begin{figure}[htbp]
\begin{center}
\includegraphics[scale=0.7]{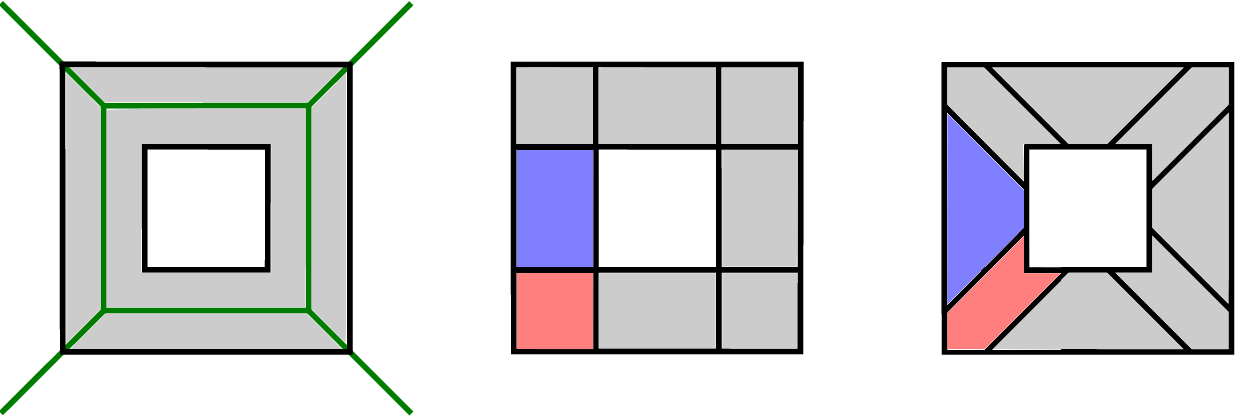}
\end{center}
\caption{The set $N_\kappa^w$ with $w=0$ and its partitions}
\label{fg:nbd}
\end{figure}

\end{example}

\begin{lemma}\label{lm:kappa}
If the constant $\kappa >0$ is sufficiently small, the following hold:
\begin{enumerate}
\item For any $n \in N_{\kappa}^w$, the convex hull of $\lc w \rc \cup K_\kappa^n$ is in the triangulation $\scrT$ of $\Delta$. 
\item For any $n \in N_{2\kappa}^w$, the convex hull of $\lc w \rc \cup L_\kappa^n$ is also in $\scrT$.
\item For any monomial $k_m z^m$ $(m \in A)$ of the polynomial $f$, the common denominator $n$ appearing in the coefficient $k_m= \sum_{j \in \bZ}c_j x^{j/n}$ satisfies $1/n >2 \kappa$.
\end{enumerate}
\end{lemma}
\begin{proof}
The last condition (3) obviously holds since $A$ is a finite set.
We discuss (1) and (2).
Take a small constant $k_0>0$ such that for any $\lc c_m \in \ld 0, k_0 \rd \rc_{m \in A}$, even if we replace each $\lambda_m$ with $\lambda_m-c_m$, the induced triangulation $\scrT$ of $\Delta$ does not change.
For an element $n \in N_\bR$, we set
\begin{align}
J_\kappa^n:=\lc k \in A \relmid \mu_k (n) \leq \trop \lb f \rb(n)+\kappa \rc.
\end{align}
When $\kappa \leq k_0$, we have $\conv \lb J_\kappa^n \rb \in \scrT$ for any $n \in N_\bR$.

Suppose $\kappa < k_0/3$.
For any $n \in N_{\kappa}^w$ and $k \in K_\kappa^n$, we have
\begin{align}
\trop \lb f \rb(n)=\min \lc \mu_w(n), \min_{m \in A \setminus \lc w \rc} \mu_m (n) \rc 
\geq \mu_w(n)-\kappa 
\geq \mu_k (n)-2\kappa.
\end{align}
From this, we can see that $\lc w \rc \cup K_\kappa^n$ is a subset of $J_{2\kappa}^n$.
Since $2\kappa<k_0$ and $\scrT$ is a triangulation, we have 
\begin{align}
\conv \lb \lc w \rc \cup K_\kappa^n \rb \prec
\conv \lb J_{2\kappa}^n \rb \in \scrT.
\end{align}
Similarly, for any $n \in N_{2\kappa}^w$ and $k \in L_\kappa^n$, we have
\begin{align}
\trop \lb f \rb(n)=\min \lc \mu_w(n), \min_{m \in A \setminus \lc w \rc} \mu_m (n) \rc 
\geq \min_{m \in A \setminus \lc w \rc} \mu_m (n)-2\kappa 
\geq \mu_k (n)-3\kappa.
\end{align}
We also have $\trop \lb f \rb(n) \geq \mu_w(n)-2\kappa$.
From these, we can see that $\lc w \rc \cup L_\kappa^n$ is a subset of $J_{3\kappa}^n$.
Again, since $3\kappa<k_0$ and $\scrT$ is a triangulation, we have 
\begin{align}
\conv \lb \lc w \rc \cup L_\kappa^n \rb \prec
\conv \lb J_{3\kappa}^n \rb \in \scrT.
\end{align}
We obtained the claim.
\end{proof}

\begin{remark}\label{rm:kappa}
In \cite[Section 5.2]{MR4194298}, they also take a constant $\kappa>0$ that plays the same role.
The above (1) and (2) in \pref{lm:kappa} correspond to the conditions (i) and (ii) in \cite[Section 5.2]{MR4194298} respectively.
Concerning (2), they take a constant $\kappa>0$ so that $L_\kappa^n$ is linearly independent for any $n \in N_\bR$.
To be precise, this is not possible in general also in the setup of \cite{MR4194298}.
For instance, consider the case where $d=1$, $w=0$, and
\begin{align}
\Delta:=\conv \lb \lc \pm e_1, \pm e_2 \rc \rb \subset M_\bR \cong \bR^2, \quad
\lambda(m)
=
\left\{ \begin{array}{ll}
0 & m=0 \\
1 & m= \pm e_1 \\
2 & m= \pm e_2,
\end{array} 
\right.
\end{align}
where $\lc e_1, e_2 \rc$ is a basis of $M \cong \bZ^2$.
The set $L_\kappa^n$ with $n \in \lc r e_2^\ast \in N_\bR \cong \bR^2 \relmid -1 \leq r \leq 1 \rc$ contains $\pm e_1$ for any $\kappa \geq 0$, and is not linearly independent.
This affects \cite[Lemma 5.6]{MR4194298} which corresponds to \pref{lm:grad} of this article, since we use the condition (ii) for proving the lemma.
In this article, we impose the condition (2) only for elements in $N_{2\kappa}^w$ so that we are able to take a constant $\kappa>0$ satisfying the conditions.
This is the reason why we consider \eqref{eq:nkw}.
Since we impose the condition (2) only for elements in $N_{2\kappa}^w$, the claim of \cite[Lemma 5.6]{MR4194298} (or \pref{lm:grad} of this article) gets weaker (cf.~\pref{rm:kappa2}).
However, this is actually sufficient for our purpose of constructing a perturbation $\delta_t$ in \pref{pr:delta} as we will see in the proof of \pref{pr:delta}.
\end{remark}

We fix $\kappa>0$ satisfying the conditions of \pref{lm:kappa}.
For an element $n \in N_{\kappa}^w$, we have $K_\kappa^n \neq \emptyset$ since an element $k \in A \setminus \lc w \rc$ such that $\mu_k(n)=\min_{m \in A \setminus \lc w \rc} \mu_m(n)$ is contained in $K_\kappa^n$.
For every cone $C \in \Sigma_w \setminus \lc 0 \rc$, we consider the subset
\begin{align}\label{eq:N-dec}
N_{\kappa}^w(C):=\lc n \in N_{\kappa}^w \relmid \bR_{\geq 0} \cdot \lb \conv \lb \lc w \rc \cup K_\kappa^n \rb-w \rb =C \rc.
\end{align}
By (1) of \pref{lm:kappa}, the subsets $\lc N_{\kappa}^w(C) \relmid C \in \Sigma_w \setminus \lc 0 \rc \rc$ cover the whole $N_{\kappa}^w$.
Notice also that $N_{\kappa}^w(C)$ is contained in a small neighborhood of the cell in $\scrP$ dual to $\conv \lb \lc w \rc \cup K_\kappa^n \rb \in \scrT$ with $n \in N_{\kappa}^w(C)$.

Let $\varepsilon >0$ be another small real number, and set $D_\varepsilon:=\lc x \in \bC \relmid |x| < \varepsilon \rc$.
(Notice that we distinguish this $\varepsilon$ and $\epsilon$ appearing in \pref{th:main1} and \pref{th:main2}.)
We choose a smooth function $\phi \colon N_{\kappa}^w \times D_\varepsilon \to N_\bR$ such that
\begin{align}\label{eq:phase-shift}
\la m-w, \phi(n, x) \ra = \Arg \lb 1+x \rb - \arg \lb -\frac{c_m}{c_w} \rb
\end{align}
for any $n \in N_{\kappa}^w, x \in D_\varepsilon$, and $m \in K_\kappa^n$.
Here $\Arg$ denotes the principal value of the argument, i.e., the value of the argument in $\lb - \pi, \pi \rd$.
The branch of $\arg \lb -c_m/c_w \rb$ is the one we choose in order to transport the positive real locus $\tilde{Z}_t^w \cap N_{\bR_{>0}}$ in \eqref{eq:ch-family}.
Such a function $\phi$ is called a \emph{phase-shifting function} in \cite[Section 5.2]{MR4194298}.
We refer the reader to Example 5.4 in loc.cit.~for an example of a function $\phi$ (with $w=0, x=0, c_w=-1$).

The function $\phi$ can be constructed as follows: 
We take a Euclidean metric on $N_\bR \cong \bR^{d+1}$, and identify $N_\bR \cong M_\bR$.
For $x \in D_\varepsilon$, we consider the piecewise linear function on the normal fan $\Sigma_w$ of $\nabla^w$, which takes the value $\Arg \lb 1+x \rb - \arg \lb -c_m/c_w \rb$ at the primitive generator $(m-w)$ of the $1$-dimensional cone $\bR_{\geq 0} \cdot (m-w) \in \Sigma_w$.
By varying $x \in D_\varepsilon$, we obtain a function on $M_\bR \times D_\varepsilon \cong N_\bR \times D_\varepsilon$.
We further take the gradient of a smoothing of it, and compose it with the projection
\begin{align}
T \lb N_\bR \times D_\varepsilon \rb \cong T \lb N_\bR \rb \times T \lb D_\varepsilon \rb 
\cong T_0 \lb N_\bR \rb \times N_\bR \times T \lb D_\varepsilon \rb 
\to T_0 \lb N_\bR \rb \cong N_\bR,
\end{align}
where $T$ denotes the tangent bundle, and $T_0 \lb N_\bR \rb$ is the tangent space at $0 \in N_\bR$.
By taking its translation such that the origin $0 \in N_\bR$ moves to a point in $\Int \lb \nabla^w \rb$ and the subset $N_{\kappa}^w(C)$ of \eqref{eq:N-dec} for any $C \in \Sigma_w \setminus \lc 0 \rc$ is contained in the union of cones in $\Sigma_w$ containing $C$, we obtain the function $\phi$.

We set
\begin{align}
N_{\kappa, \bC}^w:=\lc n \in N_\bC \relmid \Re \lb n \rb \in N_\kappa^w \rc.
\end{align}
We also consider the maps
\begin{align}\label{eq:Phi}
\Phi_t \colon N_{\kappa}^w \times D_\varepsilon \to N_{\kappa, \bC}^w, \quad (n, x) \mapsto n+\sqrt{-1} \cdot \frac{\phi(n, x)}{\log t},
\end{align}
and
\begin{align}
i_t \colon N_\bC \to N_{\bC^\ast}
\end{align}
induced by $\bC \to \bC^\ast, c \mapsto t^c$, where $t >0$.
We set
\begin{align}
R_t:=\lc n \in N_\bR \relmid \frac{1}{2} \leq \tilde{f}_t^w \lb i_t(n) \rb \leq \frac{3}{2} \rc.
\end{align}

\begin{lemma}
One has $R_t \subset N_\kappa^w$ for sufficiently small $t>0$.
\end{lemma}
\begin{proof}
Let $n \in N_\bR \setminus N_\kappa^w$ be an element.
We will show $n \nin R_t$ for sufficiently small $t>0$.
We have either $ \mu_w (n)-\kappa > \min_{m \in A \setminus \lc w \rc} \mu_m(n)$ or $\min_{m \in A \setminus \lc w \rc} \mu_m(n) > \mu_w(n)+\kappa$.
In the former case, we have
\begin{align}
\tilde{f}_t^w \lb i_t(n) \rb \geq \frac{r_{m_0}}{r_w} \cdot t^{\lb \mu_{m_0}-\mu_w \rb(n)}
\geq \frac{r_{m_0}}{r_w} \cdot t^{-\kappa}
\geq C \cdot t^{-\kappa},
\end{align}
where $m_0 \in A \setminus \lc w \rc$ is an element such that $\mu_{m_0}(n)=\min_{m \in A \setminus \lc w \rc} \mu_m(n)$, and $C:=\min_{m \in A \setminus \lc w \rc} r_m /r_w>0$.
When $t>0$ is sufficiently small, we have $C \cdot t^{-\kappa} > 3/2$, which implies $n \nin R_t$.
Also in the latter case, we have
\begin{align}
\tilde{f}_t^w \lb i_t(n) \rb=\sum_{m \in A \setminus \lc w \rc} \frac{r_m}{r_w} t^{\lb \mu_m-\mu_w\rb(n)}
\leq C' \cdot t^{\kappa},
\end{align}
where $C':=\left| A \setminus \lc w \rc \right| \cdot \max_{m \in A \setminus \lc w \rc} r_m /r_w>0$.
When $t>0$ is sufficiently small, we have $C \cdot t^{\kappa} < 1/2$, which implies $n \nin R_t$.
We obtained the claim.
\end{proof}

\begin{proposition}{\rm(cf.~\cite[Proposition 5.3]{MR4194298})}\label{pr:delta}
For sufficiently small $t >0$, there exists a smooth map $\delta_t \colon R_t \times D_\varepsilon \to N_\bC$ satisfying the following conditions:
\begin{enumerate}
\item For all $(n, x) \in R_t \times D_\varepsilon$, one has
\begin{align}\label{eq:xff}
\frac{1}{1+x} \cdot f_t^w \lb i_t \lb \Phi_t(n, x)+\delta_t(n, x) \rb \rb=\frac{1}{|1+x|} \cdot \tilde{f}_t^w \lb i_t(n) \rb.
\end{align}
\item $\left| \left| \delta_t \right| \right|_{C^1} = O(t^\kappa)$, where $|| \bullet ||_{C^1}$ denotes the $C^1$-norm over $R_t \times D_\varepsilon$.
\end{enumerate}
\end{proposition}

We define
\begin{align}
\widetilde{\Phi}_t \colon R_t \times D_\varepsilon \to N_\bC \times D_\varepsilon, \quad (n, x) \mapsto 
\lb \Phi_t(n, x)+\delta_t(n, x), x \rb
\end{align}
and set
\begin{align}\label{eq:Btw}
B_t^{w}&:=\lc n \in R_t \relmid \tilde{f}_t^w \lb i_t \lb n \rb \rb=1 \rc \\ \label{eq:Ctw}
C_t^w&:=i_t \circ \pi_1 \lb \widetilde{\Phi}_t \lb B_t^w \times \lc 0 \rc \rb \rb,
\end{align}
where $\pi_1 \colon N_\bC \times D_\varepsilon \to N_\bC$ is the first projection.
By \eqref{eq:xff} with $x=0$, we have $C_t^w \subset Z_t$.
In the limit $t \to +0$, the set $B_t^w$ converges to the boundary of the polytope $\nabla^w \in \scrP$.
The set $C_t^w$ is the cycle over which we integrate forms.

We further take a real number $\varepsilon_0$ such that $0 < \varepsilon_0 < \varepsilon$.
We set $S_{\varepsilon_0}^1:=\lc x \in \bC \relmid |x|=\varepsilon_0 \rc \subset D_\varepsilon$ and 
\begin{align}\label{eq:Swt}
S^w_t&:=\lc (n, x) \in R_t \times S_{\varepsilon_0}^1 \relmid \tilde{f}_t^w \lb i_t \lb n \rb \rb=|1+x| \rc\\
T^w_t&:=i_t \circ \pi_1 \lb \widetilde{\Phi}_t \lb S^w_t \rb \rb.
\end{align}
From \eqref{eq:xff} again, we can see that $T^w_t \subset N_{\bC^\ast} \setminus Z_t$ is a tube over the cycle $C_t^w$.
The rest of this section is devoted to prove \pref{pr:delta}.

\begin{proof}[Proof of \pref{pr:delta}]
We define a holomorphic function $g_t^w \colon N_\bC \times D_\varepsilon \to \bC$ by
\begin{align}
g_t^w \lb n, x \rb:=\frac{1}{1+x} f_t^w \lb i_t \lb n\rb \rb
=\frac{1}{1+x} \sum_{m \in A \setminus \lc w \rc} \lb - \frac{k_{m, t}}{k_{w, t}} \rb t^{\la m-w, n \ra},
\end{align}
and the function $\xi_t^w \colon N_{\kappa}^w \times D_\varepsilon \to \bC$ by
\begin{align}\label{eq:df-xi}
\xi_t^w \lb n, x \rb:=\frac{1}{|1+x|} \tilde{f}_t^w \lb i_t(n) \rb-g_t^w \lb \Phi_t (n, x), x \rb.
\end{align}

\begin{lemma}{\rm(cf.~\cite[Lemma 5.2]{MR4194298})}\label{lm:diff}
There is some constant $C >0$ such that for any $(n, x) \in N_\kappa^w \times D_\varepsilon$, one has
\begin{align}
\left| \xi_t^w (n, x) \right| \leq C \cdot t^\kappa.
\end{align}
\end{lemma}

\begin{proof}
By \eqref{eq:phase-shift}, one has
\begin{align}
g_t^w \lb \Phi_t (n, x), x \rb
&=\frac{1}{1+x} \sum_{m \in A \setminus \lc w \rc} \lb - \frac{k_{m, t}}{k_{w, t}} \rb t^{\la m-w, n+\sqrt{-1} \frac{\phi \lb n, x\rb}{\log t} \ra} \\
&=O \lb t^\kappa \rb+\frac{1}{\left| 1+x \right|} \sum_{m \in K_\kappa^n} \frac{r_m}{r_w} \lb 1+O(t^{2\kappa}) \rb t^{\lb \mu_m-\mu_w\rb(n)},
\end{align}
where we also used (3) of \pref{lm:kappa}.
One also has
\begin{align}
\frac{1}{|1+x|} \tilde{f}_t^w \lb i_t(n) \rb
&=\frac{1}{|1+x|} \sum_{m \in A \setminus \lc w \rc} \frac{r_m}{r_w} t^{\lb \mu_m-\mu_w\rb(n)}\\
&=O \lb t^\kappa \rb+\frac{1}{\left| 1+x \right|} \sum_{m \in K_\kappa^n} \frac{r_m}{r_w} t^{\lb \mu_m-\mu_w\rb(n)}.
\end{align}
By combining these, we get
\begin{align}
\xi_t^w \lb n, x \rb=O \lb t^\kappa \rb+\frac{1}{\left| 1+x \right|} \sum_{m \in K_\kappa^n} O(t^{2\kappa}) \frac{r_m}{r_w} t^{\lb \mu_m-\mu_w\rb(n)}.
\end{align}
Since we have $\min_{m \in A \setminus \lc w \rc} \mu_m(n) -\mu_w (n) \geq -\kappa$ for $n \in N_\kappa^w$, this is $O \lb t^\kappa \rb$. 
We obtained the claim.
\end{proof}

We consider the gradient vector field of $g_t^w \lb \bullet, x\rb$ with fixed $x \in D_\varepsilon$ on $N_\bC \times \lc x \rc$
\begin{align}
\grad g_t^w(n, x):=\lb \overline{\frac{\partial g_t^w}{\partial n_0}}, \cdots, \overline{\frac{\partial g_t^w}{\partial n_d}} \rb,
\end{align}
where $(n_0, \cdots, n_{d})$ are $\bC$-coordinates on $N_\bC \cong \bC^{d+1}$.

\begin{lemma}{\rm(cf.~\cite[Lemma 5.6]{MR4194298})}\label{lm:grad}
When $t>0$ is sufficiently small, one has $\left| \grad g_t^w (n, x) \right| \neq 0$ on $N_{2\kappa, \bC}^w \times D_\varepsilon$.
Furthermore, there exist constants $C_1, C_2 >0$ such that for any $(n, x) \in N_{2\kappa, \bC}^w \times D_\varepsilon$ satisfying $g_t^w \lb n, x \rb \neq 0$, we have
\begin{align}\label{eq:grad/g}
\frac{\left| \grad g_t^w (n, x) \right|}{\left| g_t^w (n, x) \right|}
\geq
(-\log t)\lb C_1-C_2 t^\kappa \rb.
\end{align}
\end{lemma}

\begin{remark}\label{rm:kappa2}
\pref{lm:grad} can be shown in the same way as \cite[Lemma 5.6]{MR4194298}.
For the proof of \cite[Lemma 5.6]{MR4194298} (and the above lemma), we need the real part of the element $n$ to satisfy the condition (2) of \pref{lm:kappa}.
As mentioned in \pref{rm:kappa}, the condition (2) does not hold for all elements in $N_\bR$ but for all elements in $N_{2 \kappa}^w$ in general.
Therefore, we suppose $n \in N_{2\kappa, \bC}^w$ in \pref{lm:grad}, although it is wrongly claimed in \cite[Lemma 5.6]{MR4194298} that the claim holds for all elements $n$ in $N_\bC$.
\end{remark}

We consider the differential equation for an unknown function $c \colon R_t \times D_\varepsilon \times \ld 0, 2\rd \to N_\bC$
\begin{align}\label{eq:de}
\frac{d}{ds} c \lb n, x, s \rb=\xi_t^w \lb n, x \rb \cdot \frac{\grad g_t^w}{\left| \grad g_t^w \right|^2} \lb c \lb n, x, s \rb, x \rb
\end{align}
with the initial condition $c \lb n, x, 0 \rb=\Phi_t (n, x) \in N_{\kappa, \bC}^w$.
\pref{pr:delta} is proved by showing that there exists a global solution $c \lb n, x, s \rb \in N_{2 \kappa, \bC}^w$ for \eqref{eq:de}, and the map $\delta_t \colon R_t \times D_\varepsilon \to N_\bC$ defined by
\begin{align}
\delta_t (n, x):=c(n, x, 1)-\Phi_t (n, x)
\end{align}
satisfies the conditions (1) and (2) in \pref{pr:delta}.
These can be shown basically in the same way as done in \cite[Proposition 5.3]{MR4194298}.
However, in their proof, \cite[Lemma 5.6]{MR4194298} is used for the solution $c \lb n, x, s \rb$ (with $x=0$).
Since we actually  need to impose $n \in N_{2\kappa, \bC}^w$ in \cite[Lemma 5.6]{MR4194298} (and \pref{lm:grad}) as explained in \pref{rm:kappa2}, we need to modify their proof of \cite[Proposition 5.3]{MR4194298} so that we can also see that the solution $c \lb n, x, s \rb$ sits in $N_{2 \kappa, \bC}^w$.
We give the detail of how to do it in the following.

We fix $(n, x) \in R_t \times D_\varepsilon$.
Let $s(n, x) \in \ld 0, 2\rd$ be the supremum of $s' \in \ld 0, 2\rd$ such that there exists a solution $c(n, x, s)$ of \eqref{eq:de} on the interval $\ld 0, s' \rb$ and $c(n, x, s) \in N_{2\kappa, \bC}^w$ for all $s \in \ld 0, s' \rb$.
For any $s_0 \in \ld 0, s(n, x) \rb$, we have
\begin{align}
\int_0^{s_0} \overline{\grad g_t^w \lb c \lb n, x, s \rb, x \rb} \cdot \frac{d}{ds} c \lb n, x, s \rb ds=\int_0^{s_0} \xi_t^w \lb n, x \rb ds=s_0 \cdot \xi_t^w \lb n, x \rb.
\end{align}
This is also equal to $g_t^w \lb c \lb n, x, s_0 \rb, x \rb-g_t^w \lb c \lb n, x, 0 \rb, x \rb$.
Hence, we have
\begin{align}\label{eq:ggxi}
g_t^w \lb c \lb n, x, s_0 \rb, x \rb=g_t^w \lb \Phi_t (n, x), x \rb+s_0 \cdot \xi_t^w \lb n, x \rb.
\end{align}
By this and \eqref{eq:df-xi}, one can get
\begin{align}
\left| g_t^w \lb c \lb n, x, s_0 \rb, x \rb \right| &\geq \left| g_t^w \lb \Phi_t (n, x), x \rb \right|-s_0 \left| \xi_t^w (n, x) \right| \\
&\geq \frac{1}{|1+x|} \cdot \tilde{f}_t^w \lb i_t(n) \rb- \lb 1+s_0 \rb \left| \xi_t^w (n, x) \right| \\
&\geq \frac{1}{2|1+x|}-\lb 1+s_0 \rb C \cdot t^\kappa,
\end{align}
where we used \pref{lm:diff} and $1/2 \leq \tilde{f}_t^w \lb i_t(n) \rb$ for $n \in R_t$ in the last inequality.
This is greater than, for instance, $1/3$ when $\varepsilon, t >0$ are sufficiently small.
Thus by \pref{lm:grad}, we have
\begin{align}\label{eq:gradg}
\left| \grad g_t^w \lb c \lb n, x, s_0 \rb, x \rb \right| \geq (-\log t) \rho_0
\end{align}
with some constant $\rho_0 > 0$ for sufficiently small $t >0$.
Therefore, \eqref{eq:de} implies
\begin{align}
\left| \frac{d}{ds} c \lb n, x, s \rb \right| 
= \frac{\left| \xi_t^w \lb n, x \rb \right|}{\left| \grad g_t^w \lb c \lb n, x, s \rb, x \rb \right|}
\leq 
\frac{\left| \xi_t^w \lb n, x \rb \right|}{(-\log t) \rho_0}
\leq C \rho_0^{-1} (-\log t)^{-1} t^\kappa
\end{align}
for $s \in \ld 0, s(n, x) \rb$.
This implies the limit $\lim_{s \to s(n, x)-0} c(n, x, s)$ exists.
Suppose $s(n, x) <2$.
Then the solution for \eqref{eq:de} can be extended to a larger interval $\ld 0, s(n, x)+\varepsilon_1 \rb$ with some small $\varepsilon_1 >0$.
For any $s_0 \in \ld 0, s(n, x)+\varepsilon_1 \rb$, one has
\begin{align}\label{eq:cc}
\left| c(n, x, s_0)-c(n, x, 0) \right| \leq \int_0^{s_0}  \left| \frac{d}{ds} c \lb n, x, s \rb \right| ds \leq 2C \rho_0^{-1} (-\log t)^{-1} t^\kappa
\end{align}
Since $c \lb n, x, 0 \rb=\Phi_t (n, x) \in N_{\kappa, \bC}^w$, one can get
\begin{align}
\left| \lb \min_{m \in A \setminus \lc w \rc} \mu_m - \mu_w \rb \lb \Re \lb c(n, x, s_0 )\rb \rb \right|
&=\left| \lb \min_{m \in A \setminus \lc w \rc} \mu_m - \mu_w \rb 
\lb \Re \lb c(n, x, s_0)-c(n, x, 0)\rb +\Re \lb c(n, x, 0 )\rb \rb \right| \\
& \leq C' (-\log t)^{-1} t^\kappa  +\kappa \\
& \leq 2 \kappa
\end{align}
when $t >0$ is sufficiently small.
Here $C' >0$ is some constant.
Hence, one has $c \lb n, x, s_0 \rb \in N_{2\kappa, \bC}^w$ for any $s_0 \in \ld 0, s(n, x)+\varepsilon_1 \rb$.
This contradicts the original assumption on $s(n, x)$.
We conclude that $s(n, x)=2$ and the solution $c(n, x, s) \in N_{2\kappa, \bC}^w$ exists on the interval $\ld 0, 2\rd$.

The remaining claim to prove is that $\delta_t$ satisfies the conditions (1) and (2) in \pref{pr:delta}.
This can be proved in a similar manner to \cite[Proposition 5.3]{MR4194298}.
(In the proof, we use \eqref{eq:ggxi} with $s_0=1$.
Therefore, we need $1 \in \ld 0, s(n, x) \rb$.
This is the reason why we consider the differential equation \eqref{eq:de} on $R_t \times D_\varepsilon \times \ld 0, 2\rd$ rather than on $R_t \times D_\varepsilon \times \ld 0, 1\rd$.)
We conclude \pref{pr:delta}.
\end{proof}

\section{Proof of \pref{th:main1}}\label{sc:integral}

We keep the same assumptions and notation as in \pref{sc:main} and \pref{sc:sphere}.
We fix elements $w \in W$ and $v \in V_l$ $(l \geq 1)$.
We consider the asymptotics of the period 
\begin{align}
\int_{C_t^{w}} \Omega_t^{l, v} 
=\frac{1}{2 \pi \sqrt{-1}} \int_{T_t^{w}} \omega_t^{l, v}
= \frac{1}{2 \pi \sqrt{-1}} \int_{S_t^{w}} \widetilde{\Phi}_t^\ast \pi_1^\ast i_t^\ast \omega_t^{l, v}.
\end{align}
Following \cite{MR4194298}, we compute this by decomposing the domain $S_t^{w}$ into regions on which different monomials of $f_t^w$ are dominant among the monomials of $f_t^w$.
Let $\kappa>0$ be a small real number satisfying the conditions of \pref{lm:kappa}. 
We take a small constant $\epsilon >0$ so that $\epsilon <\kappa /2$.
For each pair of an element $q \in A \setminus \lc w \rc$ and a subset $K \subset A \setminus \lc w, q \rc$, we set
\begin{align}
S_t^{w, q, K}:=\lc (n, x) \in S_t^{w} \relmid 
\begin{array}{l}
\mu_k (n) -\mu_{q}(n) \in \ld 0 , \epsilon \rd, \forall k \in \lc q \rc \sqcup K \\
\mu_m (n) - \mu_q(n) \geq \epsilon, \forall m \in A \setminus \lb \lc w, q \rc \sqcup K \rb
\end{array}
\rc.
\end{align}
The subset $S_t^{w, q, K}$ is the region where $\mu_{q}-\mu_w$ is dominant and $\mu_k-\mu_w$ $(k \in K)$ are also nearly dominant among $\mu_m-\mu_w$ $(m \in A \setminus \lc w \rc)$.
When $\epsilon >0$ is sufficiently small, one has $S_t^{w, q, K} \neq \emptyset$ if and only if the convex hull of $\lc w, q \rc \sqcup K$ is contained in the triangulation $\scrT$ of $\Delta$.
We replace $\epsilon$ with a smaller one if necessary so that this holds.
One has
\begin{align}\label{eq:int-decomp}
\int_{C_t^{w}} \Omega_t^{l, v} 
=\frac{1}{2 \pi \sqrt{-1}} \sum_{q, K} \int_{S_t^{w, q, K}} \widetilde{\Phi}_t^\ast \pi^\ast_1 i_t^\ast \omega_t^{l, v},
\end{align}
where the sum is taken over $\lc q \rc \sqcup K \subset A \setminus \lc w \rc$ such that the convex hull of $\lc w, q \rc \sqcup K$ is contained in $\scrT$.

\begin{example}\label{eg:d=2}
Let $d=2$, and fix a basis $\lc e_1, e_2, e_3 \rc$ of the lattice $M \cong \bZ^3$.
Consider the polynomial
\begin{align}\label{eq:ex-poly}
f:=\sum_{m \in \Delta \cap M} c_m x^{\lambda_m} z^m
\end{align}
with $c_m \in \bC^\ast$ $(m \in \Delta \cap M)$ and 
\begin{align}
\Delta:=\conv \lb \lc 2e_1, -e_1, e_2, -e_2, e_3, -e_3 \rc \rb \subset M_\bR, \quad
\lambda_m:=
\left\{ \begin{array}{ll}
3 & m=2e_1 \\
0 & m=0 \\
1 & m \in \lb \Delta \cap M \rb \setminus \lc 0, 2e_1 \rc. 
\end{array} 
\right. 
\end{align}
The tropicalization $V(\trop (f)) \subset N_\bR$ is shown on the left in \pref{fg:surface}.
For an arbitrary element $x_0 \in S_{\varepsilon_0}^1$, the subset $S_t^{w} \cap \lc x=x_0 \rc \subset N_\bR$ with $w=0$ converges to the boundary $\partial \nabla^{w=0}$ of the cube $\nabla^{w=0}$ in $V(\trop (f))$ as $t \to +0$.
Under the identification $S_t^{w=0} \cap \lc x=x_0 \rc \cong \partial \nabla^{w=0}$, the right figure of \pref{fg:surface} shows the slice by $\lc x=x_0 \rc$ of the decomposition $\lc S_t^{w=0, q, K} \relmid q, K \rc$ of $S_t^{w=0}$.
For instance, the red (resp. blue, green) region shows the intersection of $S_t^{0, e_1, \lc e_2, e_3\rc}$ (resp. $S_t^{0, e_1, \lc e_2 \rc}$ $S_t^{0, e_1, \emptyset}$) with $\lc x=x_0 \rc$.

\begin{figure}[htbp]
\begin{center}
\includegraphics[scale=0.55]{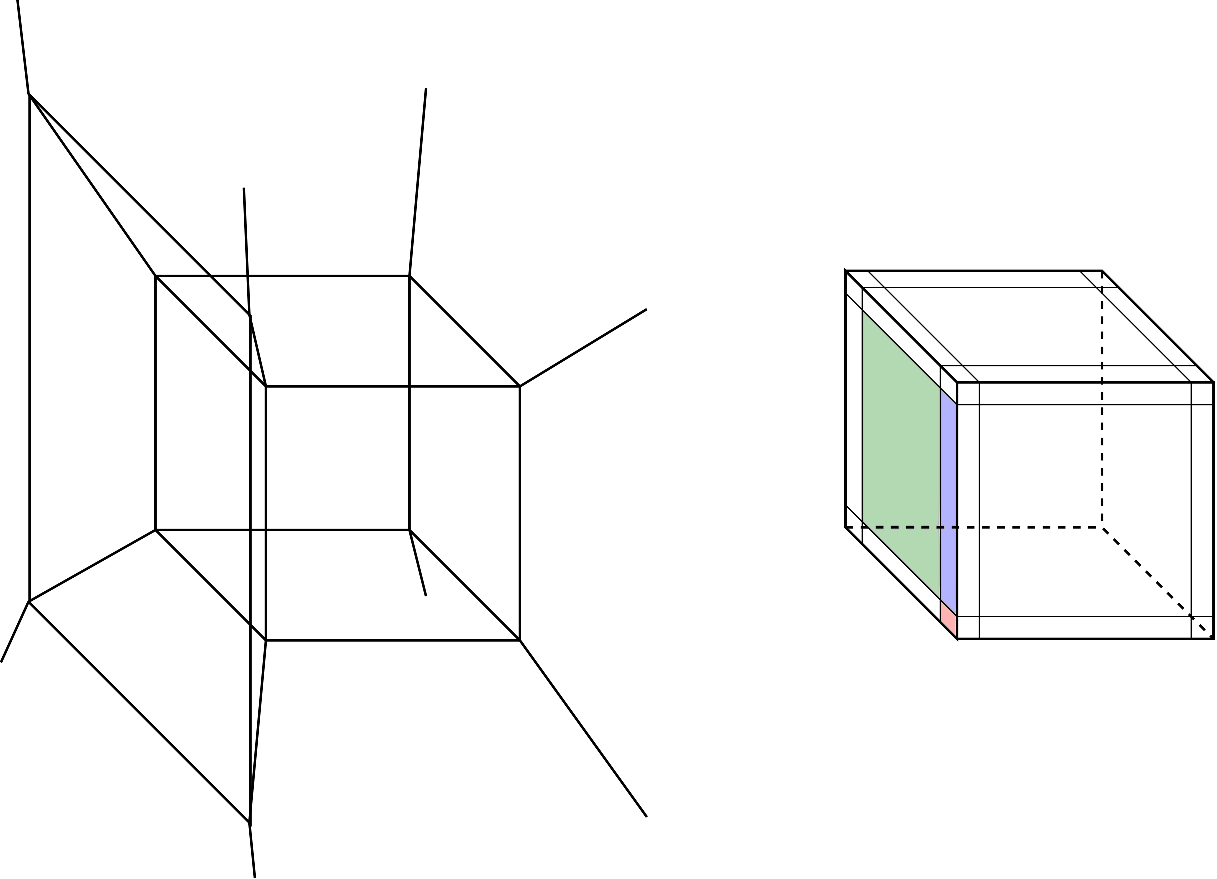}
\end{center}
\caption{The tropicalization $V(\trop (f))$ and the regions $S_t^{0, q, K} \cap \lc x=x_0 \rc$}
\label{fg:surface}
\end{figure}

\end{example}

The outline of the proof of \pref{th:main1} is as follows:
First, for each region $S_t^{w, q, K}$, we simplify the defining equation $1+x=f_t^w$ by throwing away monomials that are not dominant on $S_t^{w, q, K}$, i.e., monomials that correspond to the elements in $A \setminus \lb \lc w, q \rc \cup K \rb$ to compute the local integral $\int_{S_t^{w, q, K}} \widetilde{\Phi}_t^\ast \pi^\ast_1 i_t^\ast \omega_t^{l, v}$ approximately (\pref{sc:local-integral}).
We proceed the local computation by using the Duistermaat--Heckman theorem and a fact on the complex volumes of polytopes introduced by \cite{MR4194298}, after replacing the integration regions with some approximating polytopes (\pref{sc:further}).
Then we take the sum of the local integrals to obtain the total value of the period integral (\pref{sc:sum}).
We will get the formula
\begin{align}\label{eq:int-G}
\int_{C_t^{w}} \Omega_t^{l, v}=
(-1)^{d+p_w}
\int_{Y_{w}} 
t^{-\omega_{\lambda}^{w}}
\exp \lb -\sqrt{-1} \sum_{m \in A_{w}} 
\arg \lb -\frac{c_m}{c_w}\rb D_m^w \rb
\binom{\sigma^w+l-p_w-1}{l-1}
\widehat{G}+O \lb t^\epsilon \rb,
\end{align}
where $Y_{w}, D_m^w, \omega_{\lambda}^{w}, \sigma^w$ are the ones defined in \pref{sc:main}, and the cohomology class $\widehat{G}$ will be defined in \eqref{eq:hG}.
\pref{th:main1} for the case where $\conv \lb \lc w \rc \cup \tau_v \rb \nin \scrT$ immediately follows from this as we will see in the end of \pref{sc:sum}.
In order to show \pref{th:main1} for the case where $\conv \lb \lc w \rc \cup \tau_v \rb \in \scrT$, we will rewrite the formula \eqref{eq:int-G} in terms of the gamma function using the Dirichlet integral (\pref{sc:re-gamma}).

 \subsection{Local computation of the period integral}\label{sc:local-integral}

For a pair $\lb q, K \rb$, we define the integral affine functions $\alpha, \beta_k \colon N_\bC \to \bC$ $(k \in K)$ as
\begin{align}\label{eq:abg}
\alpha:=\mu_q-\mu_{w}, \quad \beta_k:=\mu_k-\mu_q.
\end{align}
Since the triangulation $\scrT$ is unimodular, we can take a collection of integral linear functions $\lc \gamma_j \rc_{j \in J}$ so that $\lc \alpha, \beta_k, \gamma_j \relmid k \in K, j \in J \rc$ forms an integral affine coordinate system on $N_\bC$.
We will write the integrand $\widetilde{\Phi}_t^\ast \pi^\ast_1 i_t^\ast \omega_t^{l, v}$ of the integral $\int_{S_t^{w, q, K}} \widetilde{\Phi}_t^\ast \pi^\ast_1 i_t^\ast \omega_t^{l, v}$ in terms of this integral affine coordinate system to compute the integral.

\begin{lemma}\label{lm:a-range}
When $t>0$ is sufficiently small, one has
\begin{align}
\log_t C_1 \leq \alpha \leq \log_t C_2
\end{align}
on $S_t^{w, q, K}$ for some constants $C_1, C_2>0$.
In particular, we have $-\epsilon \leq \alpha \leq \epsilon$, when $t>0$ is sufficiently small.
\end{lemma}
\begin{proof}
On $S_t^{w, q, K}$, we have 
\begin{align}
\frac{r_q}{r_{w}} t^{(\mu_q-\mu_{w})(n)} \leq 
\sum_{m \in A \setminus \lc w \rc} \frac{r_m}{r_{w}} t^{(\mu_m-\mu_{w})(n)}=|1+x| \leq 2,
\end{align}
from which we get the former inequality.
We also have
\begin{align}
\frac{1}{2} \leq |1+x| =\sum_{m \in A \setminus \lc w \rc} \frac{r_m}{r_{w}} t^{(\mu_m-\mu_{w})(n)} 
\leq C \cdot t^{(\mu_q-\mu_{w})(n)}
\end{align}
for $C:=\left| A \setminus \lc w \rc \right| \cdot \max_{m\in A \setminus \lc w \rc} r_m/r_w>0$.
The latter inequality follows from this.
\end{proof}

On $S_t^{w, q, K}$, we have
\begin{align}
|1+x |&=\sum_{m \in A \setminus \lc w \rc} \frac{r_m}{r_{w}} t^{(\mu_m-\mu_{w})(n)} \\ \label{eq:1xta}
&= t^{\alpha} \lb \frac{r_q}{r_{w}}+ \sum_{k \in K} \frac{r_k}{r_w} t^{\beta_k} +
\sum_{m \in A \setminus \lb \lc w, q \rc \sqcup K \rb} \frac{r_m}{r_w} t^{\mu_m-\mu_q} 
\rb.
\end{align}
The last term $\sum_{m \in A \setminus \lb \lc w, q \rc \sqcup K \rb} \frac{r_m}{r_w} t^{\mu_m-\mu_q}$ is $O \lb t^\epsilon \rb$, and we can see by the implicit function theorem that this equation can be used to write $\alpha$ as a function $\alpha_{w, q, K} \lb \beta, \gamma, x \rb$ of the variables $\beta:=\lc \beta_k \rc_k, \gamma:=\lc \gamma_j \rc_j$, and $x$.
We set
\begin{align}\label{eq:alpha'}
\alpha_{w, q, K}' \lb \beta, x \rb:=\log_t |1+x | - \log_t \lb \frac{r_q}{r_{w}} + \sum_{k \in K} \frac{r_k}{r_{w}} t^{\beta_k} \rb.
\end{align}

\begin{lemma}\label{lm:ac}
On $S_t^{w, q, K}$, one has
\begin{align}\label{eq:aa'}
\alpha_{w, q, K} \lb \beta, \gamma, x \rb&=\alpha_{w, q, K}' \lb \beta, x \rb+O \lb t^\epsilon \rb, \\ \label{eq:ag}
\frac{\partial \alpha_{w, q, K}}{\partial \gamma_j} \lb \beta, \gamma, x \rb&=O \lb t^{\epsilon} \rb, \\ \label{eq:Pb}
\Phi_t^\ast \beta_k&=\beta_k + \frac{\sqrt{-1}}{\log t} \cdot \lb \arg \lb -\frac{c_q}{c_w} \rb- \arg \lb -\frac{c_k}{c_w} \rb \rb,
\end{align}
where $\Phi_t \colon N_\kappa^w \times D_\varepsilon \to N_{\kappa, \bC}^w$ is the map defined in \eqref{eq:Phi}.
\end{lemma}

There are similar formulas in \cite[Section 3.4, Section 5.4]{MR4194298}, and the above lemma can be checked in the same way.

The standard volume form on $N_\bR$ is given by $d\alpha \wedge \bigwedge_{k \in K} d\beta_k \wedge \bigwedge_{j \in J} d\gamma_j$.
We set $d \mathrm{vol}_{q, K}:=\bigwedge_{j \in J} d\gamma_j$.
Then we have
\begin{align}
\begin{split}
\widetilde{\Phi}_t^\ast \pi^\ast_1 i_t^\ast \omega_t^{l, v}
&=\widetilde{\Phi}_t^\ast \pi^\ast_1 i_t^\ast \lb \lb \bigwedge_{i=0}^{d} \frac{dz_i}{z_i} \rb
\frac{1}{\lc -k_{w, t} z^w \lb f_t^w-1 \rb \rc^l}
\prod_{m \in A \cap \tau_v} \lb k_{m, t} \cdot z^m \rb^{p_m} \rb \\
&=\lb \log t \rb^{d+1} \frac{1}{x^l} \cdot
\widetilde{\Phi}_t^\ast \lb \prod_{m \in A \cap \tau_v} \lb -\frac{k_{m, t}}{k_{w, t}} \rb^{p_m} \cdot t^{v-lw} \cdot d\alpha \bigwedge_{k \in K} d\beta_k \cdot d \mathrm{vol}_{q, K} \rb\\ \label{eq:integrand}
&=\lb \log t \rb^{d+1} \frac{t^{\lambda_{l, v}-l \lambda_w}}{x^l} \cdot
\widetilde{\Phi}_t^\ast \lb \prod_{m \in A \cap \tau_v} \lb 1+O(t^{2\kappa}) \rb \lb -\frac{c_{m}}{c_{w}} \rb^{p_m} \cdot t^{v-lw} \cdot d\alpha \bigwedge_{k \in K} d\beta_k \cdot d \mathrm{vol}_{q, K} \rb,
\end{split}
\end{align}
where $t^{v-lw}$ denotes the function $t^{\la v-lw, \bullet \ra}$ on $N_\bC$, and $\lambda_{l, v}:= \sum_{m \in A \cap \tau_v} p_m \lambda_m$.
On the other hand, we can write
\begin{align}\label{eq:xfw}
1+x=f_t^{w} \lb i_t \lb n \rb \rb=
-\frac{c_q}{c_{w}} t^{\alpha}
\lb 
1+
\sum_{k \in K}  
\frac{c_k}{c_q} t^{\beta_k}
+
h_t(n)
\rb
\end{align}
on $\widetilde{\Phi}_t \lb S_t^{w, q, K} \rb$ with a function $h_t \colon N_\bC \to \bC$.
It follows from the $C^0$-estimate for $\delta_t$ in \pref{pr:delta} that the function $h_t$ satisfies the uniform estimates
\begin{align}\label{eq:h-est}
h_t=O \lb t^\epsilon \rb, \quad \frac{\partial h_t}{\partial n_i}= O \lb \lb - \log t \rb t^\epsilon \rb,
\end{align}
where $(n_0, \cdots, n_{d})$ are $\bC$-coordinates on $N_\bC \cong \bC^{d+1}$.
By \eqref{eq:xfw}, we can also write
\begin{align}
\alpha=\log_t \lc -\frac{c_w}{c_{q}} (1+x)
\lb 
1+
\sum_{k \in K}  
\frac{c_k}{c_q} t^{\beta_k}
+
h_t(n)
\rb^{-1} \rc.
\end{align}
By using this, one can get
\begin{align}
d\alpha \bigwedge_{k \in K} d\beta_k \cdot d \mathrm{vol}_{q, K}
&=
\lb 
1+
\lb 1+\sum_{k \in K} \frac{c_k}{c_q} t^{\beta_k}+h_t(n) \rb^{-1}
\frac{\partial h_t}{\partial \alpha}
\frac{1}{\log t}
\rb^{-1}
\frac{1}{1+x} \frac{dx}{\log t} \bigwedge_{k \in K} d\beta_k \cdot d \mathrm{vol}_{q, K} \\
&=
\lb 
1+
t^\alpha \lb -\frac{c_q}{c_w} \rb (1+x)^{-1}
\frac{\partial h_t}{\partial \alpha}
\frac{1}{\log t}
\rb^{-1}
\frac{1}{1+x} \frac{dx}{\log t} \bigwedge_{k \in K} d\beta_k \cdot d \mathrm{vol}_{q, K}
\end{align}
on $\widetilde{\Phi}_t \lb S_t^{w, q, K} \rb$. 
From this, \eqref{eq:h-est}, \pref{lm:a-range} and the $C^0$-estimate for $\delta_t$, we obtain
\begin{align}
d\alpha \bigwedge_{k \in K} d\beta_k \cdot d \mathrm{vol}_{q, K}
=
\lb 1+ O \lb t^{\epsilon} \rb \rb \frac{1}{1+x} \frac{dx}{\log t} \bigwedge_{k \in K} d\beta_k \cdot d \mathrm{vol}_{q, K}.
\end{align}
By this and \eqref{eq:integrand}, we have
\begin{align}\label{eq:int}
\int_{S_t^{w, q, K}} \widetilde{\Phi}_t^\ast \pi^\ast_1 i_t^\ast \omega_t^{l, v}
=\lb 1+ O \lb t^{\epsilon} \rb \rb \lb \log t \rb^{d}
\int_{S_t^{w, q, K}} 
\frac{dx}{x^l (1+x)}
P_t^{v, w}(n, x) \cdot 
\widetilde{\Phi}_t^\ast \lb
\bigwedge_{k \in K} d\beta_k \cdot d \mathrm{vol}_{q, K} \rb,
\end{align}
where 
\begin{align}
P_t^{v, w}(n, x):=t^{\lambda_{l, v}-l \lambda_w} \cdot \widetilde{\Phi}_t^\ast \lb t^{v-lw} \rb
\prod_{m \in A \cap \tau_v} \lb -\frac{c_m}{c_w} \rb^{p_m}.
\end{align}

\begin{lemma}\label{lm:Ppi}
Let $\tau_{w, q, K} \in \scrT$ be the convex hull of $\lc w, q \rc \sqcup K$.
Then we have
\begin{align}\label{eq:Pt-vw}
P_t^{v, w}(n, x)=
\left\{ \begin{array}{ll}
\lb 1+ O \lb t^\epsilon \rb \rb 
\cdot (-1)^{p_w}
\cdot
(1+x)^{l-p_w} \cdot
\frac{r_q^{p_q} \prod_{k \in \tau_v \cap K} \lb r_k t^{\beta_k} \rb^{p_k}}{\lb r_q+\sum_{k \in K} r_k t^{\beta_k}\rb^{l-p_w} }
& \tau_v \subset \tau_{w, q, K} \\
O \lb t^{\epsilon} \rb
& \tau_v \not\subset \tau_{w, q, K}
\end{array} 
\right.
\end{align}
on $S_t^{w, q, K}$, where $p_w:=0$ if $w \nin \tau_v$, and $p_q:=0$ if $q \nin \tau_v$.
\end{lemma}
\begin{proof}
First, suppose $\tau_v \subset \tau_{w, q, K}$.
Since we have
\begin{align}
\lambda_{l, v}-l \lambda_w+v-lw
&=\sum_{m \in A \cap \tau_v} p_m \lc \lb \lambda_m-\lambda_w \rb+(m-w) \rc\\
&=p_q \alpha+\sum_{k \in K \cap \tau_v} p_k \lb \alpha+\beta_k \rb\\
&=(l-p_w) \alpha+\sum_{k \in K \cap \tau_v} p_k \beta_k,
\end{align}
one can see from \eqref{eq:xfw} that we have
\begin{align}
\begin{split}
t^{\lambda_{l, v}-l \lambda_w} \cdot \widetilde{\Phi}_t^\ast \lb t^{v-lw} \rb
&=
\widetilde{\Phi}_t^\ast \lb
t^{\lb l-p_w \rb \alpha+\sum_{k \in \tau_v \cap K} p_k \beta_k} \rb \\ \label{eq:Pt-vw1}
&=
\widetilde{\Phi}_t^\ast \lb
\lb -\frac{c_w}{c_q} \rb^{l-p_w}
\lb 
1+
\sum_{k \in K}  
\frac{c_k}{c_q} t^{\beta_k}
+
h_t(n)
\rb^{-l+p_w}
\lb 1+x \rb^{l -p_w}
t^{\sum_{k \in \tau_v \cap K} p_k \beta_k} \rb.
\end{split}
\end{align}
By the estimate 
\begin{align}\label{eq:est-log}
-\log t=O \lb t^{-\delta} \rb
\end{align}
that holds for any $\delta >0$, and the $C^0$-estimate for $\delta_t$ in \pref{pr:delta}, we also have 
\begin{align}\label{eq:t-delta}
\left| t^{\la m, \delta_t \ra} \right| \leq t^{Ct^\kappa}=\exp \lb C t^\kappa \log t \rb
=1+O\lb t^\kappa \log t \rb=1+O \lb t^\epsilon \rb
\end{align}
for $m \in M_\bR$, where $C>0$ is some constant.
By \eqref{eq:Pb}, we also have $t^{\Phi_t^\ast \lb \beta_k \rb}=O \lb 1 \rb$.
From this, \eqref{eq:t-delta}, and \eqref{eq:h-est}, we can see that the above \eqref{eq:Pt-vw1} is equal to
\begin{align}
\begin{split}
\lb -\frac{c_w}{c_q} \rb^{l-p_w}
\lb 
1+
\sum_{k \in K}  
\frac{c_k}{c_q} t^{\Phi_t^\ast (\beta_k)} \lb 1+O \lb t^\epsilon \rb \rb
+
O \lb t^\epsilon \rb
\rb^{-l+p_w}
\lb 1+x \rb^{l -p_w}
t^{\sum_{k \in \tau_v \cap K} p_k \beta_k} \\
=
\lb 1+ O \lb t^\epsilon \rb \rb \cdot 
\lb -\frac{c_w}{c_q} \rb^{l-p_w}
\lb 
1+
\sum_{k \in K}  
\frac{c_k}{c_q} t^{\Phi_t^\ast \lb \beta_k \rb} \rb^{-l+p_w}
\lb 1+x \rb^{l -p_w}
t^{\sum_{k \in \tau_v \cap K} p_k \Phi_t^\ast \lb \beta_k \rb}.
\end{split}
\end{align}
By this and \pref{eq:Pb}, one can get \eqref{eq:Pt-vw} for the case where $\tau_v \subset \tau_{w, q, K}$.

Next, suppose $\tau_v \not\subset \tau_{w, q, K}$.
By \eqref{eq:t-delta} again, we have
\begin{align}
t^{\lambda_{l, v}-l \lambda_w} \cdot \widetilde{\Phi}_t^\ast \lb t^{v-lw} \rb
=\widetilde{\Phi}_t^\ast \lb \prod_{m \in A \cap \tau_v} t^{p_m \lb \mu_m-\mu_w \rb} \rb
=
\lb 1+ O \lb t^\epsilon \rb \rb \cdot
\Phi_t^\ast
 \lb \prod_{m \in A \cap \tau_v} t^{p_m \lb \mu_m-\mu_w \rb} \rb.
\end{align}
We also have
\begin{align}
\left| \Phi_t^\ast \lb \prod_{m \in A \cap \tau_v} t^{p_m \lb \mu_m-\mu_w \rb} \rb \right|
&=
\prod_{m \in \lb A \cap \tau_v \rb \setminus \lc w \rc} t^{p_m \lb \mu_m-\mu_w \rb} \\
&=
\prod_{m \in \lb A \cap \tau_v \rb \setminus \lc w \rc} t^{p_m \lb \mu_m-\mu_q \rb}
\cdot
\prod_{m \in \lb A \cap \tau_v \rb \setminus \lc w \rc} t^{p_m \lb \mu_q-\mu_w \rb} \\
& \leq C_1 \cdot 
\prod_{m \in \lb A \cap \tau_v \rb \setminus \lc w \rc} t^{p_m \lb \mu_m-\mu_q \rb}
\end{align}
for some constant $C_1>0$.
We used \pref{lm:a-range} in the last inequality.
Let $m_0 \in A \cap \tau_v$ be an element that is not in $\tau_{w, q, K}$.
Then we have
\begin{align}
\prod_{m \in \lb A \cap \tau_v \rb \setminus \lc w \rc} t^{p_m \lb \mu_m-\mu_q \rb}
\leq
t^{p_{m_0} \epsilon}
\cdot 
\prod_{m \in \lb A \cap \tau_v \rb \setminus \lc w, m_0 \rc} t^{p_m \lb \mu_m-\mu_q \rb}
=O \lb t^{\epsilon} \rb.
\end{align}
The claim of the lemma in the case of $\tau_v \not\subset \tau_{w, q, K}$ follows from these.
\end{proof}

We can see from \pref{lm:Ppi} that $\frac{1}{x^l (1+x)} P_t^{v, w}(n, x)$ is uniformly bounded.
It is also obvious from the construction of $\Phi_t$ that the $C^1$-norm of $\Phi_t$ (and hence of $\left. \Phi_t \right|_{S_t^w}$) is also bounded.
Furthermore, when we define the $C^1$-norm of $\left. \delta_t \right|_{S_t^w}$ using the the Riemannian metric induced from the Euclidean metric on the ambient space $N_\bR \times D_\varepsilon$, we have $\left| \left| \left. \delta_t \right|_{S_t^w} \right| \right|_{C^1}=O \lb t^\kappa \rb$ by \pref{pr:delta}.
From these facts and \pref{lm:vol} which we will prove later, we can see that \eqref{eq:int} is written as
\begin{align}\label{eq:int2}
\lb 1+ O \lb t^\epsilon \rb \rb \lb \log t \rb^{d} \lb
\int_{S_t^{w, q, K}} 
\frac{dx}{x^l (1+x)}
P_t^{v, w}(n, x)
\Phi_t^\ast \lb
\bigwedge_{k \in K} d\beta_k \cdot d \mathrm{vol}_{q, K} \rb +O \lb t^\kappa \rb \rb.
\end{align}
From \eqref{eq:Pb} and \pref{lm:Ppi}, we can see that \eqref{eq:int2} is equal to
\begin{align}\label{eq:int3}
\lb 1+ O \lb t^\epsilon \rb \rb (-1)^{p_w} \lb \log t \rb^{d} \lb
\int_{S_{\varepsilon_0}^1} 
\frac{(1+\tilde{x})^{l-p_w-1}}{\tilde{x}^l}
d\tilde{x}
\int_{\ld 0, \epsilon \rd^{|K|}}
\lb 
\int_{S_t^{w, q, K} \cap \lc \beta=b, x=\tilde{x} \rc}
\Phi_t^\ast \lb d \mathrm{vol}_{q, K} \rb
\rb
\phi_{q, K}^{v, w} \lb b \rb d b +O \lb t^\kappa \rb \rb,
\end{align}
where $b:=\lc b_k \rc_{k \in K}$ is the coordinate system on $\ld 0, \epsilon \rd^{|K|}$ and $\phi_{q, K}^{v, w}$ is a function on $\ld 0, \epsilon \rd^{|K|}$ defined by
\begin{align}
\phi_{q, K}^{v, w}(b):=
\left\{ \begin{array}{ll}
\frac{r_q^{p_q} \prod_{k \in \tau_v \cap K} \lb r_k t^{b_k} \rb^{p_k}}{\lb r_q+\sum_{k \in K} r_k t^{b_k}\rb^{l-p_w} } & \tau_v \subset \tau_{w, q, K} \\
0 & \tau_v \not\subset \tau_{w, q, K}.
\end{array} 
\right.
\end{align}

\subsection{The Duistermaat--Heckman theorem/complex volumes of polytopes}\label{sc:further}

In this subsection, we compute the integral 
\begin{align}
\int_{S_t^{w, q, K} \cap \lc \beta=b, x=\tilde{x} \rc}
\Phi_t^\ast \lb d \mathrm{vol}_{q, K} \rb
\end{align}
appearing in \eqref{eq:int3}.
Following \cite[Section 3.4]{MR4194298}, for a subset $J \subset A \setminus \lc w, q \rc$ containing $K$, we consider the polytope $E_{q, J} \lb \lc b_j \rc_{j \in J}, x \rb$ in the $\gamma$-plane defined by
\begin{align}
\mu_j \lb \alpha_{w, q, K}'(b, x), b, \gamma \rb-\mu_q \lb \alpha_{w, q, K}'(b, x), b, \gamma \rb &=b_j, \forall j \in J \setminus K  \\
\mu_m \lb \alpha_{w, q, K}'(b, x), b, \gamma \rb-\mu_q \lb \alpha_{w, q, K}'(b, x), b, \gamma \rb &\geq 0, \forall m \in A \setminus \lb \lc w, q \rc \sqcup J \rb.
\end{align}
We also set $D_{q, K, J}(b, x):=\bigcup_{b' \in \ld 0, \epsilon \rd^{|J \setminus K|}} E_{q, J} \lb b, b', x \rb$.

\begin{lemma}\label{lm:E}
One has
\begin{align}\label{eq:inc-exc}
\int_{S_t^{w, q, K} \cap \lc \beta=b, x=\tilde{x} \rc}
\Phi_t^\ast \lb d \mathrm{vol}_{q, K} \rb
=
\sum_{J \subset A \setminus \lc w, q \rc, J \supset K} (-1)^{J \setminus K}
\int_{D_{q, K, J}(b, \tilde{x})}
\lb \Phi_t \circ s_{b, x}' \rb^\ast \lb d \mathrm{vol}_{q, K} \rb+O \lb t^{\epsilon} \rb,
\end{align}
where $s_{b, x}'$ is the map from the $\gamma$-plane to $N_\bR \times S_{\varepsilon_0}^1$ given by $\gamma \mapsto \lb \alpha_{w, q, K}' (b, x), b, \gamma, x\rb$.
\end{lemma}
\pref{lm:E} can be proved by the same argument as the one given in \cite[Section 5.4, Section 3.4]{MR4194298}.
We use \pref{lm:ac}.

\begin{lemma}\label{lm:Pba}
On the image of $D_{q, K, J}(b, x)$ by the map $s_{b, x}'$, we have
\begin{align}
\Phi_t^\ast \alpha&\equiv \alpha_{w, q, K}' (b, x)+ \frac{\sqrt{-1}}{\log t} \cdot \lb \Arg \lb 1+x \rb- \arg \lb -\frac{c_q}{c_w} \rb \rb\\
\Phi_t^\ast \beta_k&\equiv b_k+ \frac{\sqrt{-1}}{\log t} \cdot \lb \arg \lb -\frac{c_q}{c_w} \rb- \arg \lb -\frac{c_j}{c_w} \rb \rb \quad (k \in K)\\
\Phi_t^\ast \beta_j'&=\beta_j'+ \frac{\sqrt{-1}}{\log t} \cdot \lb \arg \lb -\frac{c_q}{c_w} \rb- \arg \lb -\frac{c_j}{c_w} \rb \rb \quad (j \in J \setminus K),
\end{align}
when $t>0$ is sufficiently small.
Here $\beta_j':=\mu_j-\mu_{q}$ for $j \in J \setminus K$.
\end{lemma}
\begin{proof}
Let $(n, x)$ be an element of the image of $D_{q, K, J}(b, x)$ by the map $s_{b, x}'$.
By \eqref{eq:aa'} and \pref{lm:a-range}, we have
\begin{align}
\lb \mu_q-\mu_{w} \rb(n)
=\alpha_{w, q, K}'(b, x) 
=\alpha_{w, q, K} \lb b, \gamma, x \rb+O \lb t^\epsilon \rb 
\leq \epsilon +O \lb t^\epsilon \rb \leq \kappa
\end{align}
for sufficiently small $t>0$.
For any element $j \in J$, we also have 
\begin{align}
\lb \mu_j-\mu_{w} \rb(n)
=
\lb \mu_j-\mu_{q} \rb(n)+\lb \mu_q-\mu_{w} \rb(n)
=
b_j
+
\lb \mu_q-\mu_{w} \rb(n)
\leq 2 \epsilon +O \lb t^\epsilon \rb \leq \kappa
\end{align}
for sufficiently small $t>0$.
Thus we have $\lc q, j \rc \subset K_\kappa^n$.
By using \eqref{eq:phase-shift}, we obtain the claim.
\end{proof}

\begin{lemma}{\rm(cf.~\cite[Lemma 3.2]{MR4194298})}\label{lm:DH}
The polytope $E_{q, J} \lb \lc b_j \rc_{j \in J}, x \rb$ is non-empty if and only if the convex hull of $\lc w, q \rc \sqcup J$ is contained in $\scrT$, and its affine volume is
\begin{align}\label{eq:vol}
\vol \lb E_{q, J} \lb \lc b_j \rc_{j \in J}, x \rb \rb=\int_{Y_{w}}
\exp \lb \omega_{\lambda}^{w} 
-\sum_{j \in J} b_j D_j^w - \alpha_{w, q, K}'(b, x) \sigma^w \rb
\cdot 
D_q^w \cdot \prod_{j \in J} D_j^w.
\end{align}
\end{lemma}

The above lemma can also be shown in a similar manner to \cite[Lemma 3.2]{MR4194298}.
We use the Duistermaat--Heckman theorem (cf.~\cite[Theorem 2.10]{MR1301331}) for obtaining \eqref{eq:vol}.

It turns out by \pref{lm:DH} that in \eqref{eq:inc-exc}, it suffices to take the sum only over $J$ such that the convex hull of $\lc w, q \rc \sqcup J$ is contained in $\scrT$.
By \pref{lm:Pba}, we also have $\Phi_t^\ast \lb d\beta_j' \rb=d\beta_j'$ for $j \in J \setminus K$.
When we identify the polytope $E_{q, J} \lb \lc b_j \rc_{j \in J}, x \rb$ with its image by the map $s_{b, x}'$, \eqref{eq:inc-exc} can be written as
\begin{align}\label{eq:inc-exc2}
\begin{split}
&\int_{S_t^{w, q, K} \cap \lc \beta=b, x=\tilde{x} \rc}
\Phi_t^\ast \lb d \mathrm{vol}_{q, K} \rb \\
&=
\sum_{J \subset A \setminus \lc w, q \rc, J \supset K} (-1)^{J \setminus K}
\int_{\ld 0, \epsilon \rd^{|J \setminus K|}} \bigwedge_{j \in J \setminus K} d b_j'
\int_{E_{q, J} \lb b, b', x \rb}
\Phi_t^\ast \lb d \mathrm{vol}_{q, J} \rb + O \lb t^{\epsilon} \rb,
\end{split}
\end{align}
where $b':=\lc b_j' \relmid j \in J \setminus K \rc$ is the coordinate system on $\ld 0, \epsilon \rd^{|J \setminus K|}$, and $\mathrm{vol}_{q, J}$ is the holomorphic form such that $\mathrm{vol}_{q, K}=\bigwedge_{j \in J \setminus K} d\beta_j' \wedge \mathrm{vol}_{q, J}$.

In order to compute the integral 
$\int_{E_{q, J} \lb b, b', x \rb} \Phi_t^\ast \lb d \mathrm{vol}_{q, J} \rb$ in \eqref{eq:inc-exc2}, 
we consider the \emph{complex volumes} of polytopes which were introduced in \cite[Section 5.3]{MR4194298}.
By \pref{lm:Pba}, we can see that the image of $E_{q, J}\lb b, b', x \rb$ by the map $\Phi_t$ is contained in the complex affine subspace of $N_\bC$ defined by
\begin{align}\label{eq:Pbeta}
\begin{split}
\mu_q-\mu_{w}&=\alpha_{w, q, K}' (b, x)+ \frac{\sqrt{-1}}{\log t} \cdot \lb \Arg \lb 1+x \rb- \arg \lb -\frac{c_q}{c_w} \rb \rb \\
\mu_j-\mu_{q}&=b_j+ \frac{\sqrt{-1}}{\log t} \cdot \lb \arg \lb -\frac{c_q}{c_w} \rb- \arg \lb -\frac{c_j}{c_w} \rb \rb, \quad j \in J,
\end{split}
\end{align}
which we let $S_{q, J}\lb b, b', x \rb$ denote.
Each facet of $E_{q, J}\lb b, b', x \rb$ is given by $E_{q, J}\lb b, b', x \rb \cap \lc \mu_{m_0}-\mu_q =0 \rc$ for some $m_0 \in A_w \setminus \lb \lc q \rc \sqcup J \rb$.
For any point $n$ in the facet, we have $m_0 \in K_\kappa^n$. 
We can see from \eqref{eq:Phi} that the facet is mapped by the map $\Phi_t$ to the complex affine hyperplane in $S_{q, J}\lb b, b', x \rb$ given by 
\begin{align}\label{eq:Pbeta2}
\mu_{m_0}-\mu_{q}&=\frac{\sqrt{-1}}{\log t} \cdot \lb \arg \lb -\frac{c_q}{c_w} \rb- \arg \lb -\frac{c_{m_0}}{c_w} \rb \rb.
\end{align}
The complex volume of the pair $\lb E_{q, J} \lb b, b', x \rb, \Phi_t \rb$ with respect to the form $d \mathrm{vol}_{q, J}$ in the sense of \cite[Section 5.3]{MR4194298} is defined to be $\int_{E_{q, J} \lb b, b', x \rb} \Phi_t^\ast \lb d \mathrm{vol}_{q, J} \rb$.
By \cite[Lemma 5.8]{MR4194298}, we can see that it is a polynomial function of the constant terms of affine equations \eqref{eq:Pbeta} and \eqref{eq:Pbeta2}.
By analytic continuation of \eqref{eq:vol}, we obtain
\begin{align}\label{eq:c-vol} 
\int_{E_{q, J} \lb b, b', x \rb} \Phi_t^\ast \lb d \mathrm{vol}_{q, J} \rb
=
\int_{Y_{w}}
E^{q, K, J}_t \lb b, b', x, D \rb \cdot 
D_q^w \cdot \prod_{j \in J} D_j^w,
\end{align}
where
\begin{align}
E^{q, K, J}_t \lb b, b', x, D \rb
:=
\exp \lb \omega_{\lambda}^{w} 
+ \frac{\sqrt{-1}}{\log t}
\sum_{m \in A_{w}} 
\lb \arg \lb -\frac{c_m}{c_w}\rb -\Arg(1+x)\rb D_m^w
-\sum_{j \in J} b_j D_j^w - \alpha_{w, q, K}'(b, x) \sigma^w \rb.
\end{align}
This is obtained by substituting 
$\lambda_m+\frac{\sqrt{-1}}{\log t} \arg \lb -\frac{c_m}{c_w} \rb$ to $\lambda_m$ for all $m \in A_w$, 
and $\lambda_w+\frac{\sqrt{-1}}{\log t} \Arg(1+x)$ to $\lambda_w$ in \eqref{eq:vol}.

\subsection{The sum of local integrals}\label{sc:sum}

\begin{lemma}\label{lm:1}
One has
\begin{align}\label{eq:a}
\int_{S_t^{w}} \widetilde{\Phi}_t^\ast \pi_1^\ast i_t^\ast \omega_t^{l, v}
=\lb 1+ O \lb t^\epsilon \rb \rb (-1)^{p_w} \cdot
\lb
\int_{Y_{w}} 
\int_{S_{\varepsilon_0}^1} 
\frac{(1+x)^{l-p_w-1}}{x^l}
P_t(x, D)_{d+1}
dx+
O \lb t^\epsilon \rb
\rb,
\end{align}
where $P_t(x, D)_{d+1}$ denotes the part of $P_t(x, D)$ in degree $2(d+1)$, which is defined by
\begin{align}\label{eq:PtxD}
P_t(x, D):=\lb \log t \rb^d
\sum_{\substack{q \in A_w\\ K \subset J \subset A_w \setminus \lc q \rc}}
\lb -1 \rb^{| J \setminus K |} 
\int_{\ld 0, \epsilon \rd^{|J|}} 
\phi_{q, K}^{v, w} \lb b \rb \cdot
E^{q, K, J}_t \lb b, b', x, D \rb
\cdot 
D_q^w \prod_{j \in J} D_j^w dbdb',
\end{align}
where the sum is taken over all pairs $(q, K, J)$ such that the convex hull of $\lc q \rc \sqcup J$ is in $\scrT$.
\end{lemma}
\begin{proof}
The formula is obtained by taking the sum of the local integrals $\int_{S_t^{w, q, K}} \widetilde{\Phi}_t^\ast \pi^\ast_1 i_t^\ast \omega_t^{l, v}$ computed as \eqref{eq:int3} and using \eqref{eq:inc-exc2} and \eqref{eq:c-vol}.
\end{proof}

We set
\begin{align}
F_t(x, D)&:=
t^{-\omega_{\lambda}^{w}} \cdot
\exp \lb -\sqrt{-1}
\sum_{m \in A_{w}} 
\lb \arg \lb -\frac{c_m}{c_w}\rb -\Arg(1+x)\rb D_m^w \rb,\\
G^{q, K, J}_t(x, D)&:=\int_{\ld 0, -\epsilon \log t \rd^{|J|}} 
\varphi_{q, K}^{v, w} \lb s \rb \cdot \exp \lb -\sum_{j \in J} s_j D_j^w - \sigma^w \cdot 
\log \lb 
\frac{\frac{r_q}{r_{w}} + \sum_{k \in K} \frac{r_k}{r_{w}} e^{-s_k}}{|1+x|} 
\rb \rb ds
\cdot D_q^w \prod_{j \in J} D_j^w,
\end{align}
where $s_j$ $(j \in J)$ is the coordinate on the interval $\ld 0, -\epsilon \log t \rd$, and $\varphi_{q, K}^{v, w}$ is the function on $\lb \bR_{\geq 0} \rb^{|K|}$ defined by 
\begin{align}
\varphi_{q, K}^{v, w} \lb s \rb:=
\left\{ \begin{array}{ll}
\frac{r_q^{p_q} \prod_{k \in \tau_v \cap K} \lb r_k e^{-s_k} \rb^{p_k}}{\lb r_q+\sum_{k \in K} r_k e^{-s_k} \rb^{l-p_w} } & \mathrm{all\ vertices\ of\ } \tau_v\ \mathrm{are\ contained\ in\ } \lc w, q \rc \sqcup K \\
0 & \mathrm{otherwise}.
\end{array} 
\right.
 \end{align}
We also define
\begin{align} \label{eq:QtxD}
Q_t(x, D):=
(-1)^d
\sum_{\substack{q \in A_w\\ K \subset J \subset A_w \setminus \lc q \rc}}
\lb -1 \rb^{| J \setminus K |} 
F_t(x, D) \cdot G^{q, K, J}_t(x, D).
\end{align}
Although we considered the sum only over pairs $(q, K, J)$ such that the convex hull of $\lc q \rc \sqcup J$ is contained in $\scrT$ in \eqref{eq:PtxD}, we do not impose this restriction for the sum in \eqref{eq:QtxD}.

\begin{lemma}\label{lm:2}
One has $P_t(x, D)_{d+1}=Q_t(x, D)_{d+1}$.
\end{lemma}
\begin{proof}
If we multiply each of $\omega_{\lambda}^{w}$ and $D_m^w$ in $P_t(x, D)$ of \eqref{eq:PtxD} by $- \log t$, then the degree $2(d+1)$-part of $P_t(x, D)$ is multiplied by $\lb - \log t \rb^{d+1}$, and the part $E^{q, K, J}_t \lb b, b', x, D \rb \cdot D_q^w \prod_{j \in J} D_j^w$ changes to
\begin{align}
F_t(x, D) \cdot \exp \lb -\sum_{j \in J} (-\log t) b_j D_j^w 
- \sigma^w \log \lb \frac{\frac{r_q}{r_{w}} + \sum_{k \in K} \frac{r_k}{r_{w}} t^{b_k}}{|1+x|} \rb 
\rb
\cdot 
\lb - \log t \rb^{|J|+1} \cdot D_q^w \prod_{j \in J} D_j^w.
\end{align}
Here we used \eqref{eq:alpha'}.
By this and changing the coordinates from $b_j$ to $s_j:=- \log t \cdot b_j$ $(j \in J)$, we obtain $P_t(x, D)_{d+1}=Q_t(x, D)_{d+1}$.
Notice that the summand for $(q, K, J)$ in  \eqref{eq:QtxD} vanishes if the convex hull of $\lc q \rc \sqcup J$ is not contained in $\scrT$ due to $D_q^w \prod_{j \in J} D_j^w$ in $G^{q, K, J}(x, D)$.
\end{proof}

By \pref{lm:1} and \pref{lm:2}, we obtain the following:

\begin{proposition}
One has
\begin{align}\label{eq:b}
\int_{S_t^{w}} \widetilde{\Phi}_t^\ast \pi_1^\ast i_t^\ast \omega_t^{l, v}
=\lb 1+ O \lb t^\epsilon \rb \rb (-1)^{p_w} \cdot
\lb
\int_{Y_{w}} 
\int_{S_{\varepsilon_0}^1} 
\frac{(1+x)^{l-p_w-1}}{x^l}
Q_t(x, D)_{d+1}
dx+
O \lb t^\epsilon \rb
\rb.
\end{align}
\end{proposition}

In order to compute the right hand side of \eqref{eq:b}, we will calculate the integral
\begin{align}
\int_{S_{\varepsilon_0}^1} 
\frac{(1+x)^{l-p_w-1}}{x^l}
Q_t(x, D)
dx.
\end{align}
This is equal to
\begin{align}\label{eq:Q-int}
(-1)^d t^{-\omega_{\lambda}^{w}}
\exp \lb -\sqrt{-1} \sum_{m \in A_{w}} 
\arg \lb -\frac{c_m}{c_w}\rb D_m^w \rb \\
\cdot
\int_{S_{\varepsilon_0}^1}
\frac{(1+x)^{l-p_w-1}}{x^l}
\exp \lb \sqrt{-1} \sum_{m \in A_{w}} \Arg(1+x) D_m^w \rb
\exp \lb \sigma^w \log |1+x| \rb
dx \\
\cdot
\sum_{\substack{q \in A_w\\ K \subset J \subset A_w \setminus \lc q \rc}}
\lb -1 \rb^{| J \setminus K |} 
\int_{\ld 0, -\epsilon \log t \rd^{|J|}} 
\varphi_{q, K}^{v, w} \lb s \rb
\exp \lb -\sum_{j \in J} s_j D_j^w - \sigma^w
\log \lb 
\frac{r_q}{r_{w}} + \sum_{k \in K} \frac{r_k}{r_{w}} e^{-s_k}
\rb \rb ds
\cdot 
D_q^w \prod_{j \in J} D_j^w.
\end{align}
The second line of \eqref{eq:Q-int} is equal to
\begin{align}
\begin{split}
\int_{S_{\varepsilon_0}^1}
\frac{(1+x)^{l-p_w-1}}{x^l}
\exp \lb \sigma^w \Log (1+x) \rb
dx
&=
\int_{S_{\varepsilon_0}^1} 
\frac{(1+x)^{l-p_w-1}}{x^l} (1+x)^{\sigma^w}
dx \\ \label{eq:x-int2}
&=\lb 2 \pi \sqrt{-1} \rb \cdot
\binom{\sigma^w+l-p_w-1}{l-1},
\end{split}
\end{align}
where $\Log (1+x):=\log \left| 1+x \right| +\sqrt{-1} \Arg(1+x)$ is the principal value of the complex logarithmic function, and the last term is the binomial coefficient, i.e.,
\begin{align}
\binom{\sigma^w+l-p_w-1}{l-1}=
\frac{1}{(l-1)!} 
\prod_{i=0}^{l-2} \lb \sigma^w+l-p_w-1-i \rb.
\end{align}
Next, we compute the third line of \eqref{eq:Q-int}.
When we expand the exponential, this is equal to
\begin{align}\label{eq:ex-ex}
\begin{aligned}
I_\emptyset
+\sum_{\substack{q \in A_w\\ K \subset J \subset A_w \setminus \lc q \rc \\ J \neq \emptyset}}
\sum_{\substack{h \in \bZ_{\geq 0} \\ \vec{m} \in \lb \bZ_{\geq 0} \rb^J}}
\frac{\lb -1 \rb^{|K|}}{h ! \prod_{j \in J} m_j !}
\lc
\int_{\ld 0, -\epsilon \log t \rd^{|J|}}
\lb \prod_{j \in J} s_j^{m_j} \rb
\log^h
\lb 
\frac{r_q}{r_{w}} + \sum_{k \in K} \frac{r_k}{r_{w}} e^{-s_k}
\rb
\varphi_{q, K}^{v, w} \lb s \rb ds
\rc \\ 
\cdot (- \sigma^w)^h \cdot D_q^w \prod_{j \in J} \lb -D_j^w \rb^{m_j+1},
\end{aligned}
\end{align}
where
\begin{align}\label{eq:Ie}
I_\emptyset:=
\left\{ \begin{array}{ll}
\sum_{q \in A_w} \exp \lb - \sigma^w \log \lb \frac{r_q}{r_w} \rb \rb D_q^w & \tau_v=w \\
\exp \lb - \sigma^w \log \lb \frac{r_q}{r_w} \rb \rb D_q^w & \mathrm{either\ } \tau_v=q \mathrm{\ or\ } \tau_v= \conv \lb \lc w, q \rc \rb \mathrm{\ for\ some\ } q \in A_w \\
0 & \mathrm{otherwise.}
\end{array}
\right.
\end{align}
The first term $I_\emptyset$ arises from the case of $J=K=\emptyset$.
By taking the sum concerning $K$ first, we write \eqref{eq:ex-ex} as
\begin{align}\label{eq:I-sum}
I_\emptyset+
\sum_{\substack{\lc q \rc \sqcup J \subset A_w \\ J \neq \emptyset}}
\sum_{\substack{h \in \bZ_{\geq 0} \\ \vec{m} \in \lb \bZ_{\geq 0} \rb^J}}
\frac{I_{h, \vec{m}}^{\epsilon, q}(t)}{h ! \prod_{j \in J} m_j !}
\cdot (- \sigma^w)^h \cdot D_{q}^w \prod_{j \in J} \lb -D_j^w \rb^{m_j+1},
\end{align}
where
\begin{align}
I_{h, \vec{m}}^{\epsilon, q}(t)&:=
\int_{\ld 0, -\epsilon \log t \rd^{|J|}}
\prod_{j \in J} s_j^{m_j}
\sum_{K \subset J} \lb -1 \rb^{|K|}
\log^h \lb \frac{r_q}{r_{w}} + \sum_{k \in K} \frac{r_k}{r_{w}} e^{-s_k} \rb
\varphi_{q, K}^{v, w} \lb s \rb
ds.
\end{align}
We also define
\begin{align}\label{eq:Iq}
I_{h, \vec{m}}^{q}&:=
\int_{\ld 0, \infty \rb^J}
\prod_{j \in J} s_j^{m_j}
\sum_{K \subset J} \lb -1 \rb^{|K|}
\log^h \lb \frac{r_q}{r_{w}} + \sum_{k \in K} \frac{r_k}{r_{w}} e^{-s_k} \rb
\varphi_{q, K}^{v, w} \lb s \rb
ds.
\end{align}

\begin{lemma}{\rm(cf.~\cite[Section 1.4, Lemma 3.4]{MR4194298})}\label{lm:I-int-conv}
The integral \eqref{eq:Iq} converges, and one has 
\begin{align}
I_{h, \vec{m}}^{\epsilon, q}(t)=I_{h, \vec{m}}^{q}+O \lb (-\log t)^{\left| \vec{m} \right|} t^{\epsilon} \rb,
\end{align}
where $\left| \vec{m} \right|:=\sum_{j \in J} m_j$.
\end{lemma}
\begin{proof}
First, we show the former claim.
If we change the coordinates from $s_j$ to $X_j:=e^{-s_j}$ $(j \in J)$, then \eqref{eq:Iq} is equal to
\begin{align}\label{eq:int-01}
\int_{\ld 0, 1 \rd^{|J|}}
g_{h, J}^{q} \lb X \rb \cdot 
\prod_{j \in J} \lb -\log X_j \rb^{m_j}
 \frac{d X_j}{X_j},
\end{align}
where $X= \lc X_j \rc_{j \in J}$, and $g_{h, J}^{q} \lb X \rb$ is defined to be
\begin{align}\label{eq:g-hq}
\sum_{K \subset J} \lb -1 \rb^{|K|}
\log^h \lb \frac{r_q}{r_{w}} + \sum_{k \in K} \frac{r_k}{r_{w}} X_k \rb
\lb {r_q}^{p_q} \prod_{k \in \lb A_w \cap \tau_v \rb \setminus \lc q \rc} \lb r_k X_k \rb^{p_k} \rb
\lb r_q+\sum_{k \in K} r_k X_k \rb^{-l+p_w}.
\end{align}
The first sum in \eqref{eq:g-hq} is taken over all $K \subset J$ such that all vertices of $\tau_v$ are contained in $\lc w, q \rc \sqcup K$.
The function $g_{h, J}^{q} \lb X \rb$ is analytic in a neighborhood of $\ld 0, 1 \rd^{|J|}$, and obviously vanishes along the coordinate hyperplane $\lc X_k=0 \rc$ for $k \in \lb A_w \cap \tau_v \rb \setminus \lc q \rc$.
Furthermore, it vanishes also along the coordinate hyperplane $\lc X_j=0 \rc$ for $j \in J \setminus \tau_v$, since the terms of $K=K_0$ and $K_0 \sqcup \lc j \rc$ cancel each other out for any $K_0 \subset J$ such that $j \nin K_0$ and $\lc w, q \rc \sqcup K_0$ contains all vertices of $\tau_v$.
By these and $\ld \lb A_w \cap \tau_v \rb \setminus \lc q \rc \rd \cup \lb J \setminus \tau_v \rb=J$, we have
\begin{align}\label{eq:b-01}
g_{h,J}^{q} \lb X \rb \leq C \cdot \prod_{j \in J} X_j
\end{align}
on $\ld 0, 1 \rd^{|J|}$, where $C >0$ is some constant.
From this and \eqref{eq:int-01}, we can see that the integral \eqref{eq:Iq} converges.

By using the fact
$
\int_{-\epsilon \log t}^\infty e^{-s} s^{m} ds
=O \lb (-\log t)^m t^{\epsilon} \rb
$
and \eqref{eq:b-01}, we can also obtain
\begin{align}
\left| I_{h, \vec{m}}^{q}-I_{h, \vec{m}}^{\epsilon, q}(t) \right|
\leq 
C
\int_{\ld -\epsilon \log t, \infty \rb^J}
\lb \prod_{j \in J} e^{-s_j} \rb \cdot
\prod_{j \in J} s_j^{m_j}ds_j
=O \lb (-\log t)^{\left| \vec{m} \right|} t^{\epsilon} \rb.
\end{align}
\end{proof}

We define
\begin{align}\label{eq:hG}
\widehat{G}:=
I_\emptyset+
\sum_{\substack{\lc q \rc \sqcup J \subset A_w \\ J \neq \emptyset}}
\sum_{\substack{h \in \bZ_{\geq 0} \\ \vec{m} \in \lb \bZ_{\geq 0} \rb^J}}
\frac{I_{h, \vec{m}}^{q}}{h ! \prod_{j \in J} m_j !}
(- \sigma^w)^h \cdot D_{q}^w \cdot \prod_{j \in J} \lb -D_j^w \rb^{m_j+1}.
\end{align}
This is the one obtained by replacing $I_{h, \vec{m}}^{\epsilon, q}(t)$ in \eqref{eq:I-sum} with $I_{h, \vec{m}}^{q}$.
From the above computations and \pref{lm:I-int-conv}, we can see that \eqref{eq:int-decomp} is equal to
\begin{align}
\lb 1+ O \lb t^\epsilon \rb \rb
(-1)^{d+p_w} \lc 
\int_{Y_{w}} 
t^{-\omega_{\lambda}^{w}}
\exp \lb -\sqrt{-1} \sum_{m \in A_{w}} 
\arg \lb -\frac{c_m}{c_w}\rb D_m^w \rb
\binom{\sigma^w+l-p_w-1}{l-1}
\widehat{G}+O \lb \lb -\log t \rb^d t^\epsilon \rb \rc.
\end{align}
Then by using \eqref{eq:est-log} and reducing the constant $\epsilon>0$, we obtain the formula \eqref{eq:int-G}.

If $\conv \lb \lc w \rc \cup \tau_v \rb \nin \scrT$, then there is a vertex of $\tau_v$ which is not contained in $A_w \sqcup \lc w \rc$.
This implies $I_\emptyset, I_{h, \vec{m}}^{q}=0$, and $\widehat{G}=0$.
Therefore, in this case, we have $\int_{C_t^{w}} \Omega_t^{l, v} =O \lb t^\epsilon \rb$ which is one of the claims of \pref{th:main1}.

\begin{remark}
In \cite{MR4194298}, they compute the period integral by expressing the holomorphic volume form as 
\begin{align}\label{eq:hol-vol}
\frac{1}{d f_t} \lb \bigwedge_{i=0}^{d} \frac{dz_i}{z_i} \rb
\end{align}
and by integrating it over the cycle $C_t^w$ $(w=0)$ directly.
When we think of the period integral as an integral of $\omega_t^{l, v}$ over the tube $T_t^w$, what they do amounts to integrating $\omega_t^{l, v}$ first with respect to the coordinate $x$ (that goes around the cycle $C_t^w$) in our notation.
On the other hand, in this article, we integrate $\omega_t^{l, v}$ first with respect to the coordinates corresponding to $\vol_{q, J}$ (in \eqref{eq:c-vol}) before integrating with respect to the coordinate $x$ (in \eqref{eq:x-int2}).
Therefore, the orders of integration are different.

Also in our setup, it might be possible to compute the period integral by expressing our form $\Omega_t^{l, v}$ in the same manner as \eqref{eq:hol-vol} and by integrating it over the cycles $C_t^w$.
($df_t$ will be replaced with $d \lb f_t \rb^l$.)
The author tried to do it.
However, the integrand in our setup is more complicated than that of \cite{MR4194298}, and the author did not succeed to do it.
Instead, it was possible to compute the period integral by thinking of it as an integral over the tube and by changing the order of integration.
\end{remark}

\subsection{The Dirichlet integral}\label{sc:re-gamma}

We assume $\conv \lb \lc w \rc \cup \tau_v \rb \in \scrT$, and show \pref{th:main1} for this case by rewriting \eqref{eq:hG} in terms of the gamma function.
For \eqref{eq:hG}, we regard $D=\lc D_j^w \rc_{j \in A_w}$ as positive real numbers and define the function $G \colon \lb \bR_{> 0} \rb^{|A_w|} \to \bR$ by
\begin{align}\label{eq:def-G}
G(D):=\sum_{\lc q \rc \sqcup J \subset A_w} D_q^w
\prod_{j \in J} \lb -D_j^w \rb 
\int_{\ld 0, \infty \rb^J} \exp \lb -\sum_{j \in J} s_j D_j^w \rb
\lc
\sum_{K \subset J} \lb -1 \rb^{K} 
\lb \frac{r_q}{r_{w}} + \sum_{k \in K} \frac{r_k}{r_{w}} e^{-s_k} \rb^{-\sigma^w}
\varphi_{q, K}^{v, w} \lb s \rb
\rc ds^J,
\end{align}
where $\sigma^w:=\sum_{m \in A_w} D_j^w$.
Notice that if the Taylor expansion of the integrand could be exchanged with the integral in \eqref{eq:def-G}, then the result would be the formal power series $\widehat{G}$ of \eqref{eq:hG}.
By the same argument as \cite[Lemma 4.3]{MR4194298}, we can show the following:

\begin{lemma}{\rm(cf.~\cite[Lemma 4.3]{MR4194298})}\label{lm:G-asy}
For a fixed $D \in \lb \bR_{> 0} \rb^{|A_w|}$, we have the asymptotic expansion
\begin{align}
G(yD) \sim \left. \widehat{G} \right|_{D_j \to yD_j} \quad (y \to +0),
\end{align}
where $\left. \widehat{G} \right|_{D_j \to yD_j}$ means the substitution of $yD_j$ to $D_j$ in the formal power series $\widehat{G}$ defined in \eqref{eq:hG}.
\end{lemma}

We also consider the function $I \colon \lb \bR_{> 0} \rb^{|A_w|} \to \bR$ defined by
\begin{align}\label{eq:I-D}
I(D):=
\prod_{j \in A_w} D_j^w 
\prod_{j \in A_{w} } \lb \frac{r_{w}}{r_{j}} \rb^{D_j^w}
\frac{\prod_{j \in A_{w} \setminus \tau_v} \Gamma \lb D_j^w \rb \prod_{j \in A_{w} \cap \tau_v} \Gamma \lb D_j^w+p_j \rb}{\Gamma \lb \sigma^w+l-p_w \rb},
\end{align}
where $\Gamma \lb \bullet \rb$ is the gamma function.
We will show $G(D)=I(D)$ as functions on $\lb \bR_{> 0} \rb^{|A_w|}$ later in \pref{lm:G-I}.
By using the identity $\Gamma(x+m)=\Gamma(x) \cdot \prod_{i=0}^{m-1} (x+i)$ for $x \in \bR_{>0}$ and $m \in \bZ_{\geq 0}$, we can see that \eqref{eq:I-D} is equal to
\begin{align}\label{eq:I-D2}
\prod_{j \in A_{w} } \lb \frac{r_{w}}{r_{j}} \rb^{D_j^w}
\lb \prod_{j \in A_w \cap \tau_v} \prod_{i=0}^{p_j-1} \lb D_j^w +i \rb \rb
\lb \prod_{i=1}^{l-p_w-1} \frac{1}{\sigma^w +i} \rb
\frac{\prod_{j \in A_{w}} \Gamma \lb D_j^w+1 \rb}{\Gamma \lb \sigma^w+1 \rb}.
\end{align}
Let $\widehat{I}$ denote the formal power series in the variables $\lc D_j^w \rc_{j \in A_w}$, which one can get by applying the power series expansion \eqref{eq:gamma-exp} of $\Gamma(1+x)$ to \eqref{eq:I-D2}.
For a fixed $D \in \lb \bR_{> 0} \rb^{|A_w|}$, we have the asymptotic expansion (Taylor expansion)
\begin{align}
I(yD) \sim \left. \widehat{I} \right|_{D_j \to yD_j} \quad (y \to +0),
\end{align}
where $\left. \widehat{I} \right|_{D_j \to yD_j}$ also means the substitution of $yD_j$ to $D_j$ in $\widehat{I}$.
By combining this, \pref{lm:G-asy}, and the equality $G(D)=I(D)$ (\pref{lm:G-I}), we obtain $\widehat{G}=\widehat{I}$ as formal power series in $\lc D_j^w \rc_{j \in A_w}$.
By substituting $\widehat{I}$ to $\widehat{G}$ in \eqref{eq:int-G}, we obtain \pref{th:main1} in the case of $\conv \lb \lc w \rc \cup \tau_v \rb \in \scrT$.

\begin{lemma}{\rm(cf.~\cite[Section 4.2]{MR4194298})}\label{lm:G-I}
One has $G(D)=I(D)$ as functions on $\lb \bR_{> 0} \rb^{|A_w|}$.
\end{lemma}
\begin{proof}
In \eqref{eq:def-G}, one can interchange the integration and summation concerning $K$ because of the factor $\exp \lb -\sum_{j \in J} s_j D_j^w \rb$.
By doing it and integrating with respect to $s_j$ $\lb j \in J \setminus K \rb$, it turns out that \eqref{eq:def-G} is equal to
\begin{align}\label{eq:G2}
\begin{aligned}
\sum_{\lc q \rc \sqcup J \subset A_w}
\sum_{K \subset J}
\lb -1 \rb^{| J \setminus K |} D_{q}^w
\prod_{k \in K} D_k^w
\int_{\ld 0, \infty \rb^K} e^{-\sum_{k \in K} s_k D_k^w}
\lb \frac{r_q}{r_{w}} + \sum_{k \in K} \frac{r_k}{r_{w}} e^{-s_k} \rb^{-\sigma^w}
\varphi_{q, K}^{v, w} \lb s \rb
ds^K.
\end{aligned}
\end{align}
For a subset $K \subset A_w \setminus \lc q \rc$ such that $\lc w, q \rc \sqcup K$ contains all vertices of $\tau_v$, one has
\begin{align}
\sum_{K \subset J \subset A_w \setminus \lc q \rc } \lb -1 \rb^{| J \setminus K |}
=
\left\{ \begin{array}{ll}
1 & K=A_w \setminus \lc q \rc \\
0 & \mathrm{otherwise},
\end{array} 
\right.
\end{align}
where the sum is taken over $J$.
Hence, \eqref{eq:G2} is equal to
\begin{align}
\begin{aligned}
\prod_{k \in A_w} D_k^w
\sum_{\lc q \rc \sqcup K= A_w}
\int_{\ld 0, \infty \rb^K} e^{-\sum_{k \in K} s_k D_k^w}
\lb \frac{r_q}{r_{w}} + \sum_{k \in K} \frac{r_k}{r_{w}} e^{-s_k} \rb^{-\sigma^w}
\frac{r_q^{p_q} \prod_{k \in \tau_v \cap K} \lb r_k e^{-s_k} \rb^{p_k}}{\lb r_q+\sum_{k \in K} r_k e^{-s_k} \rb^{l-p_w} }ds^K.
\end{aligned}
\end{align}
One can write
\begin{align}\label{eq:G-H}
G(D)=\frac{\prod_{k \in \tau_v \cap A_w} r_k^{p_k}}{\lb r_{w}\rb^{l-p_w}}
\prod_{k \in A_w} D_k^w \cdot H(D)
\end{align}
with
\begin{align}
H(D):=\sum_{\lc q \rc \sqcup K= A_w}
\int_{\ld 0, \infty \rb^K} e^{-\sum_{k \in K \setminus \tau_v} s_k D_k^w} \cdot
e^{-\sum_{k \in \tau_v \cap K} s_k \lb D_k^w+p_k \rb}
\lb \frac{r_q}{r_{w}} + \sum_{k \in K} \frac{r_k}{r_{w}} e^{-s_k} \rb^{-\sigma^w-l+p_w} ds^K.
\end{align}

In order to compute $H(D)$, we consider the \emph{tropical projective space} $\bT P^{| A_{w}|-1}$ defined by
\begin{align}
\bT P^{| A_{w}|-1}:=\left. \lb \lb \bR_{\geq 0} \rb^{| A_{w}|} \setminus \lc 0 \rc \rb \middle/ \bR_{>0} \right.,
\end{align}
where $\bR_{>0}$ acts on $\lb \bR_{\geq 0} \rb^{| A_{w}|} \setminus \lc 0 \rc$ diagonally by scalar multiplication.
Let $\lc u_j \rc_{j \in A_w}$ denote its homogeneous coordinates.
We also consider the forms
\begin{align}
\prod_{j \in A_w \setminus \lc q \rc} d \log \frac{u_j}{u_q}=\prod_{j \in A_w \setminus \lc q \rc} d \log t_j
\end{align}
with $q \in A_w$, each of which is defined on the affine chart
\begin{align}\label{eq:aff-chart}
\bT P^{| A_{w}|-1} \setminus \lc u_q=0 \rc  \xrightarrow{\cong}  \lb \bR_{\geq 0} \rb^{| A_{w}|-1}, \quad \lc u_j \rc_{j \in A_w} \mapsto \lc t_j:=\frac{u_j}{u_q} \rc_{j \in A_w \setminus \lc q \rc}.
\end{align}
They agree with each other on the overlap, and define a volume form on $\bT P^{| A_{w}|-1}$, which will be denoted by $d \vol$ (cf.~\cite[Section 4.2]{MR4194298}).
One can show
\begin{align}\label{eq:int-pr}
H(D)
=
\int_{\bT P^{| A_{w}|-1}} 
\frac{\lb \prod_{j \in A_{w}} u_j^{D_j^w} \rb \cdot \lb \prod_{j \in A_{w} \cap \tau_v} u_j^{p_j} \rb}{\lb \sum_{j \in A_{{w}}} \frac{r_j}{r_{w}} u_j \rb^{\sigma^w+l-p_w}}
d \vol
\end{align}
in the same way as \cite[Lemma 4.2]{MR4194298}.
We rewrite the right hand side of \eqref{eq:int-pr} as an integral over the simplex 
\begin{align}
\nabla':=\lc \lc u_j \rc_{j \in A_w} \in \lb \bR_{\geq 0} \rb^{| A_{w}|} \setminus \lc 0 \rc \relmid \sum_{j \in A_{w}} \frac{r_j}{r_{w}} u_j=1\rc,
\end{align}
which is a slice of the diagonal action on $\lb \bR_{\geq 0} \rb^{| A_{w}|} \setminus \lc 0 \rc$.
On the simplex $\nabla'$, we have
\begin{align}
\frac{du_q}{u_q} =\frac{-\sum_{j \in A_w \setminus \lc q \rc} r_j du_j}{r_w-\sum_{j \in A_w \setminus \lc q \rc} r_j u_j}
\end{align}
and 
\begin{align}
d \vol&=\prod_{j \in A_w \setminus \lc q \rc} \lb \frac{d u_j}{u_j}-\frac{d u_q}{u_q} \rb
=
\lb 1+\frac{\sum_{j \in A_w \setminus \lc q \rc} r_j u_j}{r_w-\sum_{j \in A_w \setminus \lc q \rc} r_j u_j}\rb
\prod_{j \in A_w \setminus \lc q \rc} \frac{d u_j}{u_j}
=\frac{r_w}{r_q u_q} \prod_{j \in A_w \setminus \lc q \rc} \frac{d u_j}{u_j},
\end{align}
where $q \in A_w$.
Hence, we get
\begin{align}
\int_{\bT P^{| A_{w}|-1}} 
\frac{\lb \prod_{j \in A_{w}} u_j^{D_j^w} \rb \cdot \lb \prod_{j \in A_{w} \cap \tau_v} u_j^{p_j} \rb}{\lb \sum_{j \in A_{{w}}} \frac{r_j}{r_{w}} u_j \rb^{\sigma^w+l-p_w}}
d \mathrm{vol}=
\frac{r_{w}}{r_{q}}
\int_{\nabla'} 
\lb \prod_{j \in A_{w}} u_j^{D_j^w-1} \rb
\lb \prod_{j \in A_{w} \cap \tau_w} u_j^{p_j} \rb
\prod_{j \in A_{w} \setminus \lc q \rc} du_j.
\end{align}
If we change the variables to $v_j:=\frac{r_j}{r_{w}} u_j$, then this is equal to
\begin{align}
\prod_{j \in A_{w} } \lb \frac{r_{w}}{r_{j}} \rb^{D_j^w}
\prod_{j \in A_{w} \cap \tau_v} \lb \frac{r_{w}}{r_{j}} \rb^{p_j}
\int_{\nabla} 
\lb \prod_{j \in A_{w}} v_j^{D_j^w-1} \rb
\lb \prod_{j \in A_{w} \cap \tau_v} v_j^{p_j} \rb
\prod_{j \in A_{w} \setminus \lc q \rc} dv_j,
\end{align}
where $\nabla:=\lc \lc v_j \rc_{j \in A_w} \in \lb \bR_{\geq 0} \rb^{| A_{w}|} \setminus \lc 0 \rc \relmid \sum_{j \in A_{w}} v_j=1\rc$.
By using the Dirichlet integral \cite{Lejeune1839}, we obtain
\begin{align}
H(D)=\prod_{j \in A_{w} } \lb \frac{r_{w}}{r_{j}} \rb^{D_j^w}
\prod_{j \in A_{w} \cap \tau_v} \lb \frac{r_{w}}{r_{j}} \rb^{p_j}
\frac{\prod_{j \in A_{w} \setminus \tau_v} \Gamma \lb D_j^w \rb \prod_{j \in A_{w} \cap \tau_v} \Gamma \lb D_j^w+p_j \rb}{\Gamma \lb \sigma^w+l -p_w \rb}.
\end{align}
By this and \eqref{eq:G-H}, we get \pref{lm:G-I}.
\end{proof}

\subsection{The volume of $S_t^{w}$}

Lastly, we need to show the following lemma that we used to get \eqref{eq:int2}.

\begin{lemma}{\rm(cf.~\cite[Section 5.4]{MR4194298})}\label{lm:vol}
The volume of $S_t^{w}$ with respect to the Riemannian metric induced from the Euclidean metric on the ambient space $N_\bR \times D_\varepsilon$ is bounded as $t \to +0$.
\end{lemma}
\begin{proof}
We define the function $\tilde{f}_t^{w, \varepsilon_0} \colon N_\bR \times S^1_{\varepsilon_0} \to \bR$ by
\begin{align}
\tilde{f}_t^{w, \varepsilon_0} \lb n, \varepsilon_0 e^{\sqrt{-1}\theta} \rb
:=\tilde{f}_t^{w} \lb i_t (n) \rb-\left| 1+\varepsilon_0 e^{\sqrt{-1}\theta} \right|.
\end{align}
Then the volume form on $S_t^{w}$ is given by
\begin{align}
\left| d \tilde{f}_t^{w, \varepsilon_0} \right| \frac{\bigwedge_{i=0}^d dn_i \wedge \varepsilon_0 d \theta}{d \tilde{f}_t^{w, \varepsilon_0} },
\end{align}
where $(n_0, \cdots, n_{d})$ are $\bR$-coordinates on $N_\bR \cong \bR^{d+1}$.
We can easily see
\begin{align}
\frac{\partial \tilde{f}_t^{w, \varepsilon_0}}{\partial n_i}
=\frac{\partial \lb \tilde{f}_t^{w} \circ i_t \rb}{\partial n_i} 
=O \lb -\log t \rb, \quad 
\frac{\partial \tilde{f}_t^{w, \varepsilon_0}}{\partial \theta}=O \lb 1 \rb.
\end{align}
Hence, we have
\begin{align}\label{eq:vol-est}
\int_{S_t^{w}} \left| d \tilde{f}_t^{w, \varepsilon_0} \right| \frac{\bigwedge_{i=0}^d dn_i \wedge \varepsilon_0 d \theta}{d \tilde{f}_t^{w, \varepsilon_0} }
\leq
C (-\log t) \int_{S_t^{w}} \frac{\bigwedge_{i=0}^d dn_i \wedge d \theta}{d \tilde{f}_t^{w, \varepsilon_0} }
=
C (-\log t)
\int_0^{2 \pi} d \theta
\int_{S_t^{w, \theta}} \frac{\bigwedge_{i=0}^d dn_i}{d \lb \tilde{f}_t^{w} \circ i_t \rb}
\end{align}
for some constant $C>0$, where $S_t^{w, \theta}:=\lc n \in N_\bR \relmid \tilde{f}_t^w \lb i_t \lb n \rb \rb=\left| 1+\varepsilon_0 e^{\sqrt{-1}\theta} \right| \rc$.

We fix $\theta \in S^1$ and apply \pref{th:main1} to the polynomial $g^\theta_t$ defined by
\begin{align}
g^\theta_t:=\frac{1}{\left| 1+\varepsilon_0 e^{\sqrt{-1}\theta} \right|} \tilde{f}_t^{w} -1.
\end{align}
If we set $l=1, v=w=0$, then 
\begin{align}
\omega_t^{1, 0}=
-\lb \bigwedge_{i=0}^{d} \frac{dz_i}{z_i} \rb
\frac{1}{g^\theta_t}, \quad
C_t^{0} =i_t \lb S_t^{w, \theta} \rb, \quad
\int_{C_t^{0}} \Omega_t^{1, 0} =O \lb \lb -\log t \rb^{d} \rb.
\end{align}
Hence, we get
\begin{align}
\int_{C_t^{0}} \Omega_t^{1, 0} 
=
-\lb \log t \rb^{d+1}
\int_{S_t^{w, \theta}}  
\frac{\bigwedge_{i=0}^d dn_i}{d \lb g_t^\theta \circ i_t \rb} 
=
O \lb \lb -\log t \rb^{d} \rb.
\end{align}
From this and $d \lb \tilde{f}_t^{w} \circ i_t \rb=\left| 1+\varepsilon_0 e^{\sqrt{-1}\theta} \right| \cdot d \lb g_t^\theta \circ i_t \rb$, we can see that \eqref{eq:vol-est} is bounded by a constant.
We obtained the lemma.

Indeed we used this lemma in order to prove \pref{th:main1}, when we evaluate the effect of the perturbation $\delta$ on the period integral in \eqref{eq:int2}.
However, the perturbation $\delta$ is unnecessary for the above integral $\int_{C_t^{0}} \Omega_t^{1, 0}$, since the cycle $C_t^{0}$ is given as the positive real locus of $\lc g_t^\theta =0 \rc$, and we do not need to consider its transport in \eqref{eq:ch-family} for constructing $C_t^{0}$.
Therefore, there is no circular reasoning.
\end{proof}

\section{Proof of \pref{th:main2}}\label{sc:integral2}

Let $\sigma \in \scrP$ be a cell of dimension $d$ in the tropical hypersurface $X\lb \trop \lb f \rb \rb$ whose dual edge $\tau_\sigma$ in $\scrT$ contains an element in $W$ as its vertex.
Let $w$ denote the vertex of $\tau_\sigma$ in $W$, and $m_0 \in A$ denote the other vertex.
(The element $m_0$ may also be in $W$.)
They satisfy $\trop(f)=\mu_w=\mu_{m_0}$ on $\sigma$.
As before, we define $\alpha:=\mu_{m_0}-\mu_{w}$, and  take a collection of integral linear functions $\lc \gamma_j \rc_{j \in J}$ so that $\lc \alpha, \gamma_j \relmid j \in J \rc$ forms an integral affine coordinate system on $N_\bC$.
We use the corresponding coordinate system on $N_{\bC^\ast}$ given by
\begin{align}
y:=t^{\lambda_{m_0}-\lambda_{w}} z^{m_0-w}, \quad x_j:=z^{\gamma_j}\quad (j \in J).
\end{align}

Take a point $n_0 \in \Int \lb \sigma \rb$ and its small neighborhood $U$ in $N_\bR$ such that $\mu_m(n) -\mu_{m_0}(n) \geq \epsilon$ for any $n \in U$ and $m \in A \setminus \lc w, m_0 \rc$.
We set
\begin{align}
\tilde{Z}_t:= \lc (z, x) \in N_{\bC^\ast} \times D_\varepsilon \relmid f_t^w(z)=1+x \rc,
\end{align}
and consider the map $\Log_t \colon N_{\bC^\ast} \to N_\bR$ of \eqref{eq:Log}.
On $\tilde{Z}_t \cap \lb \Log_t^{-1} (U) \times D_\varepsilon \rb$, we have 
\begin{align}\label{eq:sigma-eq}
1+x=-\frac{c_{m_0}}{c_{w}}y \lb 1+h_t \lb z \rb \rb,
\end{align}
where $h_t(z) =O \lb t^{\epsilon} \rb$.
By the implicit function theorem, we can see that this equation can be used to write $y$ as a function $y(x_j, x)$ of the variables $x_j, x$.
We consider the map
\begin{align}
i_t \colon \lb S^1\rb^d \times D_\varepsilon \to \tilde{Z}_t \cap \lb \Log_t^{-1} (U) \times D_\varepsilon \rb, \quad (e^{\sqrt{-1}\theta_1}, \cdots, e^{\sqrt{-1}\theta_d}, x) \mapsto 
\lb y(x_j, x), x_j, x \rb,
\end{align}
where $x_j=t^{\gamma_j(n_0)} e^{\sqrt{-1}\theta_j}$.
We set
\begin{align}
S^\sigma_t&:=i_t \lb \lb S^1\rb^d \times S_{\varepsilon_0}^1 \rb \subset 
\lb \Log_t^{-1}(U) \setminus Z_t \rb \times S_{\varepsilon_0}^1 \\
T^\sigma_t&:=i_t \lb \lb S^1\rb^d \times \lc 0 \rc \rb \subset Z_t \times \lc 0 \rc.
\end{align}
$T^\sigma_t$ is the cycle mentioned in \pref{sc:main}, and $S^\sigma_t$ is a tube over it.

Suppose that $\tau_v \in \scrF$ is either $\tau_\sigma$ or one of the vertices $w, m_0$.
One can write $v=p_{w} w +p_{m_0} m_0$ with $p_{w}, p_{m_0} \in \bZ \cap \ld 0, l \rd$ such that $p_{w}+p_{m_0}=l$.
By \eqref{eq:sigma-eq}, one can get
\begin{align}
\omega_t^{l, v}
&= \lb 1+ O \lb t^\epsilon \rb \rb \cdot
\lb -1 \rb^{l} \cdot 
\lb \frac{c_{m_0}}{c_{w}} \rb^{p_{m_0}}
\cdot y^{p_{m_0}}
\frac{1}{\lb f_t^w-1\rb^l}
\frac{dy}{y}
\bigwedge_{j \in J} \frac{dx_j}{x_j} \\
&= \lb 1+ O \lb t^\epsilon \rb \rb 
\cdot \lb -1 \rb^{p_{w}}
\lb 1+x \rb^{p_{m_0}-1} \frac{dx}{x^l}
\bigwedge_{j \in J} \frac{dx_j}{x_j}
\end{align}
on $\tilde{Z}_t \cap \lb \Log_t^{-1} (U) \times D_\varepsilon \rb$.
Hence, we get
\begin{align}
\begin{split}
\int_{T_t^\sigma} \Omega_t^{l ,v}
=\frac{1}{2 \pi \sqrt{-1}} \int_{S_t^\sigma} \omega_t^{l ,v}
&=\lb 2 \pi \sqrt{-1} \rb^d \cdot \lb -1 \rb^{p_w} \cdot \binom{p_{m_0}-1}{l-1} + O \lb t^\epsilon \rb \\ \label{eq:int-pi}
&=
\left\{ \begin{array}{ll}
-\lb 2 \pi \sqrt{-1} \rb^d + O \lb t^\epsilon \rb & \tau_v=w \\
\lb 2 \pi \sqrt{-1} \rb^d + O \lb t^\epsilon \rb & \tau_v=m_0 \\
O \lb t^\epsilon \rb & \mathrm{otherwise.}
  \end{array} 
\right.
\end{split}
\end{align}

Next, suppose that $\tau_v \in \scrF$ is neither $\tau_\sigma$ nor the vertices $w, m_0$.
One can write $v= \sum_{m \in A \cap \tau_v} p_m \cdot m$ with $p_m \in \bZ \cap \lb 0, l \rd$ such that $\sum_{m \in A \cap \tau_v} p_m =l$.
Again by \eqref{eq:sigma-eq}, one can get
\begin{align}
\omega_t^{l, v}
&= \lb 1+ O \lb t^\epsilon \rb \rb \cdot
\prod_{m \in A \cap \tau_v} 
\lb - \frac{c_m}{c_{w}} \rb^{p_m}
\lb \frac{t^{\lambda_m} z^m}{t^{\lambda_w}z^w} \rb^{p_m}
\cdot
\frac{1}{\lb f_t^w-1\rb^l}
\frac{dy}{y}
\bigwedge_{j \in J} \frac{dx_j}{x_j} \\
&= \lb 1+ O \lb t^\epsilon \rb \rb \cdot
\prod_{m \in A \cap \tau_v} 
\lb - \frac{c_m}{c_{w}} \rb^{p_m}
\cdot
\lb \frac{t^{\lambda_m} z^m}{t^{\lambda_w}z^w} \rb^{p_m}
\cdot 
\frac{dx}{x^l \lb 1+x \rb}
\bigwedge_{j \in J} \frac{dx_j}{x_j}
\end{align}
on $\tilde{Z}_t \cap \lb \Log_t^{-1} (U) \times D_\varepsilon \rb$.
On $\Log_t^{-1} (U)$, we have $t^{\lambda_m} z^m /t^{\lambda_w}z^w=O \lb t^\delta \rb$ for some $\delta >0$, for any $m \in A \setminus \lc m_0, w \rc$.
Hence, after replacing $\epsilon$ with a smaller one if necessary, we have
\begin{align}
\omega_t^{l, v}
=
O \lb t^\epsilon \rb
\cdot 
\frac{dx}{x^l \lb 1+x \rb}
\bigwedge_{j \in J} \frac{dx_j}{x_j}.
\end{align}
Therefore, we obtain
\begin{align}
\int_{T_t^\sigma} \Omega_t^{l ,v}
=\frac{1}{2 \pi \sqrt{-1}} \int_{S_t^\sigma} \omega_t^{l ,v}
=O \lb t^\epsilon \rb.
\end{align}

If $m_0 \nin W$, then the second case $\tau_v=m_0$ in \eqref{eq:int-pi} can not happen since $v$ is in the interior of $l \cdot \Delta$.
Therefore, $\int_{T_t^\sigma} \Omega_t^{l ,v}=O \lb t^\epsilon \rb$ unless $\tau_v \in W$ and $\trop(f)=\mu_{\tau_v}$ on $\sigma$.
If $\tau_v \in W$ and $\trop(f)=\mu_{\tau_v}$ on $\sigma$, then $\tau_v$ is a vertex of $\tau_\sigma$, and one can suppose $w=\tau_v$.
In this case, we get $\int_{T_t^\sigma} \Omega_t^{l ,v}=-\lb 2 \pi \sqrt{-1} \rb^d + O \lb t^\epsilon \rb$ by \eqref{eq:int-pi}.
We obtained \pref{th:main2}.

\begin{remark}\label{rm:orientation}
The orientation of $T_t^\sigma$ that we used is the one determined by the above ordered coordinates $\lc \theta_j \rc_{j \in J}$.
In the above computation, if $m_0 \in W$ and we swap $w$ and $m_0$, then the orientation of $T_t^\sigma$ also switches.
The orientation of $B_t^w$ that we used when we compute the integral over $C_t^w$ is the one defined by the interior product of the standard volume form on $N_\bR$ with an incoming normal vector field on it.
When we choose orientations of the cycles in these ways, the intersection number of $T_t^\sigma$ and $C_t^w$ is 
\begin{align}
\frac{\lb \bigwedge_{j \in J} d\theta_j \rb \wedge \lb \bigwedge_{j \in J} d\gamma_j \rb}{\bigwedge_{j \in J} \lb d\theta_j \wedge d \gamma_j \rb}=(-1)^{d(d-1)/2}.
\end{align}
Notice that since $d \lb t^{\gamma_j}\rb=t^{\gamma_j} \log t \cdot d \gamma_j$ and $t^{\gamma_j} \log t<0$, the standard orientation of $Z_t \cap \Log_t^{-1} (U)$ is given by $\bigwedge_{j \in J} \lb d\theta_j \wedge d \gamma_j \rb$.
\end{remark}

\section{Leading terms of periods}\label{sc:lead}

From \pref{th:main1}, we can see that the affine volumes of bounded cells in the tropical hypersurface $X\lb \trop \lb f \rb \rb$ appear in the leading terms of periods $\int_{C_t^{w}}  \Omega_t^{l, v}$.
Suppose $\conv \lb \lc w \rc \cup \tau_v \rb \in \scrT$.
For $m \in A_w$, let $\sigma_m \in \scrP$ be the cell of dimension $d$ in the tropical hypersurface $X\lb \trop \lb f \rb \rb$, which is dual to $\conv \lb \lc w, m \rc \rb \in \scrT$.
We also let $\sigma_v \in \scrP$ denote the cell that is dual to $\conv \lb \lc w \rc \cup \tau_v \rb \in \scrT$.

\begin{corollary}
The leading term of $\int_{C_t^{w}}  \Omega_t^{l, v}$ is given by
\begin{align}\label{eq:lead1}
(-1)^{d+1}
\lb - \log t \rb^d
\cdot 
\sum_{m \in A_{w}} 
\vol \lb \sigma_m \rb
\end{align}
when $\tau_v=w$, and by
\begin{align}\label{eq:lead2}
(-1)^{d}
\lb - \log t \rb^{\dim \sigma_v}
\cdot
\frac{\prod_{j \in A_{w} \cap \tau_v} 
(p_j-1)! }{(l-1)!}
\vol \lb \sigma_v \rb
\end{align}
when $w \nin \tau_v$, where $\vol$ denotes the affine volume of the polytope.
\end{corollary}
\begin{proof}
When $\tau_v=w$, we have $p_w=l$ and $E_{v, w}=\prod_{i=0}^{l-1} \lb \sigma^w -i \rb$.
From \pref{th:main1}, we can see that the leading term of $\int_{C_t^{w}}  \Omega_t^{l, v}$ is given by
\begin{align}
\frac{(-1)^{d+l}}{(l-1)!}
\prod_{i=1}^{l-1} \lb -i \rb \cdot
\int_{Y_{w}} 
\exp \lb (-\log t) \omega_{\lambda}^{w} \rb
\cdot 
\sigma^w
=
(-1)^{d+1}
\sum_{m \in A_{w}}
\int_{Y_{w}} 
\exp \lb \lb - \log t \rb \omega_{\lambda}^{w} \rb
\cdot 
D^w_m.
\end{align}
When $w \nin \tau_v$, we have $p_w=0$ and $E_{v, w}=\prod_{j \in A_w \cap \tau_v} \prod_{i=0}^{p_j-1} \lb D_j^w +i \rb$.
The leading term of $\int_{C_t^{w}}  \Omega_t^{l, v}$ is
\begin{align}
\frac{(-1)^{d}}{(l-1)!}
\lb \prod_{j \in A_w \cap \tau_v} \lb p_j-1\rb! \rb
\int_{Y_{w}} 
\exp \lb (-\log t) \omega_{\lambda}^{w} \rb
\cdot 
\prod_{j \in A_w \cap \tau_v} D_j^w.
\end{align}
The cells $\sigma_m$ and $\sigma_v$ are contained in the polytope $\nabla^w$ of \eqref{eq:nabla} as its faces.
The normal fan of $\nabla^w$ is $\Sigma_w$, and the faces $\sigma_m$ and $\sigma_v$ correspond to the strata $D_m^w$ and $\bigcap_{j \in A_w \cap \tau_v} D_j^w$ in the toric variety $Y_w$ respectively.
By \cite[Theorem 2.10]{MR1301331} again, one can get
\begin{align}\label{eq:vol-sig}
\int_{Y_{w}} 
\exp \lb \omega_{\lambda}^{w} \rb
\cdot 
D^w_m
=\vol \lb \sigma_m \rb, \quad
\int_{Y_{w}} 
\exp \lb \omega_{\lambda}^{w} \rb
\cdot 
\prod_{j \in A_w \cap \tau_v} D_j^w
=
\vol \lb \sigma_v \rb.
\end{align}
By using these, we obtain the claim.
\end{proof}

\begin{remark}
When $d=1, l=1$, the leading terms \eqref{eq:lead1} and \eqref{eq:lead2} are written as $\lb - \log t \rb \cdot l(w, w)$ and $-\lb - \log t \rb \cdot l(v, w)$ respectively, where $l(w, w), l(v, w)$ are the tropical periods of the tropical curve $X\lb \trop \lb f \rb \rb$, which we recalled in \pref{sc:log}.
The fact that the leading terms of the periods of a degenerating family of (plane) curves are given by the tropical periods of the tropical curve obtained by tropicalization was first observed by Iwao \cite{MR2576286}.
Lang \cite{MR4099633} also studied the leading terms of the periods of a degenerating family of curves under the tropical limit in a more general setup.
See Remark 3.7 in loc.cit.
It is also known that the valuation of the $j$-invariant of an elliptic curve over a non-archimedean valuation field coincides with the affine length of the cycle in the tropical elliptic curve obtained by tropicalization \cite{MR2457725, MR2570928}.
\end{remark}

\section{Example}\label{sc:ex}

In order to illustrate \pref{th:main1}, we write down the result of a period integral for \pref{eg:d=2}.
For example, suppose $w=0$, $l=2$, and $v=2e_1$.
Then $\tau_v=\lc e_1\rc, p_{m=e_1}=2, p_{w=0}=0$, and $A_{w=0}=\lc \pm e_1, \pm e_2, \pm e_3\rc$.
The toric variety $Y_{w=0}$ is $\bP^1 \times \bP^1 \times \bP^1$, and
\begin{align}
\omega_{\lambda}^{w=0}=\sum_{m \in A_{0}} D_{m}^{w=0}=\sigma^{w=0}
\end{align}
is the anticanonical divisor on it.
We also have
\begin{align}
E_{v=2e_1, w=0}&=D_{e_1}^{w=0} \cdot \lb D_{e_1}^{w=0} +1 \rb=D_{e_1}^{0}=\lc 0 \rc \times \bP^1 \times \bP^1,\\
\widehat{\Gamma}_{w=0}&=\exp \lb \sum_{k \geq 2} (-1)^k \frac{\zeta(k)}{k} \lb \sum_{m \in A_0} \lb D_m^0 \rb^k-\lb \sigma^0 \rb^k \rb \rb.
\end{align}
According to \pref{th:main1}, we have
\begin{align}
\int_{C_t^{w=0}}  \Omega_t^{l=2, v=2e_1}
&=
\int_{Y_{0}} 
t^{-\sigma^0}
\cdot 
\prod_{m \in A_{0}} 
\lb
-\frac{c_m}{c_{w=0}}
\rb^{-D_m^0}
\cdot 
D_{e_1}^0
\cdot 
\widehat{\Gamma}_0
+O \lb t^\epsilon \rb \\
&=
\int_{Y_{0}} 
\lb 1+ \sigma^0 \lb -\log t \rb + \frac{1}{2} \lb \sigma^0 \rb^2 \lb -\log t \rb^2 \rb\\ &\qquad \qquad 
\cdot \prod_{m \in A_{0}} \lb 1-D_m^0 \log \lb - \frac{c_m}{c_0} \rb \rb \cdot D_{e_1}^0 \cdot
\lb 1-\frac{\zeta(2)}{2} \lb \sigma^0 \rb^2 \rb +O \lb t^\epsilon \rb.
\end{align}
For instance, the top term is 
\begin{align}
\int_{Y_{0}} \frac{1}{2} \lb \sigma^0 \rb^2 \lb -\log t \rb^2 \cdot D_{e_1}^0
=
4 \lb -\log t \rb^2 
=
\vol \lb \sigma_{v} \rb \lb -\log t \rb^2,
\end{align}
where $\sigma_{v} \in \scrP$ is the $2$-cell dual to $\conv \lb \lc w=0 \rc \cup \tau_v=\lc e_1 \rc \rb \in \scrT$ (cf.~\pref{sc:lead}), and the second term is
\begin{align}
\sigma^0 \lb -\log t \rb \cdot \lb -\sum_{m \in A_0} D_m^0 \log \lb - \frac{c_m}{c_0} \rb \rb \cdot D_{e_1}^0
=
-2 \lb -\log t \rb \cdot \sum_{m \in \lc \pm e_2, \pm e_3 \rc} \log \lb - \frac{c_m}{c_0} \rb.
\end{align}

\section{Polarized logarithmic Hodge structure of curves}\label{sc:log2}

We work under the same assumptions and use the same notation as in \pref{sc:log}.
In particular, we assume $d=1$.
We will prove \pref{cr:log}.
Recall that when we constructed the sphere cycle $C_t^w$ of \eqref{eq:Ctw} in \pref{sc:sphere}, we chose branches of $\arg \lb -c_{m}/c_{w} \rb$ for all $m \in A \setminus \lc w \rc$.
In order to prove \pref{cr:log}, we will choose them so that we can easily see the intersection numbers between the sphere cycles $C_t^w$ $(w \in W)$.
This is what we will first do in this section (\pref{lm:arg}, \pref{lm:intersect}).

We fix a basis $\lc e_1, e_2 \rc$ of the lattice $M \cong \bZ^2$.
We set
\begin{align}
\scA&:=\lc (m_0, m_1) \in A \times A \relmid m_0 \neq m_1 \rc\\
\scB&:=\lc (m_0, m_1) \in \scA \relmid \conv \lb \lc m_0, m_1 \rc \rb \in \scrT \rc,
\end{align}
where $A:=\Delta \cap M$.

\begin{lemma}\label{lm:arg}
One can simultaneously choose branches of $\arg \lb -c_{m_1}/c_{m_0} \rb \in \bR$ for all pairs $(m_0, m_1) \in \scA$ so that we have
\begin{align}\label{eq:condi1}
\arg \lb -\frac{c_{m_1}}{c_{m_0}} \rb=-\arg \lb -\frac{c_{m_0}}{c_{m_1}} \rb
\end{align}
for any pair $(m_0, m_1) \in \scB$ and its reversed pair $(m_1, m_0) \in \scB$, and 
\begin{align}\label{eq:condi2}
\arg \lb -\frac{c_{m_1}}{c_{m_0}} \rb-\arg \lb -\frac{c_{m_2}}{c_{m_0}} \rb
=\arg \lb -\frac{c_{m_1}}{c_{m_2}} \rb+\frac{\lb m_1-m_0 \rb \wedge \lb m_2-m_0 \rb}{e_1 \wedge e_2} \cdot \pi
\end{align}
for any (ordered) pair $(m_0, m_1, m_2) \in A \times A \times A$ such that $\conv \lb \lc m_0, m_1, m_2 \rc \rb$ is a $2$-cell in $\scrT$.
\end{lemma}
\begin{proof}
First, note the following:
The former condition \eqref{eq:condi1} implies that if we choose a branch of $\arg \lb -c_{m_1}/c_{m_0} \rb$ for a pair $(m_0, m_1) \in \scB$, then that of $\arg \lb -c_{m_0}/c_{m_1} \rb$ for the reversed pair $(m_1, m_0) \in \scB$ is automatically determined.
The latter condition \eqref{eq:condi2} implies that for a $2$-cell $\conv \lb \lc m_0, m_1, m_2 \rc \rb$ in $\scrT$, if we choose branches of $\arg \lb -c_{m_1}/c_{m_0} \rb$ and $\arg \lb -c_{m_2}/c_{m_0} \rb$, then the branch of $\arg \lb -c_{m_1}/c_{m_2} \rb$ is also automatically determined.
One can easily check that even if we use \eqref{eq:condi2} for the same elements $m_0, m_1, m_2$ but with a different order, the argument $\arg \lb -c_{m_1}/c_{m_2} \rb=-\arg \lb -c_{m_2}/c_{m_1} \rb$ determined by the choices of $\arg \lb -c_{m_1}/c_{m_0} \rb$ and $\arg \lb -c_{m_2}/c_{m_0} \rb$ is the same.

We first choose branches of $\arg \lb -c_{m_1}/c_{m_0} \rb$ for elements $(m_0, m_1)$ in $\scB$.
By adding a real constant whose absolute value is small to the value of the function $\lambda \colon A \to \bQ$ \eqref{eq:lambda} at each $m \in A$ if necessary, one can make a new function $\lambda' \colon A \to \bR$ such that it takes different values at all points in $A$, and still extends to a strictly-convex piecewise affine function on the same triangulation $\scrT$ of $\Delta$ as well as the original function $\lambda$.
For $k \in \bZ_{> 0}$ which is less than or equal to the number of elements in $A$, let $m_k \in A$ denote the unique element at which the function $\lambda' \colon A \to \bR$ takes the $k$-th smallest value, and we set
\begin{align}
A(k):=\lc m_1, \cdots, m_k \rc, \quad \mathring{A}(k):=\lc m \in A(k) \relmid \conv \lb \lc m, m_{k+1} \rc \rb \in \scrT \rc.
\end{align}

\begin{claim}
One has $\mathring{A}(k) \neq \emptyset$.
\end{claim}
\begin{proof}
Suppose $\mathring{A}(k) = \emptyset$.
This implies that the values of the function $\lambda'$ at points in 
\begin{align}\label{eq:amk}
\lc m \in A \setminus \lc m_{k+1} \rc \relmid \conv \lb \lc m, m_{k+1} \rc \rb \in \scrT \rc
\end{align}
are greater than $\lambda' \lb m_{k+1} \rb$.
Take the line segment whose endpoints are $m_k, m_{k+1}$.
It intersects with a $1$-cell in $\scrT$ whose endpoints are contained in \eqref{eq:amk}, since the point $m_{k+1}$ is surrounded by such $1$-cells.
Let $p \in \Delta$ denote the intersection point.
The value at the point $p$ of the extension of the function $\lambda'$ to $\Delta$ is greater than $\lambda' \lb m_{k+1} \rb>\lambda' \lb m_{k} \rb$.
This contradicts the convexity of the extension of the function $\lambda'$ to $\Delta$.
Therefore, we have $\mathring{A}(k) \neq \emptyset$.
\end{proof}

In particular, we have $\mathring{A}(1) \neq \emptyset$, i.e., $\conv \lb \lc m_1, m_{2} \rc \rb \in \scrT$.
We first choose a branch of $\arg \lb -c_{m_2}/c_{m_1} \rb=-\arg \lb -c_{m_1}/c_{m_2} \rb$ arbitrarily.
We continue to choose branches of $\arg$ inductively as follows:
Suppose that we have chosen branches of $\arg \lb -c_{m_j}/c_{m_i} \rb$ for all pairs $\lb m_i, m_j \rb \in \scB$ such that $m_i, m_j \in A(k)$ so that all of them satisfy the conditions \eqref{eq:condi1} and \eqref{eq:condi2}.
There exists at least one element $\overline{m} \in \mathring{A}(k)$ satisfying the following:
\begin{condition}\label{cd:2cell}
The number of $2$-cells in $\scrT$, which contains $\overline{m}, m_{k+1}$, and another element in $A(k)$ is at most one.
\end{condition}
(Otherwise elements in $A$ which are adjacent to $m_{k+1}$ in $\scrT$ are all contained in $A(k)$, and the element $m_{k+1}$ is contained in $\Int \lb \Delta \rb$.
Since the values of $\lambda'$ at points in $A(k)$ are less than $\lambda' \lb m_{k+1} \rb$, we get contradiction with the convexity of $\lambda'$ again.)
We fix such an element $\overline{m} \in \mathring{A}(k)$.
We choose a branch of $\arg \lb -c_{m}/c_{m_{k+1}} \rb=-\arg \lb -c_{m_{k+1}}/c_{m} \rb$ for every $m \in \mathring{A}(k)$ one by one counterclockwise or clockwise, starting from choosing for the element $\overline{m} \in \mathring{A}(k)$.
We choose either counterclockwise or clockwise so that the element in $\mathring{A}(k)$ for which we choose a branch of $\arg$ just after $\overline{m}$ is contained in the $2$-cell of \pref{cd:2cell} if that $2$-cell exists.
If the $2$-cell does not exist, then one may choose either counterclockwise or clockwise arbitrarily.

For $\overline{m} \in \mathring{A}(k)$, one may choose a branch arbitrarily.
Between choosing a branch for $m$ and $m'$ in $ \mathring{A}(k)$, if there is a $2$-cell in $\scrT$, which contains $m, m', m_{k+1}$, then the branch of $\arg \lb -c_{m'}/c_{m_{k+1}} \rb=-\arg \lb -c_{m_{k+1}}/c_{m'} \rb$ is automatically determined, once we choose that of $\arg \lb -c_{m}/c_{m_{k+1}} \rb=-\arg \lb -c_{m_{k+1}}/c_{m} \rb$ by \eqref{eq:condi2}.
If there is not such a $2$-cell, then we may choose a branch of $\arg \lb -c_{m'}/c_{m_{k+1}} \rb=-\arg \lb -c_{m_{k+1}}/c_{m'} \rb$ arbitrarily again.
In this way, we are able to choose branches of $\arg \lb -c_{m}/c_{m_{k+1}} \rb=-\arg \lb -c_{m_{k+1}}/c_{m} \rb$ for all $m \in \mathring{A}(k)$ so that they satisfy \eqref{eq:condi1} and \eqref{eq:condi2}.
What might be worried about is if there is a $2$-cell in $\scrT$ containing $\overline{m}, m_{k+1}$, and the element in $\mathring{A}(k)$ for which we lastly chose a branch of $\arg$, and \eqref{eq:condi2} fails for that $2$-cell.
However, such a $2$-cell does not exist due to \pref{cd:2cell} for $\overline{m}$ and the way to choose counterclockwise/clockwise.

By continuing this process, we can choose branches of $\arg \lb -c_{m_1}/c_{m_0} \rb$ for all elements $(m_0, m_1) \in \scB$ as the claim of the lemma requires.
For elements $(m_0, m_1) \in \scA \setminus \scB$, one may choose branches of $\arg \lb -c_{m_1}/c_{m_0} \rb$ arbitrarily, since no conditions are required for them.
\end{proof}

\begin{lemma}\label{lm:intersect}
We choose branches of $\arg \lb -c_{m_1}/c_{m_0} \rb \in \bR$ for all $\lb m_0, m_1 \rb \in \scA$ as in \pref{lm:arg}.
If we construct the cycles $C_t^w$ $(w \in W)$ with these choices of branches, then the intersection number of the cycles $C_t^{w_0}$ and $C_t^{w_1}$ is zero for any $w_0, w_1 \in W$.
\end{lemma}
\begin{proof}
The set $B_t^w$ of \eqref{eq:Btw} converges to the boundary of the cell in $\scrP$ dual to $\lc w \rc \in \scrT$ as $t \to +0$.
Hence, it is obvious that the cycles $C_t^{w_0}$ and $C_t^{w_1}$ can intersect with each other only when the boundaries of the cells in $\scrP$ that are dual to $\lc w_0 \rc, \lc w_1 \rc \in \scrT$ intersect.
This happens if and only if $\conv \lb \lc w_0, w_1 \rc \rb$ is a $1$-cell in $\scrT$.
We can also see that even when $\conv \lb \lc w_0, w_1 \rc \rb$ is a $1$-cell in $\scrT$, the cycles $C_t^{w_0}$ and $C_t^{w_1}$ can intersect only in the subset $\Log_t^{-1} \lb U \rb \subset N_{\bC^\ast}$, where $U \subset N_\bR$ is a small neighborhood of the edge in $\scrP$ dual to $\conv \lb \lc w_0, w_1 \rc \rb \in \scrT$.

Suppose that $\conv \lb \lc w_0, w_1 \rc \rb$ is a $1$-cell in $\scrT$ in the following.
Let $e \in \scrP$ be the edge dual to $\conv \lb \lc w_0, w_1 \rc \rb \in \scrT$, and $v_0, v_1 \in \scrP$ be its vertices.
Let further $m_0, m_1 \in A \setminus \lc w_0, w_1\rc$ denote the elements such that $\mu_{m_i}=\trop (f)$ at $v_i$ $(i=0, 1)$ respectively.
We take small neighborhoods $U_{v_0}, U_{v_1} \subset N_\bR$ of the points $v_0, v_1$.
When these neighborhoods are sufficiently small, one has 
\begin{align}\label{eq:kappas}
K_\kappa^{n} \lb w_i \rb=\lc w_j, m_k \rc, \forall n \in U_{v_k} \quad \lb (i, j)=(0, 1), (1, 0), \ k=0 ,1 \rb,
\end{align}
where $K_\kappa^{n} \lb w_i \rb$ denotes the set $K_\kappa^{n}$ of \eqref{eq:K} for $w=w_i$.
Let $n \in U_{v_0}$ be a point.
By \eqref{eq:kappas} and \eqref{eq:phase-shift}, we have
\begin{align}\label{eq:phiw0}
\la m_0-w_0, \phi^{w_0}(n, 0) \ra &= - \arg \lb -\frac{c_{m_0}}{c_{w_0}} \rb\\ \label{eq:phiw1}
\la m_0-w_1, \phi^{w_1}(n, 0) \ra &= - \arg \lb -\frac{c_{m_0}}{c_{w_1}} \rb, \quad \la w_0-w_1, \phi^{w_1}(n, 0) \ra = - \arg \lb -\frac{c_{w_0}}{c_{w_1}} \rb,
\end{align}
where $\phi^{w_i} \colon N_\kappa^{w_i} \times D_\varepsilon \to N_\bR$ denotes the function taken in \eqref{eq:phase-shift} for $w=w_i$ $(i=0, 1)$.
By \eqref{eq:phiw1}, \pref{eq:condi2}, and \eqref{eq:condi1}, we get
\begin{align}\label{eq:phiw2}
\begin{split}
\la m_0-w_0, \phi^{w_1}(n, 0) \ra 
&= - \arg \lb -\frac{c_{m_0}}{c_{w_1}} \rb+\arg \lb -\frac{c_{w_0}}{c_{w_1}} \rb\\
&=- \arg \lb -\frac{c_{m_0}}{c_{w_0}} \rb +\frac{\lb w_0-w_1 \rb \wedge \lb m_0-w_1 \rb}{e_1 \wedge e_2} \cdot \pi.
\end{split}
\end{align}
Since this differs from \eqref{eq:phiw0} by $\pm \pi$, we can see that the cycles $C_t^{w_0}$ and $C_t^{w_1}$ do not intersect around $\Log_t^{-1} \lb v_0 \rb$.
By the same argument, we can also see that this is the case also around $\Log_t^{-1} \lb v_1 \rb$.

Take points $n_i \in B_t^{w_i} \cap U_{v_1}$ $(i=0, 1)$.
By \eqref{eq:kappas} and \eqref{eq:phase-shift} again, we have
\begin{align}\label{eq:hiw0}
\la m_1-w_0, \phi^{w_0}(n_0, 0) \ra &= - \arg \lb -\frac{c_{m_1}}{c_{w_0}} \rb, \quad \la w_1-w_0, \phi^{w_0}(n_0, 0) \ra = - \arg \lb -\frac{c_{w_1}}{c_{w_0}} \rb \\ \label{eq:hiw1}
\la m_1-w_1, \phi^{w_1}(n_1, 0) \ra &= - \arg \lb -\frac{c_{m_1}}{c_{w_1}} \rb, \quad \la w_0-w_1, \phi^{w_1}(n_1, 0) \ra = - \arg \lb -\frac{c_{w_0}}{c_{w_1}} \rb.
\end{align}
We can write $m_0-w_0=k_1 \lb w_1-w_0\rb+k_2 \lb m_1-w_0\rb$ with $k_1, k_2 \in \bR$.
By taking $\bullet \wedge \lb w_1-w_0\rb / e_1 \wedge e_2$ for the both sides of this, we obtain 
\begin{align}\label{eq:k12}
k_2=\frac{ \lb m_0-w_0 \rb \wedge \lb w_1-w_0\rb}{ \lb m_1-w_0\rb \wedge \lb w_1-w_0\rb}.
\end{align}
By \eqref{eq:hiw0} and \eqref{eq:hiw1}, we can also get
\begin{align}\label{eq:hiw01}
\la m_0-w_0, \phi^{w_0}(n_0, 0) \ra&=-k_1 \arg \lb -\frac{c_{w_1}}{c_{w_0}} \rb-k_2 \arg \lb -\frac{c_{m_1}}{c_{w_0}} \rb\\ \label{eq:hiw11}
\la m_0-w_0, \phi^{w_1}(n_1, 0) \ra&=\lb k_1+k_2 \rb \arg \lb -\frac{c_{w_0}}{c_{w_1}} \rb-k_2 \arg \lb -\frac{c_{m_1}}{c_{w_1}} \rb.
\end{align}
The difference between \eqref{eq:hiw01} and \eqref{eq:phiw0} and the difference between \eqref{eq:hiw11} and \eqref{eq:phiw2} are 
\begin{align}
-k_1 \arg \lb -\frac{c_{w_1}}{c_{w_0}} \rb-k_2 \arg \lb -\frac{c_{m_1}}{c_{w_0}} \rb+\arg \lb -\frac{c_{m_0}}{c_{w_0}} \rb
\end{align}
and
\begin{align}
\lb k_1+k_2 \rb \arg \lb -\frac{c_{w_0}}{c_{w_1}} \rb-k_2 \arg \lb -\frac{c_{m_1}}{c_{w_1}} \rb+ \arg \lb -\frac{c_{m_0}}{c_{w_0}} \rb -\frac{\lb w_0-w_1 \rb \wedge \lb m_0-w_1 \rb}{e_1 \wedge e_2} \cdot \pi
\end{align}
respectively, and the difference between these two is
\begin{align}
\begin{aligned}
-k_2 \arg \lb -\frac{c_{w_1}}{c_{w_0}} \rb-k_2 \arg \lb -\frac{c_{m_1}}{c_{w_1}} \rb+ k_2 \arg \lb -\frac{c_{m_1}}{c_{w_0}} \rb 
-\frac{\lb w_0-w_1 \rb \wedge \lb m_0-w_1 \rb}{e_1 \wedge e_2} \cdot \pi\\
=k_2 \frac{\lb m_1-w_0\rb \wedge \lb w_1-w_0\rb}{e_1\wedge e_2} \cdot \pi -\frac{\lb w_0-w_1 \rb \wedge \lb m_0-w_1 \rb}{e_1 \wedge e_2} \cdot \pi
=0,
\end{aligned}
\end{align}
where we used \eqref{eq:condi1}, \eqref{eq:condi2}, and \eqref{eq:k12}.
This means that both $C_t^{w_0}$ and $C_t^{w_1}$ wind around the cylinder in the Riemann surface $Z_t$, which corresponds to the edge $e \in \scrP$ by the same argument.
Thus we can conclude that the intersection number of these is zero.
\end{proof}

In the following, we use the cycles $C_t^w$ $(w \in W)$ considered in \pref{lm:intersect}.
Let $\beta_w$ $(w \in W)$ be the cycle class in $H_1 \lb Z_t, \bZ \rb$ represented by $C_t^w$.
The torus cycle $T_t^\sigma$ intersects with the sphere cycle $C_t^w$ if and only if $\trop \lb f \rb=\mu_w$ on $\sigma$.
In that case, they intersect positively (see \pref{rm:orientation}).
Hence, we can see from \pref{th:main2} that for any cycle class $\alpha$ in the subspace of $H_1 \lb Z_t, \bZ \rb$ generated by $\lc T_t^\sigma \rc_\sigma$, one has 
\begin{align}
\int_{\alpha} \Omega_t^{1, v} 
=
-2 \pi \sqrt{-1} \la \alpha, \beta_v \ra+ O \lb t^\epsilon \rb
\end{align}
for $v \in W$.
Therefore, if we take a basis $\lc \alpha_w \rc_{w \in W}$ of the subspace of $H_1 \lb Z_t, \bZ \rb$ generated by $\lc T_t^\sigma \rc_\sigma$ so that we have \eqref{eq:symp}, then the cohomology class represented by $\Omega_t^{1, v}$ $\lb v \in W \rb$ in $H^1 \lb Z_t, \bZ \rb$ can be written as
\begin{align}\label{eq:Omega1}
\ld \Omega_t^{1, v} \rd
=\lb -2 \pi \sqrt{-1} + O \lb t^\epsilon \rb \rb \alpha_v^\ast
+ \sum_{w \in W \setminus \lc v \rc} O \lb t^\epsilon \rb \alpha_w^\ast
+ \sum_{w \in W} \lb \int_{C_t^{w}} \Omega_t^{1, v} \rb \beta_w^\ast.
\end{align}

Consider the exact sequence of sheaves
\begin{align}
0 \to \Omega_{Y_\Sigma}^{2} \to \Omega_{Y_\Sigma}^{2} \lb Z_t \rb \to \Omega^1_{Z_t} \to 0.
\end{align}
Since we have $H^{i} \lb Y_\Sigma, \Omega^{2} \rb=0$ for $i=0, 1$ (cf.~e.g.~\cite[Corollary 12.7]{MR495499}), we can see that the Poincar\'{e} residue map \eqref{eq:Res} with $l=1$ is an isomorphism onto $H^0 \lb Z_t,  \Omega^1 \rb \subset H^1 \lb Z_t,  \bC \rb$.
Therefore, the Hodge structure of the Riemann surface $Z_t$ is given by
\begin{align}
F^0=H^1 \lb Z_t,  \bC \rb, \quad F^1=\bigoplus_{v \in W} \bC \ld \Omega_t^{1, v} \rd, \quad F^2=0.
\end{align}

\begin{proof}[Proof of \pref{cr:log}]
We refer the reader to \cite[Section 2.5.15]{MR2465224} or \cite[Section 5.2]{MR4484542} for how variations of polarized Hodge structure on a punctured disk extend to logarithmic variations of polarized Hodge structure on the whole disk.
What we need to prove are the following:
For the one-parameter family $\lc Z_q \rc_{q \in D_\rho^\ast}$ of complex curves, 
\begin{itemize}
\item the monodromy around $q=0$ for \eqref{eq:stalk} is given by $\exp \lb N \rb=\id+N$, where $N$ is the one defined in \pref{eq:nilp}, and
\item the limit Hodge structure is given by \eqref{eq:filtration}.
\end{itemize}

We first compute the monodromy.
We write the monodromy around $q=0$ for the homology group as $M \colon H_1 \lb Z_t, \bZ \rb \to H_1 \lb Z_t, \bZ \rb$.
We substitute $q=t e^{\sqrt{-1} \theta}$ $(\theta \in \ld 0, 2 \pi \rd)$ to the indeterminate $x$ in $f$, and consider the limit $t \to +0$ of the periods.
By \pref{th:main1} and \pref{th:main2}, we get
\begin{align}\label{eq:period-theta}
\int_{\beta_w}  \Omega_t^{1, v}
&=
\left\{ \begin{array}{ll}
\int_{Y_{w}} 
t^{-\omega_{\lambda}^{w}}
\cdot 
\prod_{m \in A_{w}} 
\lb
-\frac{c_m e^{\sqrt{-1} \lambda_m \theta}}{c_{w} e^{\sqrt{-1} \lambda_w \theta}}
\rb^{-D_m^w}
\cdot 
\sigma^w +O \lb t^\epsilon \rb & v=w  \\
-
\int_{Y_{w}} 
t^{-\omega_{\lambda}^{w}}
\cdot 
\prod_{m \in A_{w}} 
\lb
-\frac{c_m e^{\sqrt{-1} \lambda_m \theta}}{c_{w} e^{\sqrt{-1} \lambda_w \theta}}
\rb^{-D_m^w}
\cdot 
D_v^w +O \lb t^\epsilon \rb & v \in A_w  \\
O \lb t^\epsilon \rb &  \mathrm{otherwise}
\end{array} 
\right.\\ \label{eq:period-theta2}
\int_{\alpha_w} \Omega_t^{1, v}
&=
\left\{ \begin{array}{ll}
-2 \pi \sqrt{-1} + O \lb t^\epsilon \rb 
& v=w \\
O \lb t^\epsilon \rb & v \neq w
  \end{array} 
\right.
\end{align}
for $v, w \in W$.
The leading terms of \eqref{eq:period-theta} and \eqref{eq:period-theta2} do not depend on $\theta$.
We can see that we have $M \lb \alpha_w \rb=\alpha_w$, and can write
\begin{align}
M \lb \beta_w \rb=\beta_w+\sum_{w' \in W} M(w, w') \alpha_{w'}
\end{align}
with $M(w, w') \in \bZ$.
We also have
\begin{align}
\prod_{m \in A_{w}} 
\lb
-\frac{c_m e^{\sqrt{-1} \lambda_m \theta}}{c_{w} e^{\sqrt{-1} \lambda_w \theta}}
\rb^{-D_m^w}
&=
\prod_{m \in A_{w}} 
\lb
-\frac{c_m}{c_w}
\rb^{-D_m^w}
\exp \lb -\sqrt{-1} \theta \lb \lambda_m -\lambda_w \rb D_m^w \rb\\
&=
\exp \lb -\sqrt{-1} \theta \omega_\lambda^w \rb
\prod_{m \in A_{w}} 
\lb
-\frac{c_m}{c_w}
\rb^{-D_m^w}
\end{align}
and
\begin{align}
\int_{M \lb \beta_w \rb} \Omega_t^{1, v}
=
\int_{C_t^w} \Omega_t^{1, v}-2 \pi \sqrt{-1} M \lb w, v \rb + O \lb t^\epsilon \rb.
\end{align}
From these, \eqref{eq:period-theta} with $\theta=2 \pi$, and \eqref{eq:vol-sig}, we can get
\begin{align}
M(w, v)=
\left\{ \begin{array}{ll}
(-1)^{1+\delta_{w, v}} l(w, v) & v \in A_w \sqcup \lc w \rc \\
0 & \mathrm{otherwise}.
\end{array} 
\right.
\end{align}
We computed the monodromy $M$ for the homology group.
By taking the dual of $M$ and subtracting $\id$, we obtain $N$ of \pref{eq:nilp}.

Next we compute the limit Hodge structure.
It is the filtration $\lc F^p_\infty \rc_{0 \leq p \leq 2}$ determined by
\begin{align}
F^1_\infty = \lim_{t \to +0} \exp \lb \frac{1}{2 \pi \sqrt{-1} } (-\log t) N \rb \cdot
\lb \bigoplus_{v \in W} \bC \ld \Omega_t^{1, v} \rd \rb
\end{align}
and $F^0_\infty=H^1 \lb Z_t, \bZ \rb, F^2_\infty=\lc 0 \rc$.
By \eqref{eq:Omega1} and \pref{eq:nilp}, the cohomology class $\exp \lb \frac{1}{2 \pi \sqrt{-1} } (-\log t) N \rb \cdot \ld \Omega_t^{1, v} \rd$ is equal to
\begin{align}
\ld \Omega_t^{1, v} \rd
-(-\log t) 
\lb \sum_{w \in \lc v \rc \sqcup \lb A_v \cap W \rb}
(-1)^{1+\delta_{v, w}}
\cdot
l(v, w)
\beta_{w}^\ast \rb
+\sum_{w \in W} O\lb (-\log t) t^\epsilon \rb \beta_w^\ast.
\end{align}
By \eqref{eq:period-theta} with $\theta=0$, we can see that this is equal to
\begin{align}
-2 \pi \sqrt{-1} \alpha_v^\ast
+\lb \int_{Y_{v}} \prod_{m \in A_{v}} \lb -\frac{c_m}{c_{v}} \rb^{-D_m^v} \sigma^v \rb \beta_v^\ast
-\sum_{w \in A_v \cap W} \lb \int_{Y_{w}} \prod_{m \in A_{w}} \lb -\frac{c_m}{c_{w}} \rb^{-D_m^w} D_v^w \rb \beta_w^\ast
\end{align}
plus $\sum_{w \in W} O \lb t^\epsilon \rb \alpha_w^\ast+\sum_{w \in W} O\lb (-\log t) t^\epsilon \rb \beta_w^\ast$.
By taking the limit $t \to +0$, we obtain
\begin{align}
F^1_\infty = 
\bigoplus_{v \in W} \bC \cdot \lb -2 \pi \sqrt{-1} \alpha_v^\ast + \sum_{w \in W} P(v, w) \beta_{w}^\ast \rb,
\end{align}
where $P(v, w)$ is the one defined in \eqref{eq:Pvw}.
We proved the claim.
\end{proof}

\begin{example}
We consider the periods and the polarized logarithmic Hodge structure of \pref{eg:curve}.
Let $\sigma_0 \in \scrP$ (resp. $\sigma_1 \in \scrP$) be the edge in $V(\trop (f))$, on which $\mu_0$ (resp. $\mu_{e_1}$) and $\mu_{e_2}$ attain the minimum of $\trop(f)$.
We consider the corresponding torus cycles $T_{t}^{\sigma_0}$ and $T_t^{\sigma_1}$.
We also choose the branches of $\arg \lb - c_{m_1}/c_{m_0}\rb$ $(m_0, m_1) \in \scA$ so that
\begin{align}
\arg \lb - c_{m}/c_{0} \rb=0 \quad (m \in A_0), \quad \arg \lb - c_{e_2}/c_{e_1} \rb=\arg \lb - c_{2e_1}/c_{e_1} \rb=\pi, \quad \arg \lb - c_{-e_2}/c_{e_1} \rb=-\pi
\end{align}
to construct the sphere cycles $C_{t}^0, C_t^{e_1}$.
\pref{fg:curve} shows the curve $Z_t$ and the cycles $T_{t}^{\sigma_0}, T_t^{\sigma_1}, C_{t}^0, C_t^{e_1}$.
\begin{figure}[htbp]
\begin{center}
\includegraphics[scale=0.3]{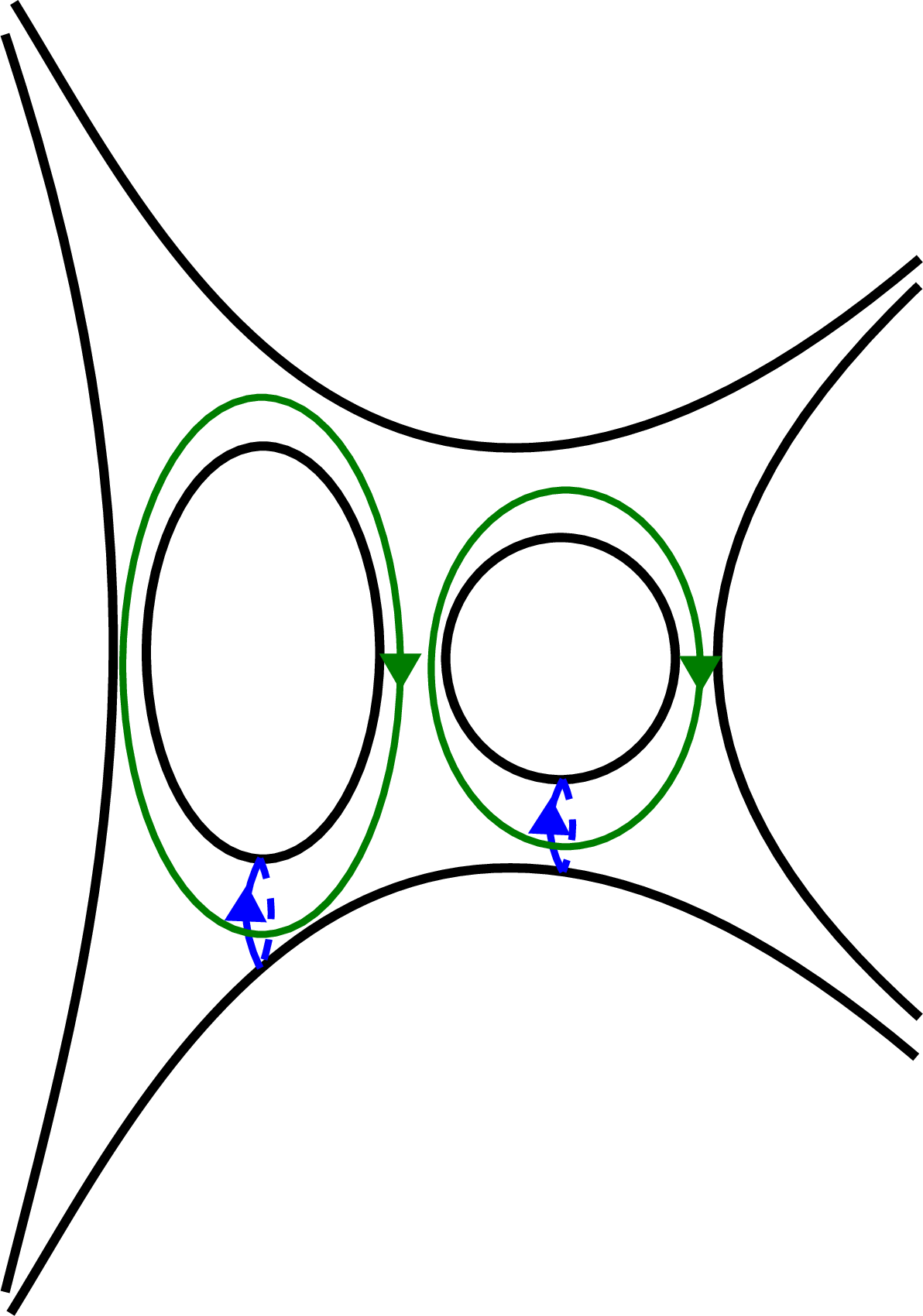}
\end{center}
\caption{The curve $Z_t$ and the cycles $T_{t}^{\sigma_0}, T_t^{\sigma_1}, C_{t}^0, C_t^{e_1}$}
\label{fg:curve}
\end{figure}
The right (resp. left) green cycle is $C_{t}^0$ (resp. $C_t^{e_1}$), and the right (resp. left) blue cycle is $T_{t}^{\sigma_0}$ (resp. $T_t^{\sigma_1}$).
According to \pref{th:main1}, one has
\begin{align}
\int_{C_t^{0}}  \Omega_t^{1, 0} 
&=
8 \cdot (-\log t) +O \lb t^\epsilon \rb, \quad
\int_{C_t^{e_1}}  \Omega_t^{1, e_1} 
=
8 \cdot (-\log t) - 4 \pi \sqrt{-1} + O \lb t^\epsilon \rb
\\
\int_{C_t^{0}}  \Omega_t^{1, e_1} 
&=
\int_{C_t^{e_1}}  \Omega_t^{1, 0} 
=
-2 \cdot (-\log t)
+O \lb t^\epsilon \rb\\
\int_{T_t^{\sigma_0}} \Omega_t^{1, 0} 
&=
\int_{T_t^{\sigma_1}} \Omega_t^{1, e_1} 
=
-2 \pi \sqrt{-1} + O \lb t^\epsilon \rb, \quad
\int_{T_t^{\sigma_0}} \Omega_t^{1, e_1} 
=
\int_{T_t^{\sigma_1}} \Omega_t^{1, 0} 
=
O \lb t^\epsilon \rb
\end{align}
as $t \to +0$, for some constant $\epsilon >0$.

Let $\alpha_0, \alpha_{e_1}, \beta_0, \beta_{e_1} \in H_1 \lb Z_t, \bZ \rb$ be the cycle classes represented by $T_t^{\sigma_0}, T_t^{\sigma_1}, C_t^0$, and $C_t^{e_1}$ respectively.
These classes form a symplectic basis of $H_1 \lb Z_t, \bZ \rb$.
The monodromy of the locally constant sheaf $H_\bZ$ of the polarized logarithmic Hodge structure $\lb H_\bZ, Q, \scrF \rb$ of \pref{cr:log} is given by $\id + N$ with
\begin{align}
N \lb \beta_0^\ast \rb=N \lb \beta_{e_1}^\ast \rb=0, \quad
N \lb \alpha_0^\ast \rb=8 \beta_0^\ast-2 \beta_{e_1}^\ast, \quad
N \lb \alpha_{e_1}^\ast \rb=-2 \beta_{0}^\ast+8 \beta_{e_1}^\ast.
\end{align}
The filtration $\scrF$ is given by \eqref{eq:filtration0} and \eqref{eq:filtration} with
\begin{align}
P(0, 0)=P(e_1, 0)=P(0, e_1)=0, \quad P(e_1, e_1)=-4 \pi \sqrt{-1}.
\end{align}
\end{example}

\par
{\it Acknowledgment: } 
I thank Hiroshi Iritani for his kind answers to my questions on \cite{MR3112512, MR4194298} and many valuable comments on the draft of this article.
I am also grateful to the anonymous referees for their comments which helped me to improve this article. 
This work was supported by the Institute for Basic Science (IBS-R003-D1). \\

\bibliographystyle{amsalpha}
\bibliography{bibs}

\end{document}